\documentclass[11pt,letterpaper]{amsart}

\usepackage{amsmath,amssymb,amsthm,mathtools,mathrsfs}
\usepackage{bm}
\allowdisplaybreaks
\usepackage[all]{xy}
\usepackage{tikz-cd}
\usepackage{hyperref}

\usepackage{here}
\usepackage{time}

\usepackage{xcolor}
\usepackage{aliascnt}
\usepackage[capitalize,nameinlink,noabbrev,nosort]{cleveref}
\hypersetup{
	colorlinks=true,       
	linkcolor=blue,          
	citecolor=blue,        
	filecolor=blue,      
	urlcolor=blue,           
}

\makeatletter
\@namedef{subjclassname@2020}{\textup{2020} Mathematics Subject Classification}
\makeatother

\theoremstyle{remark}
\newtheorem{remark}{Remark}

\theoremstyle{definition}
\newtheorem{theorem}{Theorem}[section]

\newaliascnt{definition}{theorem}
\newtheorem{definition}[definition]{Definition}
\aliascntresetthe{definition}

\newaliascnt{lemma}{theorem}
\newtheorem{lemma}[lemma]{Lemma}
\aliascntresetthe{lemma}

\newaliascnt{fact}{theorem}
\newtheorem{fact}[fact]{Fact}
\aliascntresetthe{fact}

\newaliascnt{proposition}{theorem}
\newtheorem{proposition}[proposition]{Proposition}
\aliascntresetthe{proposition}

\newaliascnt{corollary}{theorem}
\newtheorem{corollary}[corollary]{Corollary}
\aliascntresetthe{corollary}

\newaliascnt{conjecture}{theorem}

\aliascntresetthe{conjecture}

\newaliascnt{question}{theorem}

\aliascntresetthe{question}

\newaliascnt{exercise}{theorem}

\aliascntresetthe{exercise}

  \newtheoremstyle{mystyle}
    {}
    {}
    {\normalfont}
    {}
    {\bfseries}
    {}
    { }
    {}
\theoremstyle{mystyle}
\newtheorem*{proposition*}{Proposition}
\newtheorem*{corollary*}{Corollary}
\newtheorem*{conjecture*}{Conjecture}

\newtheorem*{theorem*}{Theorem}

\numberwithin{equation}{section}

\newcommand{\Id}{\mathrm{Id}}

\newcommand{\ad}{\mathrm{Ad}}

\newcommand{\Z}{\mathbb{Z}}
\newcommand{\Q}{\mathbb{Q}}
\newcommand{\R}{\mathbb{R}}
\newcommand{\C}{\mathbb{C}}

\DeclareMathOperator{\ReNew}{Re}
\renewcommand{\Re}{\ReNew}

\DeclareMathOperator{\GL}{GL}

\def\cInd{\mathrm{c}\text{-}\mathrm{Ind}}%

\def\bs{\backslash}

\DeclareMathOperator{\vol}{vol}
\DeclareMathOperator{\Hom}{Hom}

\newcommand{\fa}{\mathfrak{a}}

\newcommand{\bA}{\mathbb{A}}

\newcommand{\bD}{\mathbb{D}}

\newcommand{\bF}{\mathbb{F}}
\newcommand{\bG}{\mathbb{G}}

\newcommand{\bL}{\mathbb{L}}
\newcommand{\bM}{\mathbb{M}}

\newcommand{\bT}{\mathbb{T}}

\newcommand{\bX}{\mathbb{X}}

\newcommand{\cO}{\mathcal{O}}

\newcommand{\cV}{\mathcal{V}}
\newcommand{\cW}{\mathcal{W}}
\newcommand{\cX}{\mathcal{X}}

\newcommand{\Ext}{\mathrm{Ext}}

\begin{document}

\title[]{The algebraic theory of non-archimedean intertwining periods}

\author[F. Chen]{Fei Chen}
\address{Yau Mathematical Sciences Center,
	Tsinghua University, Beijing, China}
\email{fchen@tsinghua.edu.cn}

\author[N. Matringe]{Nadir Matringe}
\address{Shanghai Institute for Mathematics
	and Interdisciplinary Sciences,
	Block A, International Innovation Plaza,
	No. 657 Songhu Road, Yangpu District,
	Shanghai, China}
\email{nadirmatringe@outlook.fr}
\date{\today}


\subjclass[2020]{}

\begin{abstract}
We develop an algebraic theory of local intertwining periods for symmetric
pairs $(G,H,\theta)$ over non-archimedean local fields $F$ of characteristic
different from $2$, and finite-length representations of Levi subgroups
of $G$ with coefficients in an algebraically closed field $R$ of
characteristic different from the residual characteristic of $F$.
For open orbits this algebraic version is essentially contained in the
work of Blanc and Delorme. The theory relies on a crucial new uniqueness result, extending a result of Blanc and Delorme for open periods, which asserts that an admissible intertwining period is the only $H$-equivariant linear form extending the corresponding integral on the space of well-supported sections. In the group case it recovers at once the results of Dat on algebraic intertwining operators. Our algebraic construction does
not require the unimodularity hypothesis imposed in the work of
Matringe--Offen--Yang. When $R=\C$, we also remove this assumption in the classical non-archimedean analytic setting, and prove convergence of intertwining periods for all symmetric spaces. 
At the same time we extend the context of the whole theory in the following
manner: we allow $F$ to be of positive odd characteristic, we allow twisting characters in the periods, and we allow $H$
to lie between $(\bG^\theta)^\circ(F)$ and $G^\theta$.
Finally, in the algebraic setting, we obtain explicit polynomial
bounds for the singularities of intertwining periods, and compute them in a concrete case. 
\end{abstract}


\maketitle
 \setlength{\leftmargini}{20pt}
\section{Introduction}

Local intertwining periods are equivariant linear forms on parabolically
induced representations, constructed from corresponding forms on the inducing data. Their
dependence on unramified characters, their singularities and their functional
equations provide tools for studying distinction, both locally and globally. We refer to \cite{FLO}, \cite{MatJFA}, \cite{SX24}, \cite{MSSS}, \cite{MOY} for a far from being exhaustive list.

The purpose of this paper is to give a purely algebraic construction of these
periods for all non-archimedean symmetric spaces, to obtain explicit bounds for
their denominators, and to establish their analytic properties without the 
unimodularity assumption which was previously indispensable (see \cite{MO} and \cite{MOYjfa}).

The theory associated with open parabolic orbits in the symmetric space has its origins in the work
of Brylinski--Delorme \cite{BrD} and Carmona--Delorme \cite{CrD} on invariant
distribution vectors and Eisenstein integrals for real reductive symmetric
spaces. In the non-archimedean setting, and for $\theta$-split parabolic subgroups,
Blanc--Delorme \cite{BD} constructed rational families of invariant forms by
smooth group homology, working over
the algebra of regular functions on a torus of unramified characters. Their
argument identifies these forms through their restriction to sections
supported on open orbits and supplies polynomials controlling their
singularities. Thus the algebraic construction and the denominator problem
already occur together in the open case.

Intertwining periods attached to more general admissible parabolic orbits
also arise in the global theory of periods of Eisenstein series; see
Lapid--Rogawski \cite{LR}. Their relations with standard intertwining
operators are governed by the geometry of parabolic orbits and a directed
graph of involutions. The second named author and Offen \cite{MO} developed the corresponding
local theory for $p$-adic Galois symmetric pairs, initially under a boundedness
assumption on the inducing representation. The same authors together with Yang
\cite{MOYjfa} removed this boundedness assumption and established convergence,
functional equations and meromorphic continuation for arbitrary finite-length
inducing representations under the aforementioned unimodularity assumption, to be explained below. They also obtained
general sufficient conditions for distinction of induced representations without any assumption by using iterated periods.
The question of when regularized intertwining periods generate all invariant
forms on an induced representation was subsequently studied by Lu--Matringe
\cite{LuMcomp}.

Our main contributions are an intrinsic characterization of each admissible
intertwining period by a support formula, explicit bounding polynomials
which yield useful pole-order estimates, and the removal of the modulus
assumption in both the algebraic and the non-archimedean analytic theory.
We first describe the setting and the algebraic construction.

Let $F$ be a non-archimedean local field of characteristic different from $2$
and residual characteristic $p$. Let $\bG$ be a connected reductive group
defined over $F$, let $\theta$ be an $F$-rational involution, and put
$G=\bG(F)$. We allow any closed subgroup $H$ satisfying
\[
 (\bG^\theta)^\circ(F)\subseteq H\subseteq G^\theta.
\]
Let $R$ be an algebraically closed field of characteristic different from
$p$, and let $\eta:H\to R^\times$ be a smooth character trivial on unipotent
elements. The latter condition is automatic when $F$ has characteristic
zero. In particular, our algebraic results allow coefficients of positive
characteristic and base fields of positive odd characteristic.

Fix a parabolic subgroup $P=MV$ of $G$ and a finite-length smooth
$R$-representation $\sigma$ of $M$. For $u\in G$, put
\[
 x=x_u=u\theta(u)^{-1},\qquad
 \theta_x=\operatorname{ad}(x)\circ\theta,
\]
and assume that $x$ is $M$-admissible, that is, $\theta_x(M)=M$.
Write $M_x=M\cap uHu^{-1}$ and $\eta^u(m)=\eta(u^{-1}mu)$, and set
\[
 \delta_{P,x}
 =\left.\delta_{P^{\theta_x}}\delta_P^{-1/2}\right|_{M_x}.
\]
The unimodularity, or modulus, assumption referred to above is the condition
$\delta_{P,x}=1$. This unimodularity condition, although satisfied for Galois models (\cite[Lemma 2.5.1]{LR}), is not satisfied for many popular models, such as the symplectic and linear models of $\GL_n$. We impose no such condition.

Let $\cX_x$ be the maximal $\theta_x$-anti-invariant subtorus of the torus
of unramified characters of $M$. Denote its coordinate ring and fraction
field by
\[
 A_x=\mathcal O_R(\cX_x)=R[\Lambda_x],\qquad
 F_x=\operatorname{Frac}(A_x),
\]
where $\Lambda_x$ is its character lattice. The universal character
$\Psi_x:M\to A_x^\times$ defines the universal induced representation
$I_{P,x}(\sigma)=I_P^G(\Psi_x\otimes_R\sigma)$, with normalized induction.
Given an $A_x$-linear $M_x$-equivariant form
\[
 \ell:\Psi_x\otimes_R\sigma
 \longrightarrow A_x\otimes\delta_{P,x}\eta^u,
\]
consider the subspace $\cW(u)$ of $I_{P,x}(\sigma)\otimes_{A_x}F_x$
consisting of the sections $\phi$ for which $h\mapsto\phi(uh)$ has compact
support modulo $H\cap u^{-1}Pu$. Our main algebraic result,
\Cref{thm LIP}, asserts that there is a unique $F_x$-linear
$(H,\eta)$-equivariant form $J_P^G(u,\eta,\ell,\sigma)$ on the whole
induced representation such that
\begin{equation}\label{eq intro support formula}
 J_P^G(u,\eta,\ell,\sigma)(\phi)
 =\int_{(H\cap u^{-1}Pu)\backslash H}
       \eta(h)^{-1}\ell\bigl(\phi(uh)\bigr)\,dh,
 \qquad \phi\in\cW(u).
\end{equation}
On this subspace the integral is algebraic integration of a compactly
supported smooth section. It makes sense over $F_x$ without any convergence
or continuation argument. Moreover, $J_P^G(u,\eta,\ell,\sigma)$ is nonzero
whenever $\ell$ is nonzero. Formula \eqref{eq intro support formula}
therefore gives an algebraic definition of the period, with rationality
built into the construction. These local intertwining periods further satisfy basic properties such as the so-called simple functional equations which are proved in \Cref{ssec FE}. Specializing to the group case, we recover at once the results of Dat on algebraic intertwining operators (\cite{Dat}). 

The proof extends the central-character argument of
Blanc--Delorme from $\theta$-split to all admissible parabolic orbits, and to our more general setting as well. We carry out this argument
using elementary homological algebra in categories of smooth
representations, without the machinery of smooth group homology developed in \cite{BD}. The same argument gives explicit control of denominators before passage
to $F_x$. Namely, we construct nonzero polynomials $d\in A_x$ such that
\[
 d\,J_P^G(u,\eta,\ell,\sigma)\bigl(I_{P,x}(\sigma)\bigr)
 \subseteq A_x.
\] The key in this algebraic construction of intertwining periods is \Cref{lm root separator} below, extending the key lemma of Blanc and Delorme for open periods. The polynomial $d$ bounding the denominator of the periods is given explicitly in the statement of \Cref{thm LIP}. A slightly more optimal presentation, together with other possible choices for $d$ are also given in \Cref{prop explicit bounds} of  \Cref{ssec denominators}.

In \Cref{ssec galois speh} we compute an explicit bounding polynomial $d$ related to the Galois model for inner forms of $\GL_n$. 

As another application, we obtain a sufficient
condition for distinction of an induced representation, and a necessary
and sufficient condition when the inducing representation is cuspidal
(\Cref{thm distinction,cor cuspidal distinction}), extending the corresponding results of \cite{MO} and \cite{MOY} to our more general setting. 

One more interesting consequence of the support characterization of intertwining periods is that such periods are compatible with
reduction modulo a prime $p'$ different from $p$ under the integral hypotheses
of \Cref{ssec reduction}, as we prove in \Cref{prop reduction LIP}. In particular, under suitable hypotheses,  their functional equation with respect to standard intertwining operators reduce mod $p'$, and we expect that this will have applications to the study of modular distinction.

Finally, in \Cref{sec analytic}, when $R=\C$, we prove that the original quotient integral in
\eqref{eq intro support formula}, applied to arbitrary flat sections,
converges absolutely on a nonempty open convex domain in the space of
$\theta_x$-anti-invariant complex parameters; this is the content of \Cref{thm:convergence in general}. It is holomorphic there and extends
meromorphically to the whole parameter space, with continuation given by
the algebraic period, see \Cref{cor final}. This removes the modulus assumption from the
analytic theory of \cite{MOYjfa}, while also allowing the character $\eta$,
the above intermediate subgroups $H$, and base fields of positive odd
characteristic. The main analytic ingredients are as follows. First we obtain an extension of the growth
estimates of Lagier~\cite{Lagierthesis} to equivariant relative matrix coefficients in \Cref{sec Lagier}. Then we prove the existence of algebraic
bounds for the twisting character in \Cref{sec estimate}. Finally we obtain estimates near the boundary of the
open orbit of a generalized version of a function introduced by Brylinski--Delorme, and used in \cite{BD}. Our proof of this estimate simplifies that of \cite{BD}, and is valid in our more general setting, where in particular $F$ is allowed to have odd positive characteristic. These ingredients  give convergence in the $\theta$-split case and then at
maximal vertices; the functional equations yield the general case as in \cite{MOYjfa}.

\subsection*{Acknowledgements} The first-named author is supported by NSFC, Grant No. 12571027, and the second by his start-up grant in the SIMIS, which was also used to invite the first-named author to finalize the paper. The first-named author takes this opportunity to thank the SIMIS for its hospitality on this occasion.

\subsection*{Use of AI}  Successive versions of ChatGPT (GPT-5.5 to GPT-6.1), as
well as Claude Fable~5, were used at several levels in the
preparation of this paper. Besides editorial
assistance, their main contributions are as follows. The first is the extension of the
setting to possibly disconnected $H$ and to $F$ of positive
characteristic (started 
with Claude Fable~5 during a free trial, then continued with
ChatGPT 5.6). The second, and the most
important one, concerns the construction
itself. Our first construction of algebraic
intertwining periods relied on minimal
vertices, hence on open periods, and Dat's algebraic intertwining operators. Their characterization on well-supported
sections was only proved afterwards, using invertibility of intertwining operators and maximal vertices. The
missing ingredient for the present, more natural, although in the end same construction, namely \cref{lm root separator}, was provided
by GPT-6; its present simplified and neater  proof is due
to the authors. Finally, the earlier proof of
\cref{lm open cell estimates} was a lengthy verification that the
arguments of Blanc--Delorme and
Brylinski--Delorme can be adapted to our setting;
GPT-6 came up with the present sharper
argument, based on the determinant function there.

\clearpage
\setcounter{tocdepth}{2}
\tableofcontents
\clearpage

\section{Linear algebraic results}\label{sec linear}

We denote by $R$ a commutative unital ring, by $A$ a commutative unital $R$-algebra, and by $\Gamma$ a locally profinite group. We assume that $\Gamma$ has a compact open subgroup of pro-order invertible in $A$. Fix an $A$-valued Haar measure giving such a subgroup volume $1$, and write $\mathcal H_A[\Gamma]:=C_c^\infty(\Gamma,A)$ for the corresponding Hecke algebra.

We identify smooth $A$-representations of $\Gamma$ with nondegenerate $\mathcal H_A[\Gamma]$-modules, as in \cite[I.4.4]{V}, and compute $\Hom$ and $\Ext$ in this category. This category is abelian and has enough projectives and injectives by
\cite[I.5.10(iv) and I.5.9(c)]{V}, respectively.

\begin{lemma}\label{lm classique}
Let $V$ and $W$ be smooth $A$-representations of $\Gamma$. Assume that $V$ has a $\Gamma$-stable $A$-module filtration $\{0\}=V_0\subseteq V_1\subseteq \dots \subseteq V_n=V$ and $W$ has a $\Gamma$-stable $A$-module filtration $\{0\}=W_0\subseteq W_1\subseteq \dots \subseteq W_m=W$. Assume moreover that there exist $z_{i,j} \in Z(\Gamma)$ for $i=1,\dots,n$ and $j=1,\dots,m$, and $\lambda_1,\dots,\lambda_n,\mu_1,\dots,\mu_m$ in $A$, such that:
\begin{enumerate}
	\item $z_{i,j}\cdot \bar{v}=\lambda_i\cdot \bar{v}$ for any $i=1,\dots,n$, $\bar{v}\in \frac{V_i}{V_{i-1}}$,
	\item $z_{i,j}\cdot \bar{w}=\mu_j\cdot \bar{w}$ for any $j=1,\dots,m$, $\bar{w}\in \frac{W_j}{W_{j-1}}$.
\end{enumerate}
Then \[\prod_{i=1}^n\prod_{j=1}^m(\mu_j -\lambda_i)\Ext_{\mathcal H_A[\Gamma]}^k(V,W)=\{0\}\] for any $k\geq 0$ (where $\Ext_{\mathcal H_A[\Gamma]}^0(V,W):=\Hom_{\mathcal H_A[\Gamma]}(V,W)$). In particular if each $\mu_j-\lambda_i$ belongs to $A^\times$, then $\Ext_{\mathcal H_A[\Gamma]}^k(V,W)=\{0\}$ for any $k\geq 0$.
\end{lemma}
\begin{proof}
Assume first that $n=m=1$. The action of $z=z_{1,1}$ belongs to the center of the category of smooth $A$-representations of $\Gamma$. Therefore the endomorphisms of $\Ext_{\mathcal H_A[\Gamma]}^i(V,W)$ induced by $z$ through $V$ and $W$ coincide. The equality $(\mu_1 -\lambda_1)\Ext_{\mathcal H_A[\Gamma]}^k(V,W)=\{0\}$ follows from this observation. The general case is by induction. We keep on assuming that $m=1$. In this case it is enough to treat the $n=2$ case, which is a consequence of the long exact sequence 
\[0\to \Hom_{\mathcal H_A[\Gamma]}(\frac{V}{V_1},W) \to \Hom_{\mathcal H_A[\Gamma]}(V,W) \to \Hom_{\mathcal H_A[\Gamma]}(V_1,W)\to \dots \]
\[ \to \Ext_{\mathcal H_A[\Gamma]}^k(\frac{V}{V_1},W)\to \Ext_{\mathcal H_A[\Gamma]}^k(V,W) \to \Ext_{\mathcal H_A[\Gamma]}^k(V_1,W)\to \dots.\] To pass from $m=1$ to arbitrary $m$, we argue similarly by induction on
$m$, using the long exact sequence associated with
$0\to W_{m-1}\to W\to W/W_{m-1}\to 0$.
\end{proof}

\begin{remark}
For our later applications, the case where $m=1$ of the above lemma is enough.
\end{remark}

The following elementary, but perhaps less familiar, statement is crucial to our rationality argument.

\begin{proposition}\label{prop habile}
Let $V$, $W$ and $W'$ be smooth $A$-representations of $\Gamma$.  Assume that $W'\subseteq W$ and that
$V$ is equipped with a $\Gamma$-stable $A$-module filtration
\[
 \{0\}\subseteq V_1\subseteq\dots\subseteq V_n=V.
\]
Assume that, for each $i=2,\dots,n$, there exists $x_i\in A$ such that
\[
 x_i\Hom_{\mathcal H_A[\Gamma]}\left(V_i/V_{i-1},W/W'\right)=\{0\}.
\]
If $J\in\Hom_{\mathcal H_A[\Gamma]}(V,W)$ satisfies $J(V_1)\subseteq W'$, then
\[
 \left(\prod_{i=2}^n x_i\right)J\in\Hom_{\mathcal H_A[\Gamma]}(V,W').
\]
Equivalently, $\left(\prod_{i=2}^n x_i\right)J(V)\subseteq W'$.
\end{proposition}
\begin{proof}
	If $N$ is any smooth $A$-representation of $\Gamma$ and $I\in \Hom_{\mathcal H_A[\Gamma]}(N,W)$, we denote by $\overline{I}\in \Hom_{\mathcal H_A[\Gamma]}(N,\frac{W}{W'})$ the composition of $I$ with the natural projection $W\to \frac{W}{W'}$. We now prove recursively that $(\prod_{i=2}^kx_i)J(V_k)\subseteq W'$, i.e. that $(\prod_{i=2}^kx_i)\overline{J}$ vanishes on $V_k$. 
	By assumption $\overline{J}_{|V_1}=0$. Now since $x_2 \Hom_{\mathcal H_A[\Gamma]}(\frac{V_{2}}{V_{1}},\frac{W}{W'})=\{0\}$, we infer that $x_2\overline{J}_{|V_2}=0$. More generally, if $a_{k-1}\overline J$ vanishes on $V_{k-1}$, where $a_{k-1}=\prod_{i=2}^{k-1}x_i$, then its restriction to $V_k$ factors through $V_k/V_{k-1}$; hence $x_ka_{k-1}\overline J$ vanishes on $V_k$. This proves the claim by induction.
\end{proof}

We now further assume that $R$ is an integral domain. We record the following consequence of \Cref{prop habile}.

\begin{corollary}\label{cor habile}
Assume that $A$ is an integral domain, with fraction field $F_A$, and set $Q_A:=\frac{F_A}{A}$. Let $\eta:\Gamma\to A^\times$ be a smooth character. Let $V$ be a smooth $A$-representation of $\Gamma$ equipped with a $\Gamma$-stable $A$-module filtration \[\{0\} \subseteq V_1\subseteq \dots \subseteq V_n=V.\] Assume that for each $i=2,\dots,n$, there exists an element $x_i\in A$ such that \[x_i \Hom_{\mathcal H_A[\Gamma]}(\frac{V_{i}}{V_{i-1}},Q_A\otimes\eta)=\{0\} .\] Then if $J \in \Hom_{\mathcal H_A[\Gamma]}(V,F_A\otimes\eta)$ is such that $J(V_1)\subseteq A$, then \[(\prod_{i=2}^nx_i) J \in \Hom_{\mathcal H_A[\Gamma]}(V,A\otimes\eta),\] i.e. $(\prod_{i=2}^nx_i)J(V)\subseteq A$.
\end{corollary}
\begin{proof}
Apply \Cref{prop habile} with $W=F_A\otimes\eta$ and $W'=A\otimes\eta$.
\end{proof}

\section{Symmetric pairs, parabolic orbits and tori of unramified characters}\label{sec pairs}

Let $F$ be a non-archimedean local field of characteristic different from $2$ and residual characteristic $p$, with normalized absolute value $|\ |_F$ and residual field of cardinality $q$. Let $(G,H,\theta)$ be a symmetric pair defined over $F$. This means that $G=\bG(F)$ consists of the $F$-points of a connected reductive group $\bG$ defined over $F$, that $\theta$ is an $F$-rational involution of $\bG$, and that $H$ is a closed subgroup of $G$ satisfying 
\[(\bG^{\theta})^{\circ}(F)\subseteq H \subseteq \bG^{\theta}(F)=G^\theta,\] where $(\ )^\circ$ stands for the neutral component; we use similar boldface conventions for the algebraic groups underlying the closed (algebraic) subgroups of $G$ that we shall consider. For any algebraic subgroup $\bX$ of $\bG$, we denote by $X^*(\bX)$ the group of algebraic characters of $\bX_{\overline F}$. The following fact is a consequence of classical results of Steinberg in \cite{StAMS80}. It is proved in  \cite[Proposition~7]{Vust74} when $F$ has characteristic zero, and Vust's argument remains valid in our setting. We provide a different proof. 

\begin{fact}\label{fact pi0}
The component group \[\pi_0(\bG^\theta):=\frac{\bG^\theta}{(\bG^{\theta})^{\circ}}\] is a finite elementary abelian $2$-group.
Consequently:
\begin{enumerate}
\item $H$ is an open normal subgroup of finite index of $G^\theta$;
\item $H$ contains every unipotent element of $G^\theta$.
\end{enumerate}
\end{fact}
\begin{proof}
We denote by $\ell$ the characteristic exponent of $F$; by assumption $\ell\neq 2$. We also observe that since the characteristic of $F$ is not $2$, any involution of a linear algebraic group over $\overline{F}$ is a semisimple automorphism of it. In (I) to (IV) below we prove the main claim, and (1) and (2) are justified in (V). 

(I) We first assume that $\bG=\bA$ is a torus. We write $M=X^*(\bA)$, it is a free finitely generated $\mathbb{Z}$-module. The functor $\bD\mapsto X^*(\bD)$ defines an anti-equivalence of categories between the category of diagonalizable groups over $\overline{F}$ and that of finitely generated $\Z$-modules with no $\ell$-torsion (see \cite{SatAG}). Under this anti-equivalence, the fixed point group $\bA^{\theta}=\ker(a\mapsto a\theta(a)^{-1})$ corresponds to the quotient of \[\frac{M}{(\theta-\Id)M}\] by its $\ell$-primary torsion subgroup, whereas for any diagonalizable group $\bD$, the group $X^*(\pi_0(\bD))$ is the torsion part of $X^*(\bD)$. Since $\theta^2=\Id$, it is elementary to check that the torsion part of $\frac{M}{(\theta-\Id)M}$ is isomorphic to a finite product $(\Z/2\Z)^r$ of copies of $\Z/2\Z$, which has trivial $\ell$-primary part because $\ell\neq 2$. We deduce that \[X^*(\pi_0(\bA^\theta))\simeq (\Z/2\Z)^r,\] and using again that $\ell\neq 2$, that $\pi_0(\bA^\theta)$ is itself isomorphic to $(\Z/2\Z)^r$. Note that this step applies to any torus equipped with an involution, in particular to any $\theta$-stable torus of $\bG$.

(II) Now if we assume that $\bG$ is semisimple and simply connected, then $\bG^\theta$ is connected according to \cite[Theorem 8.1]{StAMS80}, the involution $\theta$ being a semisimple automorphism of $\bG$ since we assumed that $\ell\neq 2$. The result in this case follows at once. 

(III) Finally when $\bG$ is as in the statement, we denote by $\bA$ the neutral component of the center of $\bG$ and let $\bG_{sc}$ be the simply connected cover of the derived subgroup $\bG_{der}$ of $\bG$. The involution $\theta$ stabilizes $\bG_{der}$ and $\bA$, and its restriction to $\bG_{der}$ lifts uniquely to an automorphism of $\bG_{sc}$, which is involutive by uniqueness. We thus obtain an involution $\widetilde{\theta}$ of $\bG_{sc}\times \bA$ lifting $\theta$ along the central isogeny \[1\to \bF\to \bG_{sc}\times\bA \stackrel{\pi}{\to} \bG\to 1,\] which defines the finite central (hence abelian and $\widetilde{\theta}$-stable) subgroup $\bF$ of $\bG_{sc}\times \bA$. Taking fixed points on $\overline F$-points yields the exact sequence \[1\to \bF^{\widetilde{\theta}}(\overline F)\to (\bG_{sc}^{\widetilde{\theta}}\times \bA^{\theta})(\overline F)\to \bG^\theta(\overline F)\to H^1(\langle\theta\rangle,\bF(\overline F)).\] The last group is finite and killed by multiplication by $2$. Since $\bF$ is central, the connecting map to $H^1(\langle\theta\rangle,\bF(\overline F))$ is a group homomorphism. The image of $\bG_{sc}^{\widetilde{\theta}}\times\bA^\theta$ under $\pi$ is thus a closed normal finite index subgroup of $\bG^\theta$, and its component group is a quotient of $\pi_0(\bG_{sc}^{\widetilde{\theta}}\times \bA^\theta)=\pi_0(\bA^\theta)$ thanks to (II). Hence $\pi_0(\bG^\theta)$ is an extension of a subgroup of $H^1(\langle\theta\rangle,\bF(\overline F))$ by a quotient of $\pi_0(\bA^\theta)$, and (I) implies that its cardinality is a power of $2$.

(IV) The final step is to prove that $\pi_0(\bG^\theta)$ is in fact an elementary abelian  $2$-group. Let $g\in\bG^\theta$ with Jordan decomposition $g=su$; by uniqueness of the latter, both $s$ and $u$ lie in $\bG^\theta$. We first claim that $u\in(\bG^\theta)^\circ$: if $F$ has characteristic zero, then the Zariski closure of the group generated by $u$ is a connected subgroup of $\bG^\theta$ containing $u$, whereas if $F$ has characteristic $p>2$, then the image of $u$ in $\pi_0(\bG^\theta)$ has order both a power of $p$ and a power of $2$ by (III), hence is trivial. In any case, $g$ and $s$ have the same image in $\pi_0(\bG^\theta)$. Now $\bL:=Z_{\bG}(s)^\circ$ is a $\theta$-stable connected reductive subgroup of $\bG$, and $s$ is a central semisimple element of $\bL$ (it lies in a maximal torus of $\bG$, which is connected and centralizes $s$, hence is contained in $\bL$), so that $s$ belongs to every maximal torus of $\bL$. The semisimple automorphism $\theta|_{\bL}$ of $\bL$ stabilizes a maximal torus $\bT$ of $\bL$ by \cite[Theorem 7.5]{StAMS80}, so that $s\in \bT^{\theta}$ for the $\theta$-stable torus $\bT$. By (I), the group $\pi_0(\bT^\theta)$ is an elementary abelian $2$-group, so that \[s^2\in(\bT^\theta)^\circ\subseteq (\bG^\theta)^\circ.\] Hence every element of $\pi_0(\bG^\theta)$ has order at most $2$, and a group of exponent $2$ is an elementary abelian $2$-group, which proves the main claim.

(V) The group $(\bG^\theta)^\circ(F)$ is then the kernel of the natural continuous map from $G^\theta$ to the finite group $\pi_0(\bG^\theta)(\overline{F})$, hence is a closed normal subgroup of finite index of $G^\theta$, hence is open; as the quotient of $G^\theta$ by it is a finite elementary abelian $2$-group, any intermediate closed subgroup $H$ is open, normal and of finite index in $G^\theta$, and (1) follows. For (2), a unipotent element $u$ of $G^\theta$ lies in $(\bG^\theta)^\circ$ as shown in (IV), hence $u\in (\bG^\theta)^\circ(F)\subseteq H$.
\end{proof}

\begin{fact}\label{fact open}
Let $Q$ be a parabolic subgroup of $G$ and let $g\in G$. Then $QgH$ is open (resp. closed) in $G$ if and only if $QgG^\theta$ is. Moreover the double coset space $Q\bs G /H$ is finite.
\end{fact}
\begin{proof}
Because $H$ is normal and of finite index in $G^\theta$ (Fact \ref{fact pi0}), the set $S:=QgG^\theta$ is a finite disjoint union of distinct double cosets $Q(gh_k)H$, transitively permuted by right translation by $G^\theta$. If $QgH$ is open (resp. closed) in $G$, so are its translates, hence so is $S$. Conversely, suppose that $S$ is open or closed in $G$. It is then locally compact, and each of the above double cosets is locally closed in $S$ by \cite[Section 5.1, condition (3)]{BZ}. If they all had empty interior, they would all be nowhere dense, contradicting the Baire category theorem. Thus one of them has nonempty interior in $S$; transitivity of the actions of $Q\times H$ on each piece and of $G^\theta$ on the set of pieces then shows that every piece is open, hence also closed, in $S$. This proves the equivalence. Finally, the finiteness of $Q\bs G/H$ follows from that of $Q\bs G/G^\theta$ \cite[Proposition 6.15]{HW} together with Fact \ref{fact pi0} (1).
\end{proof}

We say that a parabolic subgroup $Q$ of $G$ is $\theta$-split if $Q\cap \theta(Q)$ is a Levi subgroup of $Q$, in which case we may say that $Q$ is $\theta$-split with respect to  $Q\cap \theta(Q)$. From now on we fix a parabolic subgroup $P$ of $G$ and a Levi decomposition $P=MV$ of $P$. By Fact \ref{fact open} the double coset space $P\backslash G/H$ is finite and we denote by $R(P\backslash G/H)$ a set of representatives.

We set \[X=\{x\in G,\ \theta(x)=x^{-1}\} \] to be the symmetric space attached to $G$ and $\theta$, on which the group $G$ acts by $\theta$-twisted conjugation:
\[
g\cdot x:=gx\theta(g)^{-1}
\qquad(g\in G,\ x\in X).
\] For $x\in X$, we set \[\theta_x:g\mapsto x\theta(g) x^{-1}.\] It is again an $F$-rational involution of $\bG$ since $x\theta(x)=e$. 
For $g\in G$, we define \[x_g:=g\cdot e \in G\cdot e\subseteq X,\] and abbreviate $\theta_g:=\theta_{x_g}$. One immediately checks the relations 
\begin{equation}\label{eq twists}
\theta_g=\ad(g)\circ \theta \circ \ad(g)^{-1},\qquad \bG^{\theta_g}=g\,\bG^{\theta}g^{-1},
\end{equation}
and \[ \theta_{g'g}=(\theta_{g})_{g'},\qquad x_{gh}=x_g \ \text{for}\ h\in G^\theta\]
for $g,g'\in G$.

\begin{remark} We warn the reader that the two twisting constructions differ on their common domain: for $x\in X$ regarded as an element of $G$, the involutions $\theta_x$ and $\theta_{x_x}$ do not agree in general (as $x_x=x^2$). We shall only ever twist by elements of $X$ through the first construction, and to prevent any ambiguity we reserve the letters $x$, $y$ for elements of $X$ and $u$, $g$, $h$ for elements of $G$ which are not considered as elements of $X$.
\end{remark}

We say that $x\in X$ is \textit{$(M,\theta)$-admissible}, or simply
$M$-\textit{admissible}, if $\theta_x(M)=M$.  We say that $x$ is
\textit{$(M,\theta)$-split} if $\theta_x(P)\cap P=M$, that is, if
$\theta_x(P)$ is opposite to $P$ with respect to $M$.  If $x_g$ is
$(M,\theta)$-split, then $PgH$ is open in $G$ by \cite[Proposition 13.4]{HW} and Fact~\ref{fact open}. 

\begin{remark}
	The converse of the above property asserted in \cite[Lemma~1.2]{LuMcomp} is false
	without an additional hypothesis. For example, take
	$G=\mathrm{GL}_2(F)$, $\theta(g)=\det(g)^{-1}g$,
	$H=G^\theta=\mathrm{SL}_2(F)$, and let $P$ be the upper triangular
	subgroup with diagonal Levi subgroup $M$. Then $PH=G$ and
	$\theta(M)=M$, whereas $\theta(P)=P$.
	
	To correct that lemma, condition~(b) should also require that
	$(M,x_g)$ be unimodular with respect to $P$, in the sense of
	\cite[Definition~1.4]{LuMcomp}, namely
	\begin{equation}\label{eq unim sl2}
	\delta_{P^{\theta_g}}
	=\delta_P^{1/2}|_{P^{\theta_g}},
	\end{equation}
	where the square root is the positive one. By twisting we may assume that $g=e$. Write
	$\mathfrak g,\mathfrak p,\mathfrak h$ for the corresponding
	Lie algebras. Since $PH$ is open, we have
	$\mathfrak g=\mathfrak p+\mathfrak h$, hence an
	$\mathrm{Ad}(P^\theta)$-equivariant isomorphism
	\[
	\mathfrak g/\mathfrak p
	\simeq \mathfrak h/\mathfrak p^\theta.
	\]
	Taking absolute values of determinants and using the
	unimodularity of $G$ and $H$ gives
	\[
	\delta_{P^\theta}=\delta_P|_{P^\theta}.
	\]
	Together with \eqref{eq unim sl2} above this implies
	$\delta_P|_{P^\theta}=1$.
	
	Finally let $A_M$ be the split component of $M$, and put
	$N=V\cap\theta(V)$, so that
	$P\cap\theta(P)=M\ltimes N$. Since $\theta(M)=M$, for every
	$a\in A_M$ we have $a\theta(a)\in M^\theta$. Consequently we obtain 
	\[
		1= \delta_P(a\theta(a))\\
		=\delta_P(a)\delta_{\theta(P)}(a)\\
		=\left|
		\det\bigl(\operatorname{Ad}(a)|_{\operatorname{Lie}(N)}\bigr)
		\right|_F^2.
	\]
	The last equality follows by comparing the roots of $A_M$ in
	$\operatorname{Lie}(V)$ and $\operatorname{Lie}(\theta(V))$:
	their common weights occur twice, and the remaining weights
	cancel in opposite pairs. Choose $a\in A_M$ such that
	$|\alpha(a)|_F<1$ for every root $\alpha$ of $A_M$ in
	$\operatorname{Lie}(V)$. The above displayed equality forces $N=\{e\}$,
	and therefore $P\cap\theta(P)=M$.
	
	This argument also applies over $\mathbb R$. Since unimodularity
	is assumed in \cite{LuMcomp} from Section~2 onwards, this
	correction does not affect the applications of Lemma~1.2,
	including the results on unitary periods.
\end{remark}

The class of symmetric pairs introduced above is stable under the operations that we shall perform, namely twisting by elements of $G$ and restriction to stable Levi subgroups.

\begin{lemma}\label{lm induced pair}
Let $g\in G$.
\begin{enumerate}
\item The triple $(G,gHg^{-1},\theta_{g})$ is a symmetric pair, with \[(\bG^{\theta_g})^\circ(F)=g(\bG^\theta)^\circ(F)g^{-1}.\]
\item If $\bL$ is a $\theta_g$-stable Levi $F$-subgroup of $\bG$ and $L:=\bL(F)$, then $(L,L\cap gHg^{-1},\theta_g|_L)$ is a symmetric pair.
\end{enumerate}
\end{lemma}
\begin{proof}
Point (1) follows from \eqref{eq twists}, since $\ad(g)$ is an $F$-rational isomorphism from $\bG^{\theta}$ onto $\bG^{\theta_g}$ and therefore maps neutral component onto neutral component. For (2), $(\bL^{\theta_g})^\circ$ is a connected subgroup of $\bG^{\theta_g}$, hence
\[(\bL^{\theta_g})^\circ(F)\subseteq L\cap (\bG^{\theta_g})^\circ(F)\subseteq L\cap gHg^{-1}\subseteq L^{\theta_g}.\]
\end{proof}

We set
\[M^1:=\bigcap_\chi\ker|\chi|_F,\]
where $\chi$ runs over the $F$-rational algebraic characters of $\bM$. It is well-known that \[\Lambda:=\frac{M}{M^1}\] is a free $\Z$-module of rank $r$ equal to that of the split component $A_M$ of $\bM$.

Let $R$ be a commutative ring. The group algebra $R[\Lambda]$ of the discrete group $\Lambda$ represents unramified characters:
\[\cX(R):=\Hom(\Lambda,R^\times)\simeq \Hom_{R\text{-}\mathrm{alg}}(R[\Lambda],R).\]
Thus $\cX(R)$ is the set of $R$-points of the torus $\cX$ of unramified characters of $M$. Writing $\mathcal O_R(\cX):=\Gamma(\cX,\mathcal O_{\cX})$ for its ring of regular functions, we have canonically \[\mathcal O_R(\cX)\simeq R[\Lambda],\] and we shall identify both rings with an equality.

If $x\in X$ is $M$-admissible, then 
$\theta_x$ acts on $\Lambda$, and one checks that the quotient \[\frac{\Lambda}{(\theta_x+1)\Lambda}\] has torsion of the form \[(\frac{\Lambda}{(\theta_x+1)\Lambda})_{tor}\simeq (\frac{\Z}{2\Z})^{a_x}\] for $0\leq a_x \leq r$. We set 
\[\Lambda_x:=(\frac{\Lambda}{(\theta_x+1)\Lambda})/(\frac{\Lambda}{(\theta_x+1)\Lambda})_{tor},\] and \[\cX_x(R):=\Hom(\Lambda_x,R^\times)\simeq \Hom_{R\text{-}\mathrm{alg}}(R[\Lambda_x],R).\] This is the set of $R$-points of the subtorus $\cX_x$ of $\cX$ with character lattice $\Lambda_x$. Over an algebraically closed field, it is the reduced neutral component of the diagonalizable group of $\theta_x$-anti-invariant unramified characters of $M$; see \cite[Remark~12.5]{Milne}. In particular it makes sense, for $\chi\in \cX_x(R)$ and $a\in \mathcal O_R(\cX_x)=R[\Lambda_x]$, to consider $a(\chi)=\mathrm{ev}_\chi(a)\in R.$

\section{Induced representations and the geometric lemma}\label{sec induced}

Here $R$ is a commutative $\Z[1/p]$-algebra, and we assume the existence of a square root 
$q^{1/2}\in R^\times$ of $q$, which is now fixed.

\subsection{Universal induced representations}

The notation $G$, $\theta$, $H$,
$X$ and $P=MV$ is as before.  We consider smooth $R$-representations of $G$
and of its closed subgroups, as in \cite{V}. Every such subgroup $J$ has a
compact open pro-$p$ subgroup. Since $p$ is invertible in $R$, the conventions
of \Cref{sec linear} apply to $J$ with coefficients in any commutative
$R$-algebra $B$. Fix an $M$-admissible
$x\in X$ and put
\[
 A_x:=\mathcal O_R(\cX_x)=R[\Lambda_x].
\]

We define the universal $\theta_x$-anti-invariant unramified character 
\[\Psi_{x}:M\to A_x^\times\] by \[\Psi_{x}(m)=[m]\in \Lambda_x\subseteq A_x^\times,\] where $[m]$ denotes the image of $m$ under $M\to\Lambda\to\Lambda_x$.

 The chosen square root of $q$ defines a square root $\delta_P^{1/2}:M\to R^\times$ of the modulus character $\delta_P$ of $P$. We let $\sigma$ be a smooth $R$-representation of $M$. We denote by $\Psi_x\otimes \sigma$ the $\mathcal H_{A_x}[M]$-module on which $M$ acts by $m\cdot (a\otimes v)=\Psi_x(m)a\otimes \sigma(m)v$. We consider the normalized parabolically induced representation
 \[I_{P,x}(\sigma):=I_P^G(\Psi_x\otimes \sigma),\] it is an $\mathcal H_{A_x}[G]$-module, where $G$ acts by right translation on the space of functions $f$ from $G$ to the $\mathcal H_{A_x}[M]$-module $\Psi_x \otimes V_\sigma$ which satisfy the transformation rule
 \[f(mvg)=\delta_P^{1/2}(m)(\Psi_x(m)\otimes \sigma(m))f(g),\] with $m\in M, v\in V, g\in G$, and are smooth.
 
\subsection{Frobenius reciprocity in all degrees}

We provide a proof of Frobenius reciprocity in all degrees for representations compactly induced from a closed subgroup. Recall that $H$ is the subgroup fixed in \Cref{sec pairs}; in particular, it is unimodular.

\begin{lemma}\label{lm ext frobenius}
Let $K$ be a closed subgroup of $H$, let $A$ be a commutative $R$-algebra, let $\eta:H\to R^\times$ be a smooth character, and let $W$ be an $A$-module, regarded as an $\mathcal H_A[H]$-module with trivial $H$-action. For every nondegenerate $\mathcal H_A[K]$-module $\tau$ and every $k\geq0$ there is an isomorphism
\[\Ext^k_{\mathcal H_A[H]}\big(\cInd_K^H\tau,\,W\otimes\eta\big)\simeq \Ext^k_{\mathcal H_A[K]}\big(\tau,\,W\otimes\delta_K\eta\big),\]
natural in $\tau$ and $W$, where the $\Ext$ groups are computed in the corresponding categories of nondegenerate modules.
\end{lemma}

\begin{proof}
We first treat the case $k=0$, by constructing mutually inverse maps between the two sides. For $v\in\tau$ and a compact open pro-$p$ subgroup $U$ of $H$ such that $v$ is fixed by $\tau(K\cap U)$ and $\eta_{|U}=1$, let $F_{v,U}\in\cInd_K^H\tau$ be the section supported on $KU$ defined by $F_{v,U}(ku):=\tau(k)v$ for $k\in K$ and $u\in U$; it is well defined because $v$ is fixed by $\tau(K\cap U)$. All volume ratios appearing below are powers of $p$, hence invertible in $R$.

Given $\Lambda\in\Hom_{\mathcal H_A[H]}(\cInd_K^H\tau,W\otimes\eta)$, set
\[\ell(v):=\frac{\vol_K(K\cap U)}{\vol_H(U)}\,\Lambda(F_{v,U}).\]
This does not depend on the choice of $U$: for $U'$ normal in $U$ one has $F_{v,U}=\sum_{u}u^{-1}\cdot F_{v,U'}$, the sum being over $u\in(K\cap U)\backslash U/U'$, and the volume ratios match. Moreover $F_{\tau(k_0)v,U}=k_0\cdot F_{v,k_0^{-1}Uk_0}$ for $k_0\in K$, so that
\[\ell(\tau(k_0)v)=\eta(k_0)\,\frac{\vol_K(K\cap U)}{\vol_K\big(k_0^{-1}(K\cap U)k_0\big)}\,\ell(v)=\delta_K(k_0)\,\eta(k_0)\,\ell(v),\]
since $H$ is unimodular. Hence $\ell\in\Hom_{\mathcal H_A[K]}(\tau,W\otimes\delta_K\eta)$.

Conversely, if $F\in\cInd_K^H\tau$ is right-invariant under $U$ as above and supported on $\bigsqcup_{j}Kh_jU$, then
\[F=\sum\nolimits_j h_j^{-1}\cdot F_{F(h_j),\,h_jUh_j^{-1}},\]
so that the sections $F_{v,V}$ and their translates span $\cInd_K^H\tau$. Given $\ell\in\Hom_{\mathcal H_A[K]}(\tau,W\otimes\delta_K\eta)$, the assignment
\[
 \Lambda(F):=\sum\nolimits_j
 \eta(h_j)^{-1}\,
 \frac{\vol_H(U)}{\vol_K\big(K\cap h_jUh_j^{-1}\big)}\,
 \ell\big(F(h_j)\big)
\]
is independent of the choices of $U$ and of the representatives $h_j$.
Indeed, replacing $h_j$ by $kh_j$, with $k\in K$, produces the three factors
$\eta(k)^{-1}$, $\delta_K(k)\eta(k)$ and $\delta_K(k)^{-1}$, which cancel.
The assignment therefore defines an element of
\[
 \Hom_{\mathcal H_A[H]}(\cInd_K^H\tau,W\otimes\eta).
\]
The two assignments are mutually inverse, as one checks on the sections
$F_{v,U}$, and natural in $\tau$ and $W$.  This settles the case $k=0$.
By Yoneda's lemma it also establishes, for every nondegenerate $\mathcal H_A[K]$-module $\sigma$, the
coinvariant identity
\[
 (\cInd_K^H\sigma)_H\simeq(\sigma\otimes\delta_K^{-1})_K.
\]

This proof applies to arbitrary $W$.  For finite-dimensional $W$ over
$\C$, the degree-zero identity also follows from the contragredient formula
and Frobenius reciprocity in \cite[1.9 (b), (d)]{BZ}:
\[
 \Hom_H(\cInd_K^H\sigma,\widetilde\tau)
 \simeq\Hom_K\big(\sigma,\widetilde{(\tau|_K)}\otimes\delta_K\big).
\]
Here the contragredient is taken after restriction to $K$.

We now pass to higher degrees. For $J=K$ or $J=H$, consider
$\mathcal H_A[J]$ with the right regular action
$(j\cdot f)(x)=f(xj)$. This is a projective generator of the
category of smooth $A$-representations of $J$, by
\cite[proof of Proposition~2.1]{Thorne}.
The functor $\cInd_K^H$ is exact and commutes with direct sums.
We have a natural isomorphism
\[
\cInd_K^H\big(\mathcal H_A[K]\big)\simeq\mathcal H_A[H],
\]
given by $F\mapsto(h\mapsto F(h)(1))$, with inverse
$\phi\mapsto\big(h\mapsto(k\mapsto\phi(kh))\big)$. Consequently, if $P_\bullet\to\tau$ is a resolution by direct sums of copies of $\mathcal H_A[K]$, then $\cInd_K^HP_\bullet\to\cInd_K^H\tau$ is a projective resolution. Applying the degree-zero isomorphism termwise identifies the complexes
\[\Hom_{\mathcal H_A[H]}(\cInd_K^HP_\bullet,W\otimes\eta)\simeq\Hom_{\mathcal H_A[K]}(P_\bullet,W\otimes\delta_K\eta),\]
and taking cohomology gives the asserted isomorphisms.
\end{proof}

\subsection{The Bernstein--Zelevinsky geometric lemma à la Offen}\label{ssec geometric}

From now on, fix a smooth character $\eta:H\to R^\times$
trivial on unipotent elements of $H$; this condition is automatic
when $F$ has characteristic zero.

In this section we describe the space 
$\Hom_{\mathcal H_{A_x}[H]}(I_{P,x}(\sigma),W)$ for suitable $\mathcal H_{A_x}[H]$-modules $W$, following \cite{OJNT}, which adapts to symmetric pairs the method of \cite{BZ}. For $S\subseteq G$ and $x\in X$ of the form $x=x_u$ with $u\in G$, we write
\[S_x:=S\cap uHu^{-1}.\] This does not depend on the choice of $u$ such that $x=x_u$: if $x_u=x_{u'}$ then $h:=u^{-1}u'$ satisfies $\theta(h)=h$, so that $u'Hu'^{-1}=u(hHh^{-1})u^{-1}=uHu^{-1}$, the subgroup $H$ being normal in $G^\theta$ (\Cref{fact pi0} (1)). By \eqref{eq twists} we have $S_x\subseteq S^{\theta_x}$, with equality when $H=G^\theta$; for $\theta_x$-fixed points we always use the superscript notation $S^{\theta_x}$, so that no ambiguity can arise. For $Q=LU$ a parabolic subgroup of $G$ and $x\in X$, we set 
\[\delta_{Q,x}:=\delta_{Q^{\theta_x}}\delta_Q^{-1/2},\] a character of $Q^{\theta_x}$.

We write 
$G=\coprod_{i=1}^N Pu_i H$, which is legitimate by Fact \ref{fact open}. By the usual ordering lemma \cite[Section 5.1, condition (3)]{BZ}, we may order the representatives so that \[Y_k:=\coprod_{i=1}^k Pu_i H\] is open for every $k=1,\dots,N$. Moreover we assume that the $u_i$ are chosen such that whenever $Pu_iH$ contains $u$ such that $x_u$ is $M$-admissible, then $x_i:=x_{u_i}$ itself is already $M$-admissible. We recall from \Cref{sec pairs} that if $x_i$ is $(M,\theta)$-split, then $Pu_iH$ is open and $x_i$ is $M$-admissible.

\begin{remark}\label{rk orbit}
If $u'=pu_ih\in Pu_iH$, then $\theta(h)=h$ yields $x_{u'}=p\cdot x_i$, where $p\cdot x:=px\theta(p)^{-1}$: the elements $x_u$ for $u\in Pu_iH$ thus form the full twisted $P$-orbit of $x_i$, exactly as in the case $H=G^\theta$. By \cite[Corollary 6.2]{OJNT}, the $M$-admissible elements of this orbit form a single twisted $M$-orbit, and since conjugation by $M$ induces the identity of $\Lambda=M/M^1$, the involution of $\Lambda$ induced by $\theta_{x_i}$, hence also $\Lambda_{x_i}$, $\cX_{x_i}$ and $\Psi_{x_i}$, do not depend on the choice of the $M$-admissible representative.
\end{remark}

Choose the representatives $u_i$ as in \cite[Lemmas 3.2 and 3.3]{OJNT}. Then, for each $i=1,\dots,N$, there exist a parabolic subgroup $Q_i=L_iU_i$ of $G$ contained in $P$, with $L_i\subseteq M$ and $V\subseteq U_i$, and a subgroup $N_i$ of $P^{\theta_{x_i}}$ consisting of unipotent elements, such that 
\begin{equation}\label{eq stab}
P^{\theta_{x_i}}=L_i^{\theta_{x_i}}\ltimes N_i,\qquad \mathrm{pr}_M(N_i)=M\cap U_i,
\end{equation}
where $\mathrm{pr}_M:P\to M$ denotes the projection along $V$.

\begin{remark}\label{rk standard}
The results of \cite{OJNT} are stated for $P$ standard with respect to a fixed pair $(A_0,P_0)$, where $A_0$ is a $\theta$-stable maximal $F$-split torus of $G$ \cite[Lemma 2.4]{HW} contained in the minimal parabolic subgroup $P_0$, and for representatives chosen as in \cite[Lemma~3.2]{OJNT}. This entails no loss of generality: write $P=gP_\star g^{-1}$ and $M=gM_\star g^{-1}$ with $P_\star \supseteq P_0$ standard and $M_\star\supseteq A_0$ a standard Levi subgroup of $P_\star$, and set
\[x_\star=g^{-1}x\theta(g),\qquad \sigma_\star=\sigma\circ\ad(g)\quad\text{on }M_\star.\]
Left translation by $g^{-1}$ identifies $P\bs G/H$ with $P_\star \bs G /H$, and the identities 
\[x_{g^{-1}u}=g^{-1}x_u\, \theta(g^{-1})^{-1},\qquad \theta_{g^{-1}u}=\ad(g)^{-1}\circ \theta_u\circ \ad(g)\] 
show that all attached data (stabilizers and their decompositions \eqref{eq stab}, modulus characters, the action on $\Lambda$, hence $\Lambda_x$, $\Psi_x$, $A_x$, and the representations $I_{P,x}(\sigma)$ together with their filtrations introduced below) correspond to the analogous data attached to $(x_\star,\sigma_\star)$ under transport of structure by $\ad(g)^{-1}$ -- the involution $\theta$, the subgroup $H$ and the pair $(A_0,P_0)$ being untouched. We use this reduction without further comment.
\end{remark}

The following lemma is the point where our conventions on $H$ intervene.

\begin{lemma}\label{lm stab}
For each $i$ we have $N_i\subseteq u_iHu_i^{-1}$, so that 
\[P_{x_i}=(L_i)_{x_i}\ltimes N_i.\]
Moreover, $N_i\subseteq U_i$ and consequently
\[P_{x_i}=(Q_i)_{x_i}.\]
The group $P_{x_i}$ (resp. $(L_i)_{x_i}$) is an open normal subgroup of finite index of $P^{\theta_{x_i}}$ (resp. of $L_i^{\theta_{x_i}}$), and 
\[\delta_{P_{x_i}}=(\delta_{P^{\theta_{x_i}}})_{|P_{x_i}},\qquad \delta_{(Q_i)_{x_i}}=(\delta_{(Q_i)^{\theta_{x_i}}})_{|(Q_i)_{x_i}}.\]
\end{lemma}
\begin{proof}
The elements of $N_i$ are unipotent and fixed by $\theta_{x_i}$, hence lie in 
\[(\bG^{\theta_{x_i}})^\circ(F)=u_i(\bG^{\theta})^\circ(F)u_i^{-1}\subseteq u_iHu_i^{-1}\] by \Cref{fact pi0} (2) and \eqref{eq twists}. If now $ln\in P_{x_i}$ with $l\in L_i^{\theta_{x_i}}$ and $n\in N_i$, then $l\in L_i\cap u_iHu_i^{-1}=(L_i)_{x_i}$, whence the semidirect decomposition of $P_{x_i}$. Since $\mathrm{pr}_M(N_i)=M\cap U_i$ and $V\subseteq U_i$, the decomposition $P=MV$ also gives $N_i\subseteq U_i$. Thus $P_{x_i}\subseteq (Q_i)_{x_i}$, and the reverse inclusion is immediate from $Q_i\subseteq P$. Since $u_iHu_i^{-1}$ is an open normal subgroup of finite index of $G^{\theta_{x_i}}$ (\Cref{fact pi0} conjugated by $u_i$), intersecting with $P$ and with $L_i$ gives the next assertion, and the equalities of modulus characters hold because the modulus character of an open subgroup of a locally compact group is the restriction of that of the ambient group.
\end{proof}

Set $\cV:=I_{P,x}(\sigma)$ and $\cV_0:=\{0\}$. For $i=1,\dots,N$, we set 
\[\cV_i:=\{f\in \cV,\ \mathrm{supp}(f)\subseteq  Y_i\}.\] Thus \[\cV_0\subseteq \cV_1\subseteq \dots \subseteq \cV_N=\cV\] is an $\mathcal H_{A_x}[H]$-module filtration. Now let $W$ be an $A_x$-module, regarded as an $\mathcal H_{A_x}[H]$-module with trivial $H$-action; when $R$ is an integral domain, the cases $W=A_x$, $W=F_x$ and $W=F_x/A_x$, for $F_x$ the fraction field of $A_x$, are the only ones that we shall need.

For a representation $\tau$ of a subgroup $S$ of $G$ and $g\in G$,
we denote by ${}^g\tau$ the representation of $gSg^{-1}$ on the same
space, defined by
\[
({}^g\tau)(s)=\tau(g^{-1}sg),\qquad s\in gSg^{-1}.
\]

\begin{proposition}\label{prop geometric}
For each $i=1,\dots,N$ and every smooth character $\eta:H\to R^\times$ trivial on unipotent elements, there is an isomorphism
\[\Hom_{\mathcal H_{A_x}[H]}(\cV_i/\cV_{i-1},W\otimes \eta)\simeq \Hom_{\mathcal H_{A_x}[(L_i)_{x_i}]}\big(\Psi_x\otimes r_{L_i,M}(\sigma),W\otimes \delta_{Q_i,x_i}\eta^{u_i}\big).\]
Here $r_{L_i,M}(\sigma)$ is the normalized Jacquet module of $\sigma$, $\delta_{Q_i,x_i}$ is restricted to $(L_i)_{x_i}$, and $\eta^{u_i}(m):=\eta(u_i^{-1}mu_i)$.
\end{proposition}
\begin{proof}
As in the proof of \cite[Proposition 4.1]{OJNT}, restriction of functions to $Pu_iH$ identifies $\cV_i/\cV_{i-1}$ with $\cInd_{H\cap u_i^{-1}Pu_i}^{H}({}^{u_i^{-1}}(\delta_P^{1/2}\Psi_x\otimes \sigma))$. Here $H$ is unimodular because it is open in the group of $F$-points of the possibly disconnected reductive group $\bG^\theta$, so the case $k=0$ of \Cref{lm ext frobenius}, applied to the closed subgroup $H\cap u_i^{-1}Pu_i$ of $H$, computes $\Hom_{\mathcal H_{A_x}[H]}(\cV_i/\cV_{i-1},W\otimes\eta)$. Transporting by $\ad(u_i)$ therefore gives
\[\Hom_{\mathcal H_{A_x}[P_{x_i}]}\big(\Psi_x\otimes\sigma,\ W\otimes\delta_{P_{x_i}}\delta_P^{-1/2}\eta^{u_i}\big).\]
Put $R_i=L_i(M\cap U_i)$, a parabolic subgroup of $M$. By \Cref{lm stab},
\[P_{x_i}=(Q_i)_{x_i}=(L_i)_{x_i}\ltimes N_i,\qquad \mathrm{pr}_M(N_i)=M\cap U_i.\]
The characters $\Psi_x\circ\mathrm{pr}_M$ and $\eta^{u_i}$ are trivial on $N_i$, the latter because $\eta$ is trivial on unipotent elements. Taking $N_i$-coinvariants and using
\[r_{L_i,M}(\sigma)=\delta_{R_i}^{-1/2}\sigma_{M\cap U_i},\qquad
  \delta_{Q_i}|_{L_i}=\delta_P|_{L_i}\,\delta_{R_i},\]
we obtain
\[\Hom_{\mathcal H_{A_x}[(L_i)_{x_i}]}\big(\Psi_x\otimes r_{L_i,M}(\sigma),\ W\otimes\delta_{(Q_i)_{x_i}}\delta_{Q_i}^{-1/2}\eta^{u_i}\big).\]
Finally, \Cref{lm stab} identifies $\delta_{(Q_i)_{x_i}}$ with the restriction of $\delta_{(Q_i)^{\theta_{x_i}}}$, which proves the proposition.
\end{proof}

We finally observe that to say that $x_i$ is $M$-admissible is the same as saying that $L_i=M$ and $U_i=V$. 

\section{Non-archimedean intertwining periods}\label{sec periods}

Throughout this section $R$ is an algebraically closed field of characteristic
different from $p$, $\sigma$ is a finite-length smooth
$R$-representation of $M$, and $\eta$ is a smooth character of $H$ trivial on unipotent elements. We generalize the algebraic part of \cite{BD} to the setting of admissible intertwining periods. 

\subsection{Generic vanishing of orbit contributions}\label{sec gen vanish}

Here we establish the generic vanishing results needed
to define local intertwining periods algebraically by an existence
and uniqueness statement. Generalizing the argument of
\cite[Lemme~2.5(ii)]{BD} thanks to the formalism in \cite{OJNT}, we construct central elements of suitable
Levi subgroups whose action yields nonzero annihilators of the
$\Ext$ groups attached to orbits in the boundary of an admissible orbit.  

By \Cref{rk standard}, we may assume that $P$ and the Levi subgroups
chosen in \Cref{ssec geometric} are standard with respect to a
$\theta$-stable maximal $F$-split torus $A_0$ and a minimal parabolic
subgroup $P_0$ containing it. Write $W=N_G(A_0)/Z_G(A_0)$ and put
\[
\fa_0:=X_*(A_0)\otimes_\Z\R,\qquad
\fa_S:=X_*(A_S)\otimes_\Z\R\subseteq\fa_0
\]
for every Levi subgroup $S$ containing $A_0$, where $A_S$ is its split
component. Fix a rational inner product on $\fa_0$ invariant under $W$
and $\theta$, and denote orthogonal projection onto a subspace $E$ by
$\operatorname{pr}_E$. If $\theta_z(S)=S$, its action on $\fa_S$ is
orthogonal: conjugating $\theta_z(A_0)$ back to $A_0$ inside $S$ realizes
this action as the restriction of $w\theta$ for some $w\in W$.
For an involution $\tau$ of a vector space $U$, write $U^{\pm\tau}$ for
its $(\pm1)$-eigenspaces. All closures below are taken in the topology
of $F$-points.

\begin{lemma}\label{lm root separator}
Let $x=x_u$ be $M$-admissible, put $\cO_x:=PuH$, and let
$\cO_j:=Pu_jH$ be contained in
$\partial\cO_x:=\overline{\cO_x}\smallsetminus\cO_x$.
Choose $u_j$ and $Q_j=L_jU_j$ as in \Cref{ssec geometric}, and put
$x_j:=x_{u_j}$; in particular, $\theta_{x_j}(L_j)=L_j$. Then
\begin{equation}\label{eq root separator}
\operatorname{pr}_{\fa_M^{-\theta_x}}
\bigl(\fa_{L_j}^{\theta_{x_j}}\bigr)\neq\{0\}.
\end{equation}
\end{lemma}

\begin{proof}
Under the finite covering $G/H\to G/G^\theta\simeq G\cdot1\subseteq X$,
the $P$-orbits above a given $P$-orbit are open and closed in its inverse
image, as in Fact~\ref{fact open}. Hence
\[
x_j\in\overline{P\cdot x}\smallsetminus P\cdot x.
\]
The map $y\mapsto yx^{-1}$ replaces $\theta$ by $\theta_x$, sends $x$ to
$1$, and preserves the involutions attached to the other points.
We may therefore assume that $x=1$ and $\theta(M)=M$.

Suppose that the projection in \eqref{eq root separator} is zero. The orthogonal eigenspace
decomposition of $\fa_{L_j}$ then gives
\[
\fa_M^{-\theta}\subseteq\fa_{L_j}^{-\theta_{x_j}}.
\]
Put
\[
S:=(A_M^{-\theta})^\circ,\qquad L:=Z_G(S),\qquad P_L:=P\cap L.
\]
For every $\lambda\in X_*(S)$, the preceding inclusion gives
\[
\theta_{x_j}\circ\lambda=\theta\circ\lambda=\lambda^{-1}.
\]
Since $\theta_{x_j}=\ad(x_j)\circ\theta$, the element $x_j$ centralizes
the image of every such $\lambda$, hence $S$. Thus $x_j\in L$.
Moreover, $L$ is a $\theta$-stable Levi subgroup containing $M$, and
$P_L$ is a parabolic subgroup of $L$ with Levi subgroup $M$.
Every root $\alpha$ of $A_M$ in $L=Z_G(S)$ vanishes on
$\fa_M^{-\theta}$, hence $\theta$ fixes every root of $A_M$
in $L$, in particular those occurring in $P_L$. Together with
$\theta(M)=M$, this gives $\theta(P_L)=P_L$.

The Bruhat closure formula for $P_L\backslash L/P_L$ gives
$1\in P_L\subseteq\overline{P_Lx_jP_L}$; see
\cite[Theorem~3.13 and Corollary~3.15]{BT72}.
Since $P_L\subseteq P\cap\theta(P)$, it follows that
\[
P\theta(P)\subseteq\overline{Px_j\theta(P)}.
\]
Conversely, $x_j\in\overline{P\cdot1}\subseteq\overline{P\theta(P)}$
gives $Px_j\theta(P)\subseteq\overline{P\theta(P)}$.
The two double classes therefore have the same closure. Being locally closed,
they are both open and dense in this closure, so they intersect and hence
coincide. In particular,
\[
x_j\in Y:=X\cap P\theta(P).
\]

We now show that every $P$-orbit in $Y$ is open and closed in $Y$.
Tangent spaces are computed on the corresponding locally closed
algebraic intersection. Let $y\in Y$ and put
$\mathfrak p:=\operatorname{Lie}(P)$.
The stabilizer $P^{\theta_y}=(P\cap\theta_y(P))^{\theta_y}$ is smooth:
the intersection of the two parabolics is smooth, and $2$ is invertible
in $F$, so \cite[Theorem~13.9]{Milne} applies. The orbit map
\[
P\longrightarrow P\cdot y,\qquad p\longmapsto py\theta(p)^{-1},
\]
has surjective differential at $1$, by
\cite[Propositions~7.12 and~7.15]{Milne}.
Right translation by $y^{-1}$ identifies the relevant tangent spaces
with the following subspaces of $\mathfrak g$:
\begin{align*}
dR_{y^{-1}}(T_yX)&=\mathfrak g^{-\theta_y},\\
dR_{y^{-1}}\bigl(T_y(P\theta(P))\bigr)
&=\mathfrak p+\theta_y(\mathfrak p),\\
dR_{y^{-1}}\bigl(T_y(P\cdot y)\bigr)
&=(1-\theta_y)\mathfrak p.
\end{align*}
For an involution $\tau$ of a vector space and a subspace $U$, with
$2$ invertible, one has
\[
\ker(1+\tau)\cap(U+\tau U)=(1-\tau)U.
\]
Indeed, if $Z=v+\tau(w)$ belongs to the left-hand side, with $v,w\in U$,
then $Z=(1-\tau)((v-w)/2)$.
Consequently,
\[
(1-\theta_y)\mathfrak p
\subseteq dR_{y^{-1}}(T_yY)
\subseteq\mathfrak g^{-\theta_y}\cap
\bigl(\mathfrak p+\theta_y(\mathfrak p)\bigr)
=(1-\theta_y)\mathfrak p,
\]
and therefore
\[
\dim(P\cdot y)\leq\dim_yY\leq\dim T_yY=\dim(P\cdot y).
\]
Thus $Y$ is smooth at $y$ and the orbit map $P\to Y$ is a submersion
at the identity. The non-archimedean inverse function theorem implies
that $P\cdot y$ is open in $Y$. Its complement is a union of open
orbits, so it is also closed. In particular,
\[
\overline{P\cdot1}\cap Y=P\cdot1,
\]
contradicting $x_j\in\overline{P\cdot1}\smallsetminus P\cdot1$ and
$x_j\in Y$. This proves \eqref{eq root separator}.
\end{proof}

\begin{lemma}\label{lm BD}
With the notation of \Cref{lm root separator}, there is an element
\[z_{x,j}\in A_{L_j}^{\theta_{x_j}}\cap u_jHu_j^{-1}
       \subseteq Z((L_j)_{x_j})\]
whose image in $\Lambda_x$ has infinite order. Consequently
$\Psi_x(z_{x,j})-c$ is nonzero in $A_x$ for every $c\in R$.
\end{lemma}

\begin{proof}
Since $\fa_{L_j}^{\theta_{x_j}}$ is spanned over $\R$ by
$X_*(A_{L_j})^{\theta_{x_j}}$, \Cref{lm root separator}
yields a cocharacter
\[
\lambda\in X_*(A_{L_j})^{\theta_{x_j}}
\quad\text{such that}\quad
\operatorname{pr}_{\fa_M^{-\theta_x}}(\lambda)\neq0.
\]
For a uniformizer $\varpi$ of $F$, set
$z_{x,j}:=\lambda(\varpi)$. This element belongs to
$A_{L_j}^{\theta_{x_j}}$ and centralizes $L_j$.
The image of $\lambda$ is a connected algebraic subgroup of
$\bG^{\theta_{x_j}}$. Hence
\[
z_{x,j}\in(\bG^{\theta_{x_j}})^\circ(F)
=u_j(\bG^\theta)^\circ(F)u_j^{-1}
\subseteq u_jHu_j^{-1},
\]
by \eqref{eq twists} and our assumption on $H$.

Finally, under the natural identification
\[
\Lambda_x\otimes_\Z\R\simeq\fa_M^{-\theta_x},
\]
the image of $z_{x,j}=\lambda(\varpi)$ corresponds to
$\operatorname{pr}_{\fa_M^{-\theta_x}}(\lambda)$, which is nonzero.
Thus the image of $z_{x,j}$ in $\Lambda_x$ has infinite order, and
$\Psi_x(z_{x,j})$ and $1$ are distinct elements of the
canonical $R$-basis of $A_x=R[\Lambda_x]$. In particular,
$\Psi_x(z_{x,j})-c\neq0$ for every $c\in R$.
\end{proof}

By \cite[II.5.13 and II.2.8]{V}, the Jacquet module
$r_{L_j,M}(\sigma)$ has finite length, and its irreducible subquotients
have central characters.

\begin{corollary}\label{cor kill}
Keep the notation of \Cref{lm BD}, and let $\omega_1,\dots,\omega_r$ be the central characters of the
irreducible subquotients of $r_{L_j,M}(\sigma)$, counted with multiplicity.  Then
\begin{equation}\label{eq boundary polynomial}
p_{x,j}:=\prod_{i=1}^{r}
 \big(\delta_{Q_j,x_j}(z_{x,j})\eta(u_j^{-1}z_{x,j}u_j)
      -\Psi_x(z_{x,j})\omega_i(z_{x,j})\big)
\end{equation}
is a nonzero element of $A_x$ and, for every $A_x$-module $W$ with trivial
$H$-action and every $k\geq0$,
\[p_{x,j}\Ext^k_{\mathcal H_{A_x}[H]}(\cV_j/\cV_{j-1},W\otimes\eta)=\{0\}.\]
In particular,
\begin{equation}\label{eq boundary Ext vanishing}
\Ext^k_{\mathcal H_{F_x}[H]}\big((\cV_j/\cV_{j-1})\otimes_{A_x}F_x,
 F_x\otimes\eta\big)=\{0\},\qquad k\geq0.
\end{equation}
\end{corollary}

	\begin{proof}
	Since $N_j$-coinvariants and inflation are exact, the argument
	of \Cref{prop geometric}, using \Cref{lm ext frobenius},
	gives
	\[\Ext^k_{\mathcal H_{A_x}[H]}
		(\cV_j/\cV_{j-1},W\otimes\eta) \simeq
		\Ext^k_{\mathcal H_{A_x}[(L_j)_{x_j}]}
		(\Psi_x\otimes r_{L_j,M}(\sigma),
		W\otimes\delta_{Q_j,x_j}\eta^{u_j}).	\]
	Put $z:=z_{x,j}$. A composition series of
	$r_{L_j,M}(\sigma)$ induces a filtration of the source
	on whose successive quotients $z$ acts by
	$\Psi_x(z)\omega_i(z)$. On the target, $z$ acts by
	$\delta_{Q_j,x_j}(z)\eta^{u_j}(z)$.
	Thus \Cref{lm classique} shows that $p_{x,j}$ annihilates
	the displayed $\Ext$ group. Now each factor of $p_{x,j}$ is nonzero by \Cref{lm BD},
	so $p_{x,j}\neq0$ since $A_x$ is an integral domain.
	Over $F_x$, these factors are invertible, and the same
	application of \Cref{lm classique} gives
	\eqref{eq boundary Ext vanishing}.
\end{proof}

\subsection{Definition of algebraic intertwining periods}
\label{ssec characterization}

We fix an $M$-admissible $x=x_u$. The following space plays a crucial role in our definition of local intertwining periods. 

\begin{definition}
	Let $x=x_u$ be $M$-admissible. We set
	\[
	\begin{split}
		\cW(u):=\big\{\phi\in I_{P,x}(\sigma)\otimes_{A_x}F_x:\;&
		h\longmapsto\phi(uh)\text{ is compactly supported}\\[-2pt]
		&\text{modulo }H\cap u^{-1}Pu\big\}.
	\end{split}
	\]
\end{definition}

The reason is that it appears as follows in the filtration of $I_{P,x}(\sigma)\otimes_{A_x}F_x$. 

\begin{lemma}\label{lm adapted filtration}
The representatives in \Cref{ssec geometric} may be ordered so that, for
some $a\geq0$,
\[Y_a=G\smallsetminus\overline{\cO_x},\qquad
  Y_{a+1}=G\smallsetminus\partial\cO_x
          =(G\smallsetminus\overline{\cO_x})\sqcup\cO_x.\]
For the corresponding filtration put
$\cV_x^-:=\cV_a$ and $\cV_x^0:=\cV_{a+1}$.  The map $\phi\mapsto(h\mapsto\phi(uh))$ gives an exact sequence 
\begin{equation}\label{eq adapted exact}
0\longrightarrow\cV_x^-\longrightarrow\cV_x^0\longrightarrow
\cInd_{H\cap u^{-1}Pu}^{H}
 ({}^{u^{-1}}(\delta_P^{1/2}\Psi_x\otimes\sigma))\longrightarrow0,
\end{equation}
and, after extension of scalars,
\begin{equation}\label{eq support space equals filtration}
 \cV_x^0\otimes_{A_x}F_x=\cW(u).
\end{equation}
All the successive strata after $\cV_x^0$ are contained in
$\partial\cO_x$.
\end{lemma}

\begin{proof}
	Apply the ordering lemma used in \Cref{ssec geometric} first to
	$G\smallsetminus\overline{\cO_x}$ and then to
	$\partial\cO_x$, inserting $\cO_x$ between these two blocks.
	This gives the required order.
	In view of the computations in the proof of \Cref{prop geometric}, restriction from $Y_{a+1}$ to its closed subset $\cO_x$, with
	open complement $Y_a$, gives \eqref{eq adapted exact} by the
	restriction and compact-induction constructions in
	\cite[Sections~5.11 and~5.14--5.15]{BZ}. These arguments apply with coefficients in $A_x$. 
	Finally, since $P\backslash G$ is compact, the same restriction
	and extension-by-zero argument gives, after extension to $F_x$,
	\[
	\cW(u)
	=\bigl\{\phi\in I_{P,x}(\sigma)\otimes_{A_x}F_x:
	\phi|_{\partial\cO_x}=0\bigr\}
	=\cV_x^0\otimes_{A_x}F_x.
	\]
	This proves \eqref{eq support space equals filtration}.
\end{proof}

Now we consider \[\ell\in\Hom_{\mathcal H_{A_x}[M_x]}\big(\Psi_x\otimes\sigma,
A_x\otimes\delta_{P,x}\eta^u\big).\]

There is a canonical functional
\begin{equation}\label{eq pairing}
	j_{\eta,\ell}(\phi):=
	\int_{(H\cap u^{-1}Pu)\backslash H}
	\eta(h)^{-1}\ell\big(\phi(uh)\big)\,dh,\qquad \phi\in\cW(u).
\end{equation}
which defines an element of
$\Hom_{\mathcal H_{A_x}[H]}(\cW(u),F_x\otimes\eta)$ and is $A_x$-valued on
$\cW(u)\cap I_{P,x}(\sigma)$. The following result provides the algebraic definition of non-archimedean intertwining periods. 

\begin{theorem}\label{thm LIP}
There is a unique form
\[J_P^G(u,\eta,\ell,\sigma)\in
\Hom_{\mathcal H_{F_x}[H]}\big(I_{P,x}(\sigma)\otimes_{A_x}F_x,
 F_x\otimes\eta\big)\]
such that
\begin{equation}\label{eq support characterization}
J_P^G(u,\eta,\ell,\sigma)(\phi)=j_{\eta,\ell}(\phi)
\qquad(\phi\in\cW(u)).
\end{equation}
If $\partial\cO_x=\coprod_j\cO_j$ and the polynomials $p_{x,j}$ are chosen
as in \eqref{eq boundary polynomial}, then
\begin{equation}\label{eq direct denominator}
\left(\prod_{\cO_j\subseteq\partial\cO_x}p_{x,j}\right)
J_P^G(u,\eta,\ell,\sigma)\big(I_{P,x}(\sigma)\big)\subseteq A_x.
\end{equation}
In particular the intertwining period is nonzero whenever $\ell$ is nonzero.
\end{theorem}

\begin{proof}
	We use the filtration of \Cref{lm adapted filtration}.
	Every successive quotient after $\cV_x^0$ corresponds to an
	orbit in $\partial\cO_x$. The cases $k=0,1$ of
	\eqref{eq boundary Ext vanishing} and the long exact sequences
	of $\Ext$ therefore show that restriction induces an isomorphism
	\[
	\Hom_{\mathcal H_{F_x}[H]}
	\big(I_{P,x}(\sigma)\otimes_{A_x}F_x,F_x\otimes\eta\big)
	\xrightarrow{\ \sim\ }
	\Hom_{\mathcal H_{F_x}[H]}
	\big(\cW(u),F_x\otimes\eta\big),
	\]
	where we use \eqref{eq support space equals filtration}.
	Thus $j_{\eta,\ell}$ extends uniquely to the required period. 
	We now prove \eqref{eq direct denominator}. We observe that the period is
	$A_x$-valued on $\cV_x^0$, where it coincides with $j_{\eta,\ell}$.
	For every $j>a+1$, \Cref{cor kill}, applied with $k=0$ and
	$W=F_x/A_x$, gives
	\[
	p_{x,j}\,
	\Hom_{\mathcal H_{A_x}[H]}
	\left(\cV_j/\cV_{j-1},(F_x/A_x)\otimes\eta\right)=0.
	\]
	The asserted denominator bound of \cref{eq direct denominator} is therefore a consequence of \Cref{cor habile}
	applied to the filtration
	\[
	\cV_x^0=\cV_{a+1}\subseteq\cV_{a+2}\subseteq\cdots
	\subseteq\cV_N=I_{P,x}(\sigma).
	\]

	Finally, if $\ell\neq0$, Frobenius reciprocity
	(\Cref{lm ext frobenius}) and the surjectivity in
	\eqref{eq adapted exact} give $j_{\eta,\ell}\neq0$,
	so its extension is nonzero.
\end{proof}

\subsection{Compatibility with transitivity of parabolic induction}\label{ssec transitivity}

We now show that an intertwining period for an induced representation
can be computed in two stages: first inside an intermediate Levi subgroup,
then in $G$. When $R=\C$ this is \cite[Proposition~3.7]{LuMcomp}, and in the algebraic setting the familiar reader will probably observe that the correct setting for such a property would rather be that of intertwining periods in families. For a field extension $K/R$, put $\sigma_K:=\sigma\otimes_R K$
and $\cX_{x,K}:=\cX_x\times_R K$. For $T$ a $K$-split torus we set \[
A_T:=\mathcal O_K(T),\qquad F_T:=\operatorname{Frac}(A_T).
\] We need the following extended version of local intertwining periods to formulate the compatibility property.

\begin{lemma}\label{lm coefficient change}	Let $K/R$ be a field extension and let $\nu:T\to\cX_{x,K}$ be a homomorphism of split
	$K$-tori. Let $p_{x,j}$ be the polynomials defined in
	\eqref{eq boundary polynomial}.
	We keep the notation $\Psi_x$ and $p_{x,j}$ after applying
	the homomorphism $A_x\to A_T$ induced by $\nu$, and assume that all $p_{x,j}\neq0$ in $A_T$. 
	Put
	\[
	I_T:=I_P^G(\Psi_x\otimes_K\sigma_K).
	\]
	
	For every
	\[
	\ell\in\Hom_{\mathcal H_{A_T}[M_x]}
	\bigl(\Psi_x\otimes_K\sigma_K,
	A_T\otimes\delta_{P,x}\eta^u\bigr),
	\]
	there is a unique $H$-equivariant $F_T$-linear form
	\[
	J_P^G(u,\eta,\ell,\sigma):
	I_T\otimes_{A_T}F_T\longrightarrow F_T\otimes\eta
	\]
	such that
	\[
	J_P^G(u,\eta,\ell,\sigma)(\phi)=
	\int_{(H\cap u^{-1}Pu)\backslash H}
	\eta(h)^{-1}\ell\bigl(\phi(uh)\bigr)\,dh
	\]
	whenever $h\mapsto\phi(uh)$ has compact support modulo
	$H\cap u^{-1}Pu$. Moreover,
	\[
	\left(\prod_jp_{x,j}\right)
	J_P^G(u,\eta,\ell,\sigma)(I_T)\subseteq A_T.
	\]
\end{lemma}

\begin{proof}
Jacquet modules commute with scalar extension hence the statement of  \Cref{cor kill} extends with $A_T$ in place of $A_x$ (and $F_T$ in place of $F_x$). Since the $p_{x,j}$ are invertible in $F_T$,
the proof of \Cref{thm LIP} gives the result.
\end{proof}

\begin{lemma}\label{lm induced central characters}
	Let $L$ be a Levi subgroup containing $M$, let $P_L$ be a
	parabolic subgroup of $L$ with Levi subgroup $M$, and let
	$\chi:M\to K^\times$ be an unramified character. Put
	\[
	\tau:=I_{P_L}^L(\chi\otimes\sigma_K).
	\]
	For every parabolic subgroup $Q_N$ of $L$ with Levi subgroup
	$N$, the Jacquet module $r_{N,L}(\tau)$ has a finite filtration
	by $N$-subrepresentations on whose successive quotients $A_N$
	acts through $K$-valued characters.
	
	Consequently, \Cref{lm coefficient change} also holds with
	$L$ in place of $M$ and $\tau$ in place of $\sigma_K$, using
	the characters of these filtrations in
	\eqref{eq boundary polynomial}.
\end{lemma}

\begin{proof}
	The geometric lemma \cite[II.2.18]{V}, refined using composition
	series of the Jacquet modules of $\sigma$, gives successive
	quotients of the form
	\[
	I_{R_w}^N\bigl({}^w(\chi|_{S_w}\otimes\pi_K)\bigr).
	\]
	Here $w$ runs through representatives of $Q_N\backslash L/P_L$
	chosen as in the geometric lemma,
	\[
	S_w=M\cap w^{-1}Nw,\qquad R_w=N\cap wP_Lw^{-1},
	\]
	and $\pi$ is an irreducible subquotient of $r_{S_w,M}(\sigma)$,
	taken with respect to $M\cap w^{-1}Q_Nw$.
	The support filtration and the exactness of unipotent
	coinvariants remain valid over $K$, since $p$ is invertible
	in $K$.
	
	Let $\omega_\pi$ be the central character of $\pi$.
	Since $R_w$ has Levi subgroup $wS_ww^{-1}$ and
	$A_N\subseteq wA_{S_w}w^{-1}$, the action of $A_N$ on this
	quotient is
	\[
	a\longmapsto
	\chi(w^{-1}aw)\omega_\pi(w^{-1}aw)\in K^\times.
	\]
	These finite filtrations give the required central
	annihilators, so the preceding lemma applies as stated.
\end{proof}

Suppose that $P=MV\subseteq Q=LU$ with $M\subseteq L$, put
$P_L=P\cap L$, and write $u=u_1u_2$ with $u_1\in L$.
Assume that $x=x_u$ is $M$-admissible and $x_2=x_{u_2}$ is
$L$-admissible. Set
\[
H_2=L\cap u_2Hu_2^{-1},\qquad
\theta_2=\theta_{u_2}|_L,\qquad
y=u_1\theta_2(u_1)^{-1}.
\]
Then $(L,H_2,\theta_2)$ is a symmetric pair by
\Cref{lm induced pair}. Since $\theta_y=\theta_u|_L$, the point $y$
is $M$-admissible in this pair, and
\[
M_y:=M\cap u_1H_2u_1^{-1}=M\cap uHu^{-1}=M_x.
\]

Let $\cX^L(M)$ be the maximal subtorus of $\cX(M)$ consisting of
characters trivial on $A_L$, and let $\cX_y^L(M)$ be its maximal
$\theta_y$-anti-invariant subtorus. Multiplication of characters gives
an isogeny
\begin{equation}\label{eq torus transitivity}
	\mu:\cX_{x_2}(L)\times\cX_y^L(M)\longrightarrow\cX_x(M),
	\qquad(\chi_2,\chi_1)\longmapsto(\chi_2|_M)\chi_1.
\end{equation}
Indeed, on rational cocharacter spaces, the restriction map from
$\cX(L)$ and the inclusion of $\cX^L(M)$ give a direct sum
decomposition of $X_*(\cX(M))\otimes_\Z\Q$.
This decomposition is $\theta_u$-stable. On the first summand,
$\theta_u$ acts as
$\theta_{u_2}$, since $u_1\in L$ acts trivially on unramified
characters of $L$. Taking the $(-1)$-eigenspaces gives
\eqref{eq torus transitivity}.

Write
\[
\begin{gathered}
	A_1:=\mathcal O_R(\cX_y^L(M)),\qquad
	A_2:=\mathcal O_R(\cX_{x_2}(L)),\\
	\widetilde A:=A_1\otimes_R A_2,\qquad
	\widetilde F:=\operatorname{Frac}(\widetilde A),
\end{gathered}
\]
and denote the corresponding universal characters by
\[
\Psi_1:M\longrightarrow A_1^\times,
\qquad
\Psi_2:L\longrightarrow A_2^\times.
\] The isogeny $\mu$ gives injections
$A_x\hookrightarrow\widetilde A$ and $F_x\hookrightarrow\widetilde F$.
By \Cref{lm coefficient change}, the scalar extension of
$J_P^G(u,\eta,\ell,\sigma)$ is the period of the pulled-back data.
We continue to write $\ell$ for its pullback to $\widetilde A$.

Put
\begin{equation}\label{eq eta transitivity}
	\eta_2=(\delta_{Q,x_2}\eta^{u_2})|_{H_2}.
\end{equation}
This character is trivial on unipotent elements. On $M_x=M_y$ one has
\begin{equation}\label{eq modular transitivity}
	\delta_{P,x}\eta^u
	=\delta_{P_L,y}\,(\eta_2)^{u_1}.
\end{equation}
To see this, take determinants in the decomposition
$\operatorname{Lie}(V)=\operatorname{Lie}(V\cap L)
\oplus\operatorname{Lie}(U)$ and its fixed-point counterpart.
Conjugation by $u_1$ identifies $U^{\theta_{u_2}}$ with
$U^{\theta_u}$, giving
\[
\delta_{P,x}(m)=\delta_{P_L,y}(m)
\delta_{Q,x_2}(u_1^{-1}mu_1),\qquad m\in M_x.
\]
Inserting $\eta$ gives \eqref{eq modular transitivity}.
Moreover, $\Psi_2$ is trivial on $H_2$ and on its conjugates in $L$.
Thus $\ell$, after pullback to $\widetilde A$, is a valid datum for
the inner period with parameter $\Psi_1$.

We first check that the polynomials in
\eqref{eq boundary polynomial}, formed inside $L$, remain nonzero
when the parameter is restricted to $\cX_y^L(M)$.
For each double coset
\[
P_Lv_jH_2\subseteq
\overline{P_Lu_1H_2}\smallsetminus P_Lu_1H_2,
\]
let $z_{y,j}$ and $p_{y,j}$ be as in \Cref{lm BD} and
\eqref{eq boundary polynomial}, for the pair $(L,H_2,\theta_2)$
and the character $\eta_2$.
Thus $p_{y,j}$ is a nonzero regular function on
$\cX_y(M)=\cX_x(M)$.
Since $z_{y,j}\in v_jH_2v_j^{-1}$, every
$\chi_2\in\cX_{x_2}(L)$ satisfies $\chi_2(z_{y,j})=1$.
The formula for $p_{y,j}$ therefore gives
\[
p_{y,j}\bigl((\chi_2|_M)\chi_1\bigr)=p_{y,j}(\chi_1).
\]
The injectivity of $A_x\to\widetilde A$ induced by
\eqref{eq torus transitivity} now shows that the restriction of
$p_{y,j}$ to $\cX_y^L(M)$ is nonzero.
We keep the notation $p_{y,j}$ for this restriction and put
\[
d_1:=\prod_jp_{y,j}\in A_1\smallsetminus\{0\}.
\]

Apply \Cref{lm coefficient change} inside $L$, with data
$(P_L,y,\eta_2,\ell)$, $K=R$, and the map
\[
\cX_{x_2}(L)\times\cX_y^L(M)\longrightarrow\cX_y(M),
\qquad
(\chi_2,\chi_1)\longmapsto\chi_1.
\]
Here $A_T=\widetilde A$, so this gives the period
\[
\mathcal J_1=J_{P_L}^{L,H_2}(u_1,\eta_2,\ell,\sigma)
\]
over $\widetilde F$; the third superscript specifies the
symmetric subgroup.
Writing $\tau_{A_1}:=I_{P_L}^L(\Psi_1\otimes\sigma)$, we also obtain
\[
d_1\mathcal J_1
\bigl(\widetilde A\otimes_{A_1}\tau_{A_1}\bigr)
\subseteq\widetilde A.
\]

To use $\mathcal J_1$ as the datum for the period induced from $Q$,
put
\[
K_1:=\operatorname{Frac}(A_1),\qquad
\tau:=K_1\otimes_{A_1}\tau_{A_1},\qquad
B:=K_1\otimes_R A_2.
\]
Since $d_1$ is invertible in $B$, the preceding inclusion gives
\[
\mathcal J_1\in
\Hom_{\mathcal H_B[H_2]}
\bigl(B\otimes_{K_1}\tau,B\otimes\eta_2\bigr).
\]

By \Cref{lm induced central characters}, the Jacquet modules of
$\tau$ have finite filtrations on whose successive quotients the
split centers act through $K_1$-valued characters.
For each double coset contained in
$\overline{Qu_2H}\smallsetminus Qu_2H$, the element $z_{x_2,j}$
of \Cref{lm BD} satisfies
\[
\Psi_2(z_{x_2,j})\notin K_1.
\]
Consequently, every factor in \eqref{eq boundary polynomial},
formed for $\tau$ with parameter $\Psi_2$, is nonzero in $B$:
it has the form $a-b\Psi_2(z_{x_2,j})$ with
$a,b\in K_1^\times$.
Thus
\[
J_Q^G(u_2,\eta,\mathcal J_1,\tau)
\in
\Hom_{\mathcal H_{\widetilde F}[H]}
\bigl(
I_Q^G(\Psi_2\otimes_{K_1}\tau)\otimes_B\widetilde F,
\widetilde F\otimes\eta
\bigr).
\]

Over $\widetilde F$, transitivity of normalized parabolic induction
gives
\begin{equation}\label{eq Gamma transitivity}
	\Gamma:I_P^G((\Psi_2|_M)\Psi_1\otimes\sigma)
	\xrightarrow{\ \sim\ }I_Q^G(\Psi_2\otimes\tau),
\end{equation}
with
\[
(\Gamma\phi)(g)(l)
=\delta_Q(l)^{-1/2}\Psi_2(l)^{-1}\phi(lg).
\]
The following generalizes \cite[Proposition~3.7]{LuMcomp}.

\begin{theorem}\label{thm transitivity}
	With the preceding notation and compatible quotient measures,
	\begin{equation}\label{eq period transitivity}
		J_P^G(u_1u_2,\eta,\ell,\sigma)
		=c(u_1)\,J_Q^G(u_2,\eta,\mathcal J_1,\tau)\circ\Gamma
	\end{equation}
	over $\widetilde F$, where
	\[
	c(u_1)=\delta_Q(u_1)^{1/2}\Psi_2(u_1).
	\]
	In particular $c(u_1)=1$ if $u_1$ belongs to a compact subgroup of $L$.
\end{theorem}

\begin{proof}
	Work over $\widetilde F$ and put
	$T=J_P^G(u,\eta,\ell,\sigma)\circ\Gamma^{-1}$.
	Since $PuH\subseteq Qu_2H$, the support formula shows that
	$T$ vanishes on $\cV_{x_2}^-$ for the filtration associated with $Q$.
	By \Cref{lm adapted filtration,lm ext frobenius}, its restriction
	to $\cV_{x_2}^0$ therefore corresponds to a form
	\[
	\lambda\in\Hom_{\mathcal H_{\widetilde F}[H_2]}
	\bigl(\widetilde F\otimes_{K_1}\tau,
	\widetilde F\otimes\eta_2\bigr).
	\]
	
	Let $f\in\widetilde F\otimes_{K_1}\tau$ have compactly supported
	restriction to $P_Lu_1H_2$ modulo $P_L$.
	Take the section on $Qu_2H$ attached to $f$ in the Frobenius
	construction and lift it to $\cV_{x_2}^0$.
	Its inverse image under $\Gamma$ belongs to $\cW(u)$.
	Integration in stages along
	\[
	H\cap u^{-1}Pu\subseteq H\cap u_2^{-1}Qu_2\subseteq H,
	\]
	as in \cite[proof of Proposition~3.7]{LuMcomp}, gives
	\[
	\lambda(f)=c(u_1)j_{\eta_2,\ell}(f).
	\]
	Here we use \eqref{eq eta transitivity}, $\Psi_2|_{H_2}=1$, and
	\[
	(\Gamma^{-1}F)(u_1u_2h)=c(u_1)F(u_2h)(u_1).
	\]
	
	Uniqueness for the inner period gives $\lambda=c(u_1)\mathcal J_1$;
	uniqueness for the outer period then proves
	\eqref{eq period transitivity}.
\end{proof}

\subsection{Standard intertwining operators}\label{ssec intertwining}

Let $P=MV$ and $P'=M'V'$ be parabolic subgroups whose Levi components, and let $n\in G$ satisfy $nMn^{-1}=M'$. Let $\sigma$ be a nonzero finite-length smooth
$R$-representation of $M$, and put
\[
\mathcal A=\mathcal O_R(\cX(M))=R[\Lambda],\qquad
\mathcal F=\operatorname{Frac}(\mathcal A). 
\] A subscript $\mathcal F$ denotes extension of scalars from
$\mathcal A$ to $\mathcal F$. 
Let \[\Psi:M\to\mathcal A^\times, \quad \Psi(m)=[m]\in \Lambda \subseteq R[\Lambda]^\times\] be the universal unramified
character. The result below can be deduced from \cite[Section~7]{Dat} and \cite{MossTrias}. We deduce it as a corollary of \Cref{thm LIP}, as it is well-known that intertwining periods in the group case are intertwining operators in disguise; see for example \cite[Section 4]{LM}.  

\begin{lemma}\label{lm scalar-intertwining-denominator}
	There is a unique $G$-equivariant $\mathcal F$-linear map
	\[
	M(n,\sigma):
	I_P^G(\Psi\otimes\sigma)_{\mathcal  F}
	\longrightarrow
	I_{P'}^G\bigl({}^n(\Psi\otimes\sigma) \bigr)_{\mathcal  F}
	\]
	satisfying
	\begin{equation}\label{eq conv}
		\bigl(M(n,\sigma)\phi\bigr)(g)=
		\int_{V'\cap nVn^{-1}\backslash V'}
		\phi(n^{-1}vg)\,d v
	\end{equation}
	whenever
	\[
	V'\cap nVn^{-1}\backslash V'
	\longrightarrow\mathcal F\otimes_R V_\sigma,\qquad
	v\longmapsto\phi(n^{-1}vg)
	\]
	has compact support.
	Moreover, there exists $d\in\mathcal A\setminus\{0\}$ such that
	\begin{equation}\label{eq car CMop}
	d\,M(n,\sigma)I_P^G(\Psi\otimes\sigma)
	\subseteq I_{P'}^G\bigl({}^n(\Psi\otimes\sigma) \bigr).
	\end{equation}
\end{lemma}
\begin{proof}
 We denote by $\sigma^\vee$ the smooth contragredient of $\sigma$. Consider the symmetric pair
	\[
	\widetilde G=G\times G,\qquad
	\widetilde H=\Delta G,\qquad
	\theta(g_1,g_2)=(g_2,g_1),
	\]
	with parabolic subgroup $\widetilde P=P\times P'$,
	Levi subgroup $\widetilde M=M\times M'$, and
	\[
	u=(1,n),\qquad x=x_u=(n^{-1},n).
	\]
	Then $x$ is $\widetilde M$-admissible and
	\[
	\widetilde M_x
	=\{(m,nmn^{-1}):m\in M\}.
	\]
	Take
	\[
	\widetilde\sigma=\sigma\boxtimes{}^n\sigma^\vee,
	\qquad
	\ell_\Delta(v\otimes v^\vee)=v^\vee(v).
	\]
	 We have $\delta_{\widetilde P,x}=1$. 
	Moreover 
	\[
	\cX_{\widetilde M,x}
	=
	\{(\xi,{}^n\xi^{-1}):\xi\in\cX(M)\},
	\]
	and its algebra of regular functions identifies with $\mathcal A$. Since $\widetilde\sigma$ has finite length,
	\Cref{thm LIP} gives a unique $\Delta G$-invariant form
	\[
	\mathcal J:I_P^G(\Psi\otimes\sigma)_{\mathcal F}\otimes_{\mathcal F} I_{P'}^G({}^n(\Psi^{-1}\otimes\sigma^\vee))_{\mathcal F}
	\longrightarrow\mathcal F,
	\]
	which is given on well-supported sections by 
	\[
	\mathcal J(f\otimes f^\vee)
	=
	\int_{(nPn^{-1}\cap P')\backslash G}
	\langle f(n^{-1}g),f^\vee(g)\rangle\,dg.
	\]
	The theorem also gives $d\in\mathcal A\setminus\{0\}$ such that
	\[
	d\,\mathcal J(I_P^G(\Psi\otimes\sigma)) \otimes_{\mathcal A} I_{P'}^G({}^n(\Psi^{-1}\otimes\sigma^\vee))\subseteq\mathcal A.
	\]
	
	By admissibility, the pairing
	\[
	\langle u,v^\vee\rangle_{P'}
	=\int_{P'\backslash G}\langle u(g),v^\vee(g)\rangle\,dg
	\]
	is perfect on
	\[
	I_{P'}^G({}^n(\Psi\otimes\sigma))_{\mathcal F}
	\times
	I_{P'}^G({}^n(\Psi^{-1}\otimes\sigma^\vee))_{\mathcal F}.
	\] It therefore defines a $G$-equivariant $\mathcal F$-linear map
	\[
	M(n,\sigma):I_P^G(\Psi\otimes\sigma)_{\mathcal F} \longrightarrow  I_{P'}^G({}^n(\Psi\otimes\sigma))_{\mathcal F}
	\]
	by
	\begin{equation}\label{eq group period intertwiner}
		\mathcal J(f\otimes f^\vee)
		=
		\langle M(n,\sigma)f,f^\vee\rangle_{P'}
		=
		\int_{P'\backslash G}
		\langle(M(n,\sigma)f)(g),f^\vee(g)\rangle\,dg.
	\end{equation}
	
	\Cref{eq conv} follows. Uniqueness in the statement follows from that in \Cref{thm LIP}, by the same perfect pairing argument. The pairing being also perfect over $\mathcal A$ by extension of sclaras, this proves \cref{eq car CMop} as well.
\end{proof}

As proven in \cite{Dat} and \cite{MossTrias}, these operators
satisfy the expected standard properties, such as reduced
multiplicativity (see \cref{cor reduced multiplicativity} below)
and compatibility with transitivity of parabolic induction.
The latter is \cite[Proposition~7.8(ii)]{Dat} and
\cite[Lemma~3.17(i)]{MossTrias}, and we recover it as a
corollary of \cref{thm transitivity}.

\begin{corollary}\label{cor intertwining induction}
	Suppose that $P$ and $P'$ are contained in a parabolic subgroup
	$Q=LU$, with $M,M'\subseteq L$, and that $n\in L$.
	Let $\Gamma_P$ and $\Gamma_{P'}$ be the normalized
	transitivity isomorphisms, given by
	\[
	(\Gamma_Af)(g)(l)=\delta_Q(l)^{-1/2}f(lg),
	\qquad A=P,P'.
	\]
	With compatible quotient measures, one has
	\[
	\Gamma_{P'}\circ M^G(n,\sigma)
	=
	\delta_Q(n)^{-1/2}
	I_Q^G\bigl(M^L(n,\sigma)\bigr)\circ\Gamma_P,
	\]
	where 
	\[
	M^L(n,\sigma):
	I_{P\cap L}^L(\Psi\otimes\sigma)_{\mathcal F}
	\longrightarrow
	I_{P'\cap L}^L({}^n(\Psi\otimes\sigma))_{\mathcal F}.
	\]
	In particular, the scalar factor is $1$ if $n$ belongs
	to a compact subgroup of $L$.
\end{corollary}

\begin{proof}
	Apply \Cref{thm transitivity} to
	$P\times P'\subseteq Q\times Q$ in the symmetric pair
	$(G\times G,\Delta G)$, with admissible point $(n^{-1},n)$
	and inducing representation $\sigma\boxtimes{}^n\sigma^\vee$.
	Using \cref{eq group period intertwiner} and the normalized
	transitivity isomorphisms, this gives equality of the pairings
	of the two sides with every section of
	$I_{P'}^G({}^n(\Psi^{-1}\otimes\sigma^\vee))_{\mathcal F}$.
	Perfectness of the pairing proves the assertion.
\end{proof}

\subsection{Simple functional equations}\label{ssec FE}

Fix, as in \Cref{rk standard}, a $\theta$-stable maximal $F$-split
torus $A_0$ contained in a minimal parabolic subgroup $P_0$. All parabolic subgroups in this subsection are standard with respect
to the preceding choice of $(A_0,P_0)$. 

\begin{definition}
Following \cite{LuMcomp}, we say that a \textit{vertex} is a pair $(M,x)$ with $M$ a standard
Levi subgroup of $G$ and $x=x_u$ an $M$-admissible element of $G\cdot e \subseteq X$.  
\end{definition}

Let $P=MV$ and $\alpha\in\Delta_P$.
Let $Q_\alpha=L_\alpha U_\alpha$ be the standard parabolic
containing $P$ corresponding to $\{\alpha\}\subseteq\Delta_P$.
The elementary symmetry associated with $(M,\alpha)$ is
\[
s_\alpha:=w_0^{L_\alpha}w_0^M\in W_{L_\alpha},
\]
where $w_0^{L_\alpha}$ and $w_0^M$ are the longest elements
of $W_{L_\alpha}$ and $W_M$ for the positive systems
induced by $P_0$. Choose a representative
$\dot s_\alpha\in N_{L_\alpha}(A_0)$ of $s_\alpha$ and take
$n\in\dot s_\alpha M$. Put \[ M_1:=nMn^{-1}, \quad P_\alpha:=P\cap L_\alpha,\] and let
$\overline{P_\alpha}$ be the opposite of $P_\alpha$ in $L_\alpha$
with respect to $M$. Then
\[
P_1:=n\overline{P_\alpha}n^{-1}U_\alpha
\]
is a standard parabolic subgroup of $G$ with Levi
component $M_1$ and we write $P_1=M_1V_1$, i.e. $V_1$ is the unipotent radical of $P_1$.  

\begin{definition} Let $(M,x)$ be a vertex. If
$\theta_x(\alpha)<0$ and $\theta_x(\alpha)\neq-\alpha$,
we set $x_1:=n\cdot x$ and write
\[
(M,x)\stackrel n\searrow(M_1,x_1).
\]
\end{definition}

This is an edge of the Lapid--Rogawski graph in the form developed in
\cite[Section~4.3]{MOYjfa}. We observe that conjugation by $n$ induces canonical isomorphisms
\[
\iota_n:\Lambda_x\simeq\Lambda_{x_1},\qquad
\iota_n:A_x\simeq A_{x_1},\qquad \iota_n:F_x\simeq F_{x_1}.
\]

\begin{lemma}\label{lm edge data}
Let $(M,x)\stackrel n\searrow(M_1,x_1)$ be an edge, choose $u$ with $x=x_u$,
and put $u_1=nu$. Conjugation by $n$  maps $M_x$ onto $(M_1)_{x_1}$ and
\begin{equation}\label{eq delta edge}
 \delta_{P_1,x_1}\circ\ad(n)=\delta_{P,x},\qquad
 \eta^{u_1}\circ\ad(n)=\eta^u\quad\text{on }M_x.
\end{equation}
Consequently, if $
\ell \in \operatorname{Hom}_{\mathcal H_{A_x}[M_x]}
\bigl(
\Psi_x \otimes \sigma,\,
A_x \otimes \delta_{P,x}\eta^u
\bigr),$ then the $A_{x_1}$-linear form defined by
\[
\ell_1(a\otimes v)
:=a\,\iota_n\bigl(\ell(1\otimes v)\bigr),
\qquad a\in A_{x_1},\quad v\in V_\sigma,
\]
belongs to
$
\Hom_{\mathcal H_{A_{x_1}}[(M_1)_{x_1}]}
\bigl(\Psi_{x_1}\otimes{}^n\sigma,\,
A_{x_1}\otimes\delta_{P_1,x_1}\eta^{u_1}\bigr).
$
\end{lemma}

\begin{proof}
By \eqref{eq twists},
$\theta_{x_1}=\ad(n)\circ\theta_x\circ\ad(n)^{-1}$, which gives the lattice
and stabilizer assertions.  The identity for $\eta$ is immediate, and that
for the modulus characters is \cite[Corollary~6.5]{OJNT}, transported by
\Cref{rk standard}; \Cref{lm stab} permits the passage from full fixed-point
groups to the stabilizers used here.
\end{proof}

Following the proof of \cite[Theorem~5.3]{MOYjfa}, we recall
why, for each edge $(M,x)\stackrel n\searrow(M_1,x_1)$,
restriction in the unramified parameter $\Psi$ from
$\cX(M)$ to the anti-invariant subtorus $\cX_x$
gives a well-defined rational family of standard
intertwining operators. So let $P$, $P_1$, $\alpha\in \Delta_P$, etc. be as above, and let $(M,x)\stackrel n\searrow(M_1,x_1)$ be an edge. Let
$\sigma$ be a finite-length smooth $R$-representation of $M$. Let
\[
\operatorname{res}^{L_\alpha}_M:
\cX(L_\alpha)\longrightarrow\cX(M),
\qquad \nu\longmapsto\nu|_M,
\]
and put
\[
T:=\cX(M)/
\operatorname{im}(\operatorname{res}^{L_\alpha}_M).
\]
We regard $\mathcal O_R(T)$ as a subalgebra of $\mathcal A$
by pullback along the quotient map from $\cX(M)$ to $T$.
With $\mathcal A$ and $\Psi$ as in \Cref{ssec intertwining},
let $M^{L_\alpha}(1,\sigma)$ denote the standard intertwining
operator defined in \Cref{ssec intertwining} for the ambient group $L_\alpha$,
$n=1$, and the parabolics $P_\alpha$ and
$\overline{P_\alpha}$. Applying \Cref{ssec intertwining} with ambient group
$L_\alpha$, we may choose a nonzero denominator
$d\in\mathcal A$ such that this intertwiner is defined over $\mathcal A[1/d]$. 

\begin{lemma}\label{lm d reg on T}
The denominator $d\in\mathcal  A-\{0\}$ of 	$M^{L_\alpha}(1,\sigma)$ can actually be chosen in $
\mathcal O_R(T)
=\mathcal A^{\operatorname{im}(\operatorname{res}^{L_\alpha}_M)}$. 
\end{lemma}	
\begin{proof}
		Put $J=M^{L_\alpha}(1,\sigma)$, and denote its source and
		target over $\mathcal A$ by $\mathcal I$ and $\mathcal I'$.
		The ideal
		\[
		\mathfrak D
		:=\{a\in\mathcal A:\ aJ(\mathcal I)\subseteq\mathcal I'\}
		\]
		contains $d$ and is therefore nonzero. Compatibility of $J$
		with unramified twists of $L_\alpha$ implies that $\mathfrak D$
		is stable under translations by the subtorus $S:=\operatorname{im}(\operatorname{res}^{L_\alpha}_M).$ 
		Indeed, such twists are trivial on the unipotent radicals,
		so this compatibility follows from \eqref{eq conv} and uniqueness. Since the $S$-action is diagonalizable, every component $d_\lambda$ of $d$ with respect to the weight $\lambda$ of $S$ belongs to $\mathfrak D$, as seen by considering the action of $S$ on $\mathcal A/\mathfrak D$. Choose a nonzero
		component $d_\lambda$ and a monomial $u$ occurring in it.
		Then $u$ is invertible and has the same $S$-weight as
		$d_\lambda$, whence
		\[
		0\neq u^{-1}d_\lambda
		\in\mathfrak D\cap\mathcal A^S
		=\mathfrak D\cap\mathcal O_R(T).
		\]
		Thus $u^{-1}d_\lambda$ is a denominator with the required property.
\end{proof}

Until the end of this paragraph, we select $d\in \mathcal O_R(T)-\{0\}$ thanks to the above lemma. For such a $d$ we now have 
\[M^{L_\alpha}(1,\sigma):\quad
	I_{P_\alpha}^{L_\alpha}(\Psi\otimes\sigma)
	\otimes_{\mathcal A}\mathcal A[1/d]\\
	\longrightarrow
	I_{\overline{P_\alpha}}^{L_\alpha}(\Psi\otimes\sigma)
	\otimes_{\mathcal A}\mathcal A[1/d].\]

Let
\[
\rho_x:\mathcal A\longrightarrow A_x,
\qquad \rho_x(\Psi(m))=\Psi_x(m).
\]
The composite
$\cX_x\hookrightarrow\cX(M)\twoheadrightarrow T$
induces the character-lattice map
\[
\rho_x^*:X^*(T)\hookrightarrow\Lambda\twoheadrightarrow\Lambda_x,
\qquad \Lambda=M/M^1.
\]
Under $\Lambda\otimes_\Z\R\simeq\fa_M$, the real span
of $X^*(T)$ in $\fa_M$ is the line
$\fa_M\cap\fa_{L_\alpha}^{\perp}$, spanned by the vector
dual to $\alpha$ for our invariant inner product fixed in \Cref{sec gen vanish}.
Since $\theta_x(\alpha)<0$, this vector is not fixed
by $\theta_x$, so its projection to $\fa_M^{-\theta_x}$
is nonzero. The map $\rho_x^*:X^*(T)\to \Lambda_x$
is therefore injective. Hence $\rho_x$ is injective
on $\mathcal O_R(T)$, and $\rho_x(d)\neq0$. Thus $\rho_x$ extends to $\mathcal A[1/d]\to F_x$,
and base change gives the following well-defined restricted intertwining operator 
\[
	M^{L_\alpha}(1,\sigma):\quad
	 I_{P_\alpha}^{L_\alpha}(\Psi_x\otimes\sigma)
	\otimes_{A_x}F_x\\
    \longrightarrow
	I_{\overline{P_\alpha}}^{L_\alpha}(\Psi_x\otimes\sigma)
	\otimes_{A_x}F_x.
\]

By \Cref{cor intertwining induction}, induction along
$Q_\alpha$, conjugation by $n$, and coefficient transport
through $\iota_n\otimes\mathrm{id}_{V_\sigma}$ give the
restricted standard intertwining operator
\[
M(n,\sigma):
I_{P,x}(\sigma)\otimes_{A_x}F_x
\longrightarrow
I_{P_1,x_1}({}^n\sigma)\otimes_{A_{x_1}}F_{x_1}.
\]
Here we use
\[
\Psi_{x_1}(nmn^{-1})=\iota_n(\Psi_x(m))
\qquad(m\in M).
\]
This operator is $\iota_n$-semilinear and is characterized,
as in \eqref{eq conv}, by
\[
(M(n,\sigma)\phi)(g)
=
(\iota_n\otimes\mathrm{id}_{V_\sigma})
\left(
\int_{(V_1\cap nVn^{-1})\backslash V_1}
\phi(n^{-1}vg)\,d\dot v
\right)
\]
whenever the integrand has compact support.
 
\begin{proposition}\label{prop FE}
With compatible Haar measures and notations as in \Cref{lm edge data}, for every edge $(M,x)\stackrel n\searrow(M_1,x_1)$,
\begin{equation}\label{eq simple FE}
 \iota_n\circ J_P^G(u,\eta,\ell,\sigma)=
 J_{P_1}^G(u_1,\eta,\ell_1,{}^n\sigma)\circ M(n,\sigma).
\end{equation}
\end{proposition}

\begin{proof}
Put
\[
 H_0=H\cap u^{-1}Pu,\qquad
 H_1=H\cap u_1^{-1}P_1u_1,\qquad
 Z_1=nV_xn^{-1}=nV^{\theta_x}n^{-1}.
\]
The second equality in the definition of $Z_1$ holds because every
unipotent element of $G^{\theta_x}$ belongs to
$(\bG^{\theta_x})^\circ(F)\subseteq uHu^{-1}$ by
\Cref{fact pi0,eq twists}.  By
\cite[Lemma~6.3]{OJNT} and \Cref{lm edge data},
\[
 u_1H_0u_1^{-1}=(M_1)_{x_1}\ltimes Z_1,
 \qquad
 u_1H_1u_1^{-1}=(M_1)_{x_1}\ltimes(V_1)_{x_1}.
\]
In particular $H_0\subseteq H_1$, and
\begin{equation}\label{eq beta edge}
 \beta:Z_1\backslash(V_1)_{x_1}\xrightarrow{\ \sim\ }H_0\backslash H_1,
 \qquad Z_1v\longmapsto H_0u_1^{-1}vu_1,
\end{equation}
is a homeomorphism.

Let $\operatorname{pr}_{L_\alpha}:Q_\alpha\to L_\alpha$ be the projection along
$U_\alpha$.  The rank-one decomposition
\[
 V_1=(V_1\cap nVn^{-1})(L_\alpha\cap V_1)
\]
induces a homeomorphism
\[
 (V_1\cap nVn^{-1})\backslash V_1
 \xrightarrow{\ \operatorname{pr}_{L_\alpha}\ }L_\alpha\cap V_1.
\]
On the other hand, \cite[Lemma~6.4(2)]{OJNT} gives the exact sequence
\[
 1\longrightarrow Z_1\longrightarrow(V_1)_{x_1}
 \xrightarrow{\ \operatorname{pr}_{L_\alpha}\ }L_\alpha\cap V_1
 \longrightarrow1.
\]
Consequently inclusion $(V_1)_{x_1}\hookrightarrow V_1$ induces a
homeomorphism
\begin{equation}\label{eq q edge}
 q:Z_1\backslash(V_1)_{x_1}\xrightarrow{\ \sim\ }
 (V_1\cap nVn^{-1})\backslash V_1.
\end{equation}

Let $\phi\in\cW(u)$, and denote by $C$ the support of
$h\mapsto\phi(uh)$ in $H_0\backslash H$.  For
$v\in(V_1)_{x_1}$ put $h_v=u_1^{-1}vu_1$.  Then
\begin{equation}\label{eq edge evaluation}
 n^{-1}vu_1h=uh_vh.
\end{equation}
For fixed $h\in H$, the support
on $(V_1\cap nVn^{-1})\backslash V_1$ of
\[
 v\longmapsto\phi(n^{-1}vu_1h)
\]
is carried by
\[
 (V_1\cap nVn^{-1})\backslash V_1
 \xrightarrow{\ q^{-1}\ }Z_1\backslash(V_1)_{x_1}
 \xrightarrow{\ \beta\ }H_0\backslash H_1
 \xrightarrow{\ \cdot h\ }(H_0\backslash H_1)h
\]
into $C\cap(H_0\backslash H_1)h$.  The latter is compact: it is the
intersection of $C$ with the closed subset $(H_0\backslash H_1)h$ of
$H_0\backslash H$.  Since the support in the source is closed, it is
therefore compact. Thus the integration formula \eqref{eq conv}, after restriction and coefficient
transport through $\iota_n$, applies at $g=u_1h$
for every $h\in H$. It follows that the support of
$h\mapsto(M(n,\sigma)\phi)(u_1h)$ in $H_1\backslash H$ is contained in the
image of $C$; in particular, $M(n,\sigma)\phi\in\cW(u_1)$.  Choose the
quotient measures in \eqref{eq beta edge} and \eqref{eq q edge} compatibly
with integration in stages.  Each $h_v$ is unipotent, and hence
$\eta(h_v)=1$.  By \eqref{eq edge evaluation}, the definition of
$\ell_1$ and the modular identity \eqref{eq delta edge},
\begin{equation}\label{eq supported intertwining identity}
 \begin{aligned}
 &\eta(h)^{-1}\ell_1\bigl((M(n,\sigma)\phi)(u_1h)\bigr)\\
 &\qquad=\int_{H_0\backslash H_1}
   \eta(h'h)^{-1}\iota_n(\ell\bigl(\phi(uh'h)\bigr))\,dh'.
 \end{aligned}
\end{equation}
Consequently, integration in stages gives
\[
 \begin{aligned}
 j_{\eta,\ell_1}(M(n,\sigma)\phi)
 &=\int_{H_1\backslash H}\int_{H_0\backslash H_1}
   \eta(h'h)^{-1}\iota_n(\ell\bigl(\phi(uh'h)\bigr))\,dh'\,dh\\
 &=\iota_n(j_{\eta,\ell}(\phi)).
 \end{aligned}
\]

Thus the $F_x$-linear form
\[
\iota_n^{-1}\circ
J_{P_1}^G(u_1,\eta,\ell_1,{}^n\sigma)\circ M(n,\sigma)
\]
satisfies \eqref{eq support characterization}.
The uniqueness in \Cref{thm LIP} proves \eqref{eq simple FE}.
\end{proof}

As a corollary, we obtain that standard intertwining operators are multiplicative along reduced decompositions of Weyl elements. This also follows from \cite[Proposition~7.8(i)]{Dat} and
\cite[Lemma~3.16]{MossTrias}. 

\begin{corollary}\label{cor reduced multiplicativity}
	Let $P^{(0)}=M^{(0)}V^{(0)}$ be a standard parabolic subgroup of $G$. 
	For $1\leq i\leq k$, choose $\alpha_i\in\Delta_{P^{(i-1)}}$,
	and let $L_iU_i$ be the standard parabolic subgroup containing
	$P^{(i-1)}$ associated with $\alpha_i$.
	Choose $n_i\in N_{L_i}(A_0)$ representing
	$w_i:=w_0^{L_i}w_0^{M^{(i-1)}},$
	and define
	\[
	M^{(i)}:=n_iM^{(i-1)}n_i^{-1},\qquad
	P^{(i)}:=n_i\overline{P^{(i-1)}\cap L_i}\,n_i^{-1}U_i,
	\]
	where the bar denotes the opposite parabolic subgroup in $L_i$
	with respect to $M^{(i-1)}$.
	Assume that the expression $w_k\cdots w_1$ is reduced. Let $\sigma$ be a finite-length smooth $R$-representation of
	$M^{(0)}$, and put
	\[
	r_0=1,\qquad r_i=n_i\cdots n_1,\qquad n=r_k,\qquad
	\sigma^{(i)}={}^{r_i}\sigma.
	\]
	Set
	\[
	\mathcal A=\mathcal O_R(\cX(M^{(0)})),\qquad
	\mathcal F=\operatorname{Frac}(\mathcal A),
	\]
	and let $\Psi:M^{(0)}\to\mathcal A^\times$ be the universal
	unramified character. For each $i$, put
	\[
	\Psi_i:M^{(i)}\to\mathcal A^\times,\qquad
	\Psi_i(m)=\Psi(r_i^{-1}mr_i).
	\]
	Then, with compatible quotient measures, one has
	\[
	M(n_k,\sigma^{(k-1)})\circ\cdots\circ M(n_1,\sigma)
	=M(n,\sigma)
	\]
	as $\mathcal F$-linear maps from
	$I_{P^{(0)}}^G(\Psi\otimes\sigma)_{\mathcal F}$ to
	$I_{P^{(k)}}^G(\Psi_k\otimes\sigma^{(k)})_{\mathcal F}$.
\end{corollary}

\begin{proof}
The proof is completely parallel to that of
\cref{lm scalar-intertwining-denominator}.
In the symmetric pair $(G\times G,\Delta G)$, put
\[
\widetilde P_i=P^{(i)}\times P^{(k)},\qquad
\widetilde M_i=M^{(i)}\times M^{(k)},\qquad
\widetilde x_i=(r_i n^{-1},nr_i^{-1}).
\]
The reduced expression gives the directed path
\[
(\widetilde M_0,\widetilde x_0)
\stackrel{(n_1,1)}{\searrow}
(\widetilde M_1,\widetilde x_1)
\stackrel{(n_2,1)}{\searrow}\cdots
\stackrel{(n_k,1)}{\searrow}
(\widetilde M_k,e).
\]
Indeed, reducedness gives
\[
-(\alpha_i,0)\neq \theta_{\widetilde x_{i-1}}(\alpha_i,0)
=(0,(w_k\cdots w_i)\alpha_i)<0.
\]
Iterating \Cref{prop FE} along this path and using
\cref{eq group period intertwiner} proves the assertion.
\end{proof}

For later use we record that a directed path is automatically reduced.  If
\[
 (M^{(0)},x^{(0)})\stackrel{n_1}\searrow\cdots
 \stackrel{n_k}\searrow(M^{(k)},x^{(k)}),
\]
put $n=n_k\cdots n_1$, $\sigma^{(i)}={}^{n_i\cdots n_1}\sigma$, and
transport $u$ and $\ell$ successively as above.

\begin{proposition}\label{prop mult}
	The associated expression in elementary symmetries is reduced, hence 
	\[
	M(n_k,\sigma^{(k-1)})\circ\cdots\circ M(n_1,\sigma)
	=M(n,\sigma).
	\]
	Consequently, with
	$\iota_n=\iota_{n_k}\circ\cdots\circ\iota_{n_1}$,
	\begin{equation}\label{eq path comparison}
		\begin{aligned}
			&\iota_n\circ
			J_{P^{(0)}}^G(u^{(0)},\eta,\ell^{(0)},\sigma^{(0)})\\
			&\qquad=
			J_{P^{(k)}}^G(u^{(k)},\eta,\ell^{(k)},\sigma^{(k)})
			\circ M(n,\sigma).
		\end{aligned}
	\end{equation}
\end{proposition}

\begin{proof}
Write $w^{(i)}=s_{\alpha_i}\cdots s_{\alpha_1}$ and let $d_i$ be the weight
of the $i$-th vertex in the sense of \cite[Section~3.3]{LR}.  Here, with $M=M^{(0)}$, $\ell_M(w)$ denotes the minimum
number of elementary symmetries in a decomposition of $w$
starting at $M$. By \cite[Lemma~3.2.1]{LR}, an elementary
symmetry changes the vertex weight by at most $2$, while
each directed edge decreases it by exactly $2$.
Thus $d_i=d_0-2i$, and
\[
2i=d_0-d_i\leq 2\ell_M(w^{(i)})\leq 2i.
\] 
Hence $\ell_M(w^{(i)})=i$ for every prefix.  Multiplicativity along this reduced expression follows from
\Cref{cor reduced multiplicativity} by base change, since
the denominators of the elementary operators remain nonzero
after restriction. Hence \eqref{eq path comparison} follows
by iterating \Cref{prop FE}.
\end{proof}

\subsection{Intertwining periods via extremal vertices and intertwining operators}
\label{ssec extremal}

In this section we express general  intertwining periods in terms of those attached to extremal vertices, by transporting them along a path thanks to a standard intertwining operator. The following extends the argument in
\cite[proof of Theorem~5.3]{MOYjfa} to our setting.

\begin{proposition}\label{prop dat}
	For every edge $(M,x)\stackrel n\searrow(M_1,x_1)$ and every
	finite-length smooth $R$-representation $\sigma$ of $M$,
	the operator $M(n,\sigma)$ is an isomorphism.
\end{proposition}

\begin{proof}
	Fix an edge and use the notation $P_\alpha$, $L_\alpha$ and $T$
	of \Cref{ssec FE}.
	Standard intertwining operators commute with inclusions
	and quotients by \cite[Proposition~3.13]{MossTrias}.
	By exactness of parabolic induction, a composition series
	of $\sigma$ induces filtrations of the source and target
	of $M(n,\sigma)$, and the maps on successive quotients
	are the operators $M(n,\tau)$ for the irreducible
	subquotients $\tau$ of $\sigma$.
	It therefore suffices to prove their invertibility.
	
	For such a $\tau$, generic irreducibility
	\cite[Corollary~1.5(3)(a)]{DHKM} and the nonvanishing of
	Dat's operators imply that the two rank-one operators
	over $\cX(M)$ are generically invertible.
	Their composites are multiplication by the nonzero
	rational function $j_\tau$ of
	\cite[Section~7.6]{Dat}, and
	\[
	j_\tau\in\operatorname{Frac}\bigl(\mathcal O_R(T)\bigr)^\times
	\]
	by \cite[Corollary~6.6]{MossTrias}. Since $\mathcal O_R(T)\to A_x$ is injective, $j_\tau$
	remains nonzero after restriction to \(\mathcal X_x\).
	Over $F_x$, each rank-one operator therefore has inverse
	$j_\tau^{-1}$ times the reverse operator.
	Compatibility with parabolic induction and conjugation by $n$
	now gives the invertibility of $M(n,\tau)$, as required.
\end{proof}

Fix $P=MV$, an $M$-admissible $x=x_u$, and data $\eta,\ell,\sigma$ as in
\Cref{thm LIP}.  Choose a minimal vertex $(M_1,x_1)$ and a maximal vertex
$(M_2,x_2)$, together with directed paths
\[
 (M,x)\rightsquigarrow(M_1,x_1),\qquad
 (M_2,x_2)\rightsquigarrow(M,x).
\]
Such paths exist by \cite[Corollary~6.7]{OJNT} and
\cite[Proposition~4.8]{MOYjfa}, respectively.  Let $P_i=M_iV_i$ be the
standard parabolic with Levi subgroup $M_i$, and let $n_i$ be the product
of the representatives along the $i$-th path, with the order of
multiplication in \Cref{prop mult}.  Set
\[
 u_1=n_1u,\qquad \sigma_1={}^{n_1}\sigma,\qquad
 u_2=n_2^{-1}u,\qquad \sigma_2={}^{n_2^{-1}}\sigma.
\]
In particular $x_i=x_{u_i}$. As in \Cref{lm edge data}, let $\ell_i$
be the $A_{x_i}$-linear forms determined, for $v\in V_\sigma$, by
\[
 \ell_1(1\otimes v)=\iota_{n_1}\bigl(\ell(1\otimes v)\bigr),\qquad
 \ell_2(1\otimes v)=\iota_{n_2}^{-1}\bigl(\ell(1\otimes v)\bigr).
\]

The periods at these two vertices can be computed by transitivity.
At the minimal vertex, \cite[Corollary~1]{MO} gives a
$\theta_{x_1}$-stable parabolic $Q_1=L_1U_1$ containing $P_1$ for which
$P_1\cap L_1$ is $\theta_{x_1}$-split.  Apply \Cref{thm transitivity} to
$P_1\subseteq Q_1$ and $u_1=e\cdot u_1$: first form the open period at
the identity for the pair $(L_1,L_1\cap u_1Hu_1^{-1})$, then the period
for $Q_1$ in $G$.  This gives
$J_{P_1}^G(u_1,\eta,\ell_1,\sigma_1)$.

At the maximal vertex, \cite[Lemma~4.7]{MOYjfa} gives a
$\theta_{x_2}$-split parabolic $Q_2=L_2U_2$ containing $P_2$ for which
$P_2\cap L_2$ is $\theta_{x_2}$-stable.  This time,
\Cref{thm transitivity}, applied to $P_2\subseteq Q_2$ and
$u_2=e\cdot u_2$, first forms the closed period at the identity for
$(L_2,L_2\cap u_2Hu_2^{-1})$, then the open period for $Q_2$ in $G$.
It gives $J_{P_2}^G(u_2,\eta,\ell_2,\sigma_2)$. Applying
\eqref{eq path comparison}, and in view of \Cref{prop dat}, we obtain the following.

\begin{proposition}\label{prop extremal periods}
With notation as above and 
	compatible Haar measures, one has
	\begin{equation}\label{eq from minimum}
		\iota_{n_1}\circ J_P^G(u,\eta,\ell,\sigma)=
		J_{P_1}^G(u_1,\eta,\ell_1,\sigma_1)\circ M(n_1,\sigma),
	\end{equation}
	and
	\begin{equation}\label{eq from maximum}
		\iota_{n_2}^{-1}\circ J_P^G(u,\eta,\ell,\sigma)=
		J_{P_2}^G(u_2,\eta,\ell_2,\sigma_2)\circ M(n_2,\sigma_2)^{-1}.
	\end{equation}
\end{proposition}

\subsection{Explicit bounds for singularities of non-archimedean intertwining periods}
\label{ssec denominators}

We give explicit denominators for non-archimedean intertwining periods.
Such bounds are useful in the study of distinguished representations,
as illustrated by \cite{MOY}. In \Cref{ssec galois speh}, 
we compute them for Galois periods of Speh representations on inner forms of $\GL_n$, with expectation to generalize the results in \cite{MOY} to the residual spectrum. 

We use the notation of \Cref{ssec extremal}.
The three bounds come from the geometric filtrations at $(M,x)$,
at the minimal vertex $(M_1,x_1)$, and at the maximal vertex $(M_2,x_2)$. For the orbits $\cO_j\subseteq\partial\cO_x$, let $\mathcal I_0$ be the
set of indices $j$ for which
\begin{equation}\label{eq nonzero geometric Hom}
	\Hom_{\mathcal H_{A_x}[(L_j)_{x_j}]}
	\bigl(\Psi_x\otimes r_{L_j,M}(\sigma),
	(F_x/A_x)\otimes\delta_{Q_j,x_j}\eta^{u_j}\bigr)\neq\{0\}.
\end{equation}
These are the Hom spaces in \Cref{prop geometric} with $W=F_x/A_x$.
For $i=1,2$, define $\mathcal I_i$ in the same way for the data
$(P_i,M_i,x_i,u_i,\sigma_i)$, and denote by
$p_{x_i,j}\in A_{x_i}$ the polynomials of
\eqref{eq boundary polynomial} for these data.

As explained in Section~\ref{ssec intertwining}, there exist nonzero denominators 
$b_1,b_2\in A_x$ denoted by $d$ there, such that
\begin{equation}\label{eq extremal operator denominators}
	\begin{aligned}
		\iota_{n_1}(b_1)M(n_1,\sigma)\bigl(I_{P,x}(\sigma)\bigr)
		&\subseteq I_{P_1,x_1}(\sigma_1),\\
		\iota_{n_2}^{-1}(b_2)M(n_2,\sigma_2)^{-1}
		\bigl(I_{P,x}(\sigma)\bigr)
		&\subseteq I_{P_2,x_2}(\sigma_2).
	\end{aligned}
\end{equation}

\begin{proposition}\label{prop explicit bounds}
	The indexing sets $(\mathcal{I}_i)_{i=0,1,2}$ being defined as above, put
	\begin{equation}\label{eq three denominator products} 
		\begin{aligned}
			d_0&=\prod_{j\in\mathcal I_0}p_{x,j},\\
			d_1&=b_1\,\iota_{n_1}^{-1}
			\left(\prod_{j\in\mathcal I_1}p_{x_1,j}\right),\\
			d_2&=b_2\,\iota_{n_2}
			\left(\prod_{j\in\mathcal I_2}p_{x_2,j}\right),
		\end{aligned}
	\end{equation}
	where an empty product is $1$. Then $d_0,d_1,d_2$ are nonzero elements
	of $A_x$, and, for $i=0,1,2$,
	\begin{equation}\label{eq three singularity bounds}
		d_iJ_P^G(u,\eta,\ell,\sigma)\bigl(I_{P,x}(\sigma)\bigr)
		\subseteq A_x.
	\end{equation}
	The bounds hold for any choices of $z_{x,j}$ and $z_{x_i,j}$
	($i=1,2$) satisfying \Cref{lm BD}.
\end{proposition}

\begin{proof}
	We apply \Cref{cor habile} to the adapted filtration starting at
	$\cV_x^0$. By \Cref{prop geometric,cor kill}, the Hom space for the
	$j$-th quotient is annihilated by $p_{x,j}$ when $j\in\mathcal I_0$
	and is zero otherwise, in which case we take the annihilator $1$.
	The support formula makes the intertwining period $A_x$-valued on $\cV_x^0$,
	which proves \eqref{eq three singularity bounds} for $i=0$. Applying the same argument at the two extremal vertices gives
	\[
	\left(\prod_{j\in\mathcal I_i}p_{x_i,j}\right)
	J_{P_i}^G(u_i,\eta,\ell_i,\sigma_i)
	\bigl(I_{P_i,x_i}(\sigma_i)\bigr)\subseteq A_{x_i},
	\qquad i=1,2.
	\]
	Combining these inclusions with
	\eqref{eq extremal operator denominators} and
	\eqref{eq from minimum}--\eqref{eq from maximum}, and applying
	$\iota_{n_1}^{-1}$ for $i=1$ and $\iota_{n_2}$ for $i=2$, proves the rest of the statement.
\end{proof}

\subsection{Intertwining periods and reduction modulo \texorpdfstring{$p'\neq p$}{p' unequal to p}}
\label{ssec reduction}

Assume that $R$ has characteristic $p'>0$, with $p'\neq p$.
Let $W(R)$ be its ring of Witt vectors, let $K$ be an algebraic closure
of $\operatorname{Frac}(W(R))$, and let $\mathfrak o$ be the integral
closure of $W(R)$ in $K$. The valuation of $W(R)$ extends uniquely to
$K$, with valuation ring $\mathfrak o$ and residue field $R$.
Write $r_{p'}:\mathfrak o\to R$ for reduction.
Choose $q^{1/2}\in\mathfrak o^\times$ reducing to the square root fixed
over $R$. We use quotient measures induced by Haar measures with values
in $\Z[1/p]$, normalized on compact open pro-$p$ subgroups, and their
reductions over $R$.

Fix $P=MV$ and an $M$-admissible $x=x_u$.
Let $\sigma$ be an integral finite-length smooth $K$-representation
of $M$, and fix an $M$-stable $\mathfrak o$-lattice $\sigma_e$ in the
sense of \cite[I.9.1]{V}. Put
\[
\bar\sigma:=\sigma_e\otimes_{\mathfrak o}R.
\]
This representation has finite length, by the argument of
\cite[II.5.11(a)]{V}.
Let $\eta:H\to\mathfrak o^\times$ be a smooth character trivial on
unipotent elements, and assume that
\[
\ell\in\Hom_{\mathcal H_K[M_x]}
\bigl(\sigma,K\otimes\delta_{P,x}\eta^u\bigr),
\qquad \ell(\sigma_e)\subseteq\mathfrak o.
\]
Denote their reductions by $\bar\eta:H\to R^\times$ and $\bar\ell:\overline{\sigma}\to R$.

Put
\[
A_{x,e}:=\mathfrak o[\Lambda_x],\qquad
A_{x,K}:=K[\Lambda_x],\qquad
F_{x,K}:=\operatorname{Frac}(A_{x,K}),
\]
and $A_x=R[\Lambda_x]$ and $F_x=\operatorname{Frac}(A_x)$.
Coefficientwise reduction defines
\[
B_x:=\left\{\frac{a}{b}:a,b\in A_{x,e},\ \bar b\neq0\right\}
\subseteq F_{x,K},
\qquad
r_{p'}:B_x\longrightarrow F_x,\quad
\frac{a}{b}\longmapsto\frac{\bar a}{\bar b}.
\]
We use the same notation $\Psi_x$ over each coefficient ring
and set
\[
I_e:=I_P^G(\Psi_x\otimes_{\mathfrak o}\sigma_e).
\]
Parabolic induction commutes with scalar extension, giving
\[
I_e \otimes_{\mathfrak o}K \simeq I_{P,x}(\sigma),\qquad
I_e\otimes_{\mathfrak o}R\simeq I_{P,x}(\bar\sigma).
\]
We write $\bar\phi$ for the reduction of $\phi\in I_e$.

\begin{proposition}\label{prop reduction LIP}
The polynomials $p_{x,j}$ of \eqref{eq boundary polynomial}, formed
from the central characters of the Jacquet modules of $\sigma$ over
$K$, belong to $A_{x,e}$ and have nonzero reductions. Put
\[
d_x:=\prod_{\cO_j\subseteq\partial\cO_x}p_{x,j}.
\]
Then
\begin{equation}\label{eq integral period denominator}
d_xJ_P^G(u,\eta,\ell,\sigma)(I_e)\subseteq A_{x,e}.
\end{equation}
In particular, for every $\phi\in I_e$, the value
$J_P^G(u,\eta,\ell,\sigma)(\phi)$ belongs to $B_x$, and
\begin{equation}\label{eq reduction LIP}
r_{p'}\bigl(J_P^G(u,\eta,\ell,\sigma)(\phi)\bigr)
=J_P^G(u,\bar\eta,\bar\ell,\bar\sigma)(\bar\phi).
\end{equation}
\end{proposition}

\begin{proof}
Jacquet functors preserve stable lattices
\cite[II.4.14(d)]{V}. Intersect a composition series of
$r_{L_j,M}(\sigma)$ with $r_{L_j,M}(\sigma_e)$.
The successive quotients are stable lattices in the corresponding
irreducible subquotients, whose central characters therefore take
values in $\mathfrak o^\times$; see \cite[I.9.3 and II.4.11]{V}.
Consequently every factor of $p_{x,j}$ has the form
$a-b\Psi_x(z_{x,j})$ with $a,b\in\mathfrak o^\times$.
By \Cref{lm BD}, the image of $z_{x,j}$ in $\Lambda_x$ is nonzero,
so each of these factors has nonzero reduction. Thus
$p_{x,j}\in A_{x,e}$ and $\bar p_{x,j}\neq0$.

Apply \Cref{lm classique} to these lattice filtrations over $A_{x,e}$.
Together with \Cref{prop geometric}, this gives the annihilation of
$\Hom$ required in \Cref{cor habile} for the adapted filtration of
$I_e$ and the target $(F_{x,K}/A_{x,e})\otimes\eta$.
On the initial term $\cV_x^0\cap I_e$, the period is
$A_{x,e}$-valued by the support formula, since $\ell(\sigma_e)\subseteq\mathfrak o$ and our quotient
integrals preserve integral coefficients.
\Cref{cor habile} therefore gives \eqref{eq integral period denominator}.

Since $\bar d_x\neq0$, the period is $B_x$-valued on $I_e$. Reduction induces an $H$-equivariant $A_x$-linear map
\[
\overline J:I_{P,x}(\bar\sigma)\longrightarrow F_x\otimes\bar\eta,
\qquad
\overline J(\bar\phi)
=r_{p'}\bigl(J_P^G(u,\eta,\ell,\sigma)(\phi)\bigr)
\quad(\phi\in I_e).
\]
We extend it $F_x$-linearly.

Let $\bar\phi\in I_{P,x}(\bar\sigma)$ be such that
$h\mapsto\bar\phi(uh)$ has compact support modulo
$H\cap u^{-1}Pu$. Since parabolic induction commutes with reduction,
choose a lift $\phi_0\in I_e$ of $\bar\phi$.
The support $C$ of $\bar\phi$ in $P\backslash G$ is compact and open.
Define
\[
\phi(g):=\mathbf1_C(Pg)\phi_0(g).
\]
Then $\phi\in I_e$ still reduces to $\bar\phi$ and has exactly the
same support modulo $P$. In particular, $\phi\in\cW(u)$.
The support formula therefore gives
\[
\begin{aligned}
	\overline J(\bar\phi)
	&=r_{p'}\bigl(J_P^G(u,\eta,\ell,\sigma)(\phi)\bigr)\\
	&=r_{p'}\bigl(j_{\eta,\ell}(\phi)\bigr)
	=j_{\bar\eta,\bar\ell}(\bar\phi).
\end{aligned}
\]

Finally, every section over $F_x$ satisfying the same support
condition becomes a section of $I_{P,x}(\bar\sigma)$ after
multiplication by a nonzero element of $A_x$, without changing
its support. By $F_x$-linearity, $\overline J$ therefore satisfies
the support characterization of
$J_P^G(u,\bar\eta,\bar\ell,\bar\sigma)$.
The uniqueness in \Cref{thm LIP} proves \eqref{eq reduction LIP}.

\end{proof}

The next corollary may be useful in the study of modular
distinction.

For $i=1,2$, let
\[
(P_i=M_iV_i,u_i,x_i=x_{u_i},\sigma_i,\sigma_{i,e},\ell_i)
\]
be integral data satisfying the hypotheses of
\Cref{prop reduction LIP}, with the same character $\eta$.
Assume that
\[
(M_1,x_1)\stackrel n\searrow(M_2,x_2)
\]
is an edge, and that
\[
\sigma_2={}^n\sigma_1,\qquad
\sigma_{2,e}={}^n\sigma_{1,e}.
\]

Put
\[
\begin{gathered}
	\Lambda_1:=M_1/M_1^1,\qquad
	\mathcal A_e:=\mathfrak o[\Lambda_1],\\
	\mathcal A_K:=K[\Lambda_1],\qquad
	\mathcal F_K:=\operatorname{Frac}(\mathcal A_K),
\end{gathered}
\]
and let $\chi_1:M_1\to\mathcal A_e^\times$ be the universal
unramified character. The corresponding universal induced lattices are
\[
\begin{aligned}
	\mathcal I_{1,e}
	&:=I_{P_1}^G(\chi_1\otimes_{\mathfrak o}\sigma_{1,e}),\\
	\mathcal I_{2,e}
	&:=I_{P_2}^G\bigl({}^n
	(\chi_1\otimes_{\mathfrak o}\sigma_{1,e})\bigr).
\end{aligned}
\]
We first consider the standard intertwining operator 
\[
	M(n,\sigma_1): I_{P_1}^G(\chi_1\otimes_K\sigma_1)
	\otimes_{\mathcal A_K}\mathcal F_K\\
 \longrightarrow
	I_{P_2}^G\bigl({}^n(\chi_1\otimes_K\sigma_1)\bigr)
	\otimes_{\mathcal A_K}\mathcal F_K.
\]

We recall that the isomorphism
\[
\iota_n:\Lambda_{x_1}\simeq \Lambda_{x_2},
\qquad [m]\longmapsto[nmn^{-1}],
\]
induces
\[	\iota_n :A_{x_1,e}\simeq A_{x_2,e},\quad 
	\iota_n :B_{x_1}\simeq B_{x_2},\quad 
	\iota_n :F_{x_1,K}\simeq F_{x_2,K}.\]
On parameters, this corresponds to $\chi\mapsto{}^n\chi$,
where ${}^n\chi(m_2)=\chi(n^{-1}m_2n)$.
We denote the corresponding isomorphisms over $R$ by the same symbol again:
\[
 \iota_n:A_{x_1}\xrightarrow{\sim}A_{x_2},
\qquad
 \iota_n:F_{x_1}\xrightarrow{\sim}F_{x_2}, 
\]
hence 
\[
r_{p'}\circ\iota_n=\iota_n\circ r_{p'}
\qquad\text{on }B_{x_1}.
\]

Let 
\[
\rho_1:\mathcal A_e\longrightarrow A_{x_1,e},
\qquad
\rho_1(\chi_1(m))=\Psi_{x_1}(m)
\]
be the restriction map on parameter rings, and put
\[
I_{i,e}:=
I_{P_i}^G(\Psi_{x_i}\otimes_{\mathfrak o}\sigma_{i,e}),
\qquad i=1,2.
\]
As explained in \Cref{ssec FE}, the restriction of the universal
operator to $I_{P_1,x_1}(\sigma_1)\otimes_{A_{x_1,K}}F_{x_1,K} $ is
well defined and gives
\[
	M(n,\sigma_1):
	 I_{P_1,x_1}(\sigma_1)\otimes_{A_{x_1,K}}F_{x_1,K} 
	 \longrightarrow
	I_{P_2,x_2}(\sigma_2)\otimes_{A_{x_2,K}}F_{x_2,K}.
\] 

We impose two additional requirements. First, this
operator restricts to 
\begin{equation}\label{eq 1 de plus}
	M(n,\sigma_1): 
 I_{1,e}\otimes_{A_{x_1,e}}B_{x_1} 
 \longrightarrow
	I_{2,e}\otimes_{A_{x_2,e}}B_{x_2}.
\end{equation}
This map is $\iota_n$-semilinear:
\[
M(n,\sigma_1)(a\phi)
=\iota_n(a)M(n,\sigma_1)(\phi)
\]
for $a\in B_{x_1}$ and $\phi$ in its source.
Second, its reduction modulo $p'$ is the standard
intertwining operator for the reduced data, formed with
the reduced quotient measure:
\[
	\overline{M(n,\sigma_1)}=M(n,\bar\sigma_1):\quad
	I_{P_1,x_1}(\bar\sigma_1)\otimes_{A_{x_1}}F_{x_1}\\
\longrightarrow
	I_{P_2,x_2}(\bar\sigma_2)\otimes_{A_{x_2}}F_{x_2}.
\]
The reduced map is again $\iota_n$-semilinear.

For later applications, a sufficient condition for these
two requirements is the existence of $D\in\mathcal A_e$ such that

\begin{enumerate}
	\item the image $d:=\rho_1(D)$ satisfies that its reduction modulo $p'$ is nonzero: $\bar d\neq0$,
	\item the operator $M(n,\sigma_1)$ restricts to
	\[M(n,\sigma_1):\quad
		 \mathcal I_{1,e}
		\otimes_{\mathcal A_e}\mathcal A_e[D^{-1}]\\
	 \longrightarrow
		\mathcal I_{2,e}
		\otimes_{\mathcal A_e}\mathcal A_e[D^{-1}].\]
\end{enumerate}

Indeed, $d$ and $\iota_n(d)$ are then invertible in
$B_{x_1}$ and $B_{x_2}$, respectively. Base change of this
universal integral model therefore gives the map in \Cref{eq 1 de plus} above.

To check the second requirement, put
\[
\mathcal A:=R[\Lambda_1],\qquad
\overline{\mathcal I}_i:=
\mathcal I_{i,e}\otimes_{\mathfrak o}R,
\qquad i=1,2,
\]
and let $\bar\rho_1:\mathcal A\to A_{x_1}$ be the reduction
of $\rho_1$. Since
$\bar\rho_1(\bar D)=\bar d\neq0$, we also have
$\bar D\neq0$, and restriction and reduction extend to
the relevant localizations, giving a commutative square
\[
\begin{tikzcd}
	\mathcal A_e[D^{-1}]
	\arrow[r,"\rho_1"]
	\arrow[d,"r_{p'}"']
	&
	A_{x_1,e}[d^{-1}]
	\arrow[d,"r_{p'}"]
	\\
	\mathcal A[\bar D^{-1}]
	\arrow[r,"\bar\rho_1"']
	&
	A_{x_1}[\bar d^{-1}].
\end{tikzcd}
\]
By \cite[Proposition~3.13(ii)]{MossTrias}, reduction of the
universal integral operator is the standard intertwining
operator formed with the reduced quotient measure:
\[
	\overline{M(n,\sigma_1)}=M(n,\bar\sigma_1):\quad
	\overline{\mathcal I}_1
	\otimes_{\mathcal A}\mathcal A[\bar D^{-1}]\\
	\longrightarrow
	\overline{\mathcal I}_2
	\otimes_{\mathcal A}\mathcal A[\bar D^{-1}].
\]
The second requirement now follows from the above
commutative diagram.

For the next corollary, write
\[
\begin{aligned}
	J_i&:=J_{P_i}^G(u_i,\eta,\ell_i,\sigma_i):I_{i,e}\otimes_{A_{x_i,e}}F_{x_i,K}
	\longrightarrow F_{x_i,K},\\
	\bar J_i&:=J_{P_i}^G
	(u_i,\bar\eta,\bar\ell_i,\bar\sigma_i):I_{P_i,x_i}(\bar\sigma_i)\otimes_{A_{x_i}}F_{x_i}
	\longrightarrow F_{x_i},
\end{aligned}
\qquad i=1,2.
\]

\begin{corollary}
	In the preceding setting, suppose that the periods over $K$
	satisfy a functional equation
	\[
	J_2\circ M(n,\sigma_1)=\iota_n\circ(cJ_1),
	\qquad c\in F_{x_1,K}.
	\]
	If $\bar\ell_1\neq0$, then $c\in B_{x_1}$, and
	\[
	\bar J_2\circ M(n,\bar\sigma_1)
	=
	\iota_n\circ\bigl(r_{p'}(c)\bar J_1\bigr).
	\]
	The same conclusion holds without assuming $\bar\ell_1\neq0$
	if $c\in B_{x_1}$ is already known.
\end{corollary}

\begin{proof}
	By \Cref{prop reduction LIP}, each $J_i$ is $B_{x_i}$-valued
	on $I_{i,e}$ and reduces to $\bar J_i$.
	We extend these forms $B_{x_i}$-linearly to
	$B_{x_i}\otimes_{A_{x_i,e}}I_{i,e}$. If $\bar\ell_1\neq0$, then $\bar J_1\neq0$ by \Cref{thm LIP}.
	Choose $\phi\in I_{1,e}$ such that
	$\bar J_1(\bar\phi)\neq0$.
	Then $J_1(\phi)$ has nonzero reduction and hence is a unit
	of $B_{x_1}$.
	The functional equation gives
	\[
	c=
	\frac{\iota_n^{-1}
		\bigl(J_2(M(n,\sigma_1)\phi)\bigr)}
	{J_1(\phi)}
	\in B_{x_1}.
	\]
	Applying $r_{p'}$ to the functional equation on $I_{1,e}$,
	using its compatibility with $\iota_n$, and then extending
	scalars to $F_{x_1}$ proves the asserted identity.
\end{proof}

\section{Application to distinction and examples}\label{sec application}

\subsection{Application to distinction}
Throughout this section, $F$ is the non-archimedean local field of \Cref{sec pairs}, $R$ is an algebraically closed field of characteristic different from $p$, and $\eta:H\to R^\times$ is a smooth character which is trivial on unipotent elements. If $K$ is a closed subgroup of $G$ and $\xi:K\to R^\times$ is a smooth character, a smooth $R$-representation $\pi$ of $G$ is said to be \emph{$(K,\xi)$-distinguished} if
\[\Hom_{\mathcal H_R[K]}(\pi,R\otimes\xi)\neq\{0\}.\]
The following theorem and its corollary are the algebraic, $\eta$-twisted analogues of \cite[Theorem 5.4 and Corollary 5.5]{MOYjfa}.

\begin{theorem}\label{thm distinction}
Let $\sigma$ be a finite-length smooth $R$-representation of $M$. Assume that there exists $u\in G$ such that $x:=x_u$ is $M$-admissible and
\[\Hom_{\mathcal H_R[M_x]}\big(\sigma,R\otimes\delta_{P,x}\eta^u\big)\neq\{0\}.\]
Then, for every $\chi\in\cX_x(R)$,
\[\Hom_{\mathcal H_R[H]}\big(I_P^G(\chi\otimes\sigma),R\otimes\eta\big)\neq\{0\}.\]
\end{theorem}

\begin{proof}
	Choose
	\[
	0\neq\ell_0\in
	\Hom_{\mathcal H_R[M_x]}
	\bigl(\sigma,R\otimes\delta_{P,x}\eta^u\bigr).
	\]
	The universal character $\Psi_x$ is trivial on $M_x$: the image of
	any $m\in M_x$ in $\Lambda_x$ is fixed by $\theta_x$, which acts
	by $-1$, and $\Lambda_x$ is torsion-free.
	Extending $\ell_0$ to $A_x$ and applying \Cref{thm LIP}, we obtain,
	after clearing denominators, a nonzero form
	\[
	\widetilde J\in
	\Hom_{\mathcal H_{A_x}[H]}
	\bigl(I_{P,x}(\sigma),A_x\otimes\eta\bigr).
	\]
	
	Fix $\chi\in\cX_x(R)$ and put
	$\mathfrak m_\chi=\ker(\operatorname{ev}_\chi:A_x\to R)$.
	Since $A_x$ is a Noetherian domain, the Krull intersection theorem gives
	\[
	\bigcap_{r\geq0}\mathfrak m_\chi^r=\{0\}.
	\]
	There is therefore an integer $r\geq0$ such that the image of
	$\widetilde J$ is contained in $\mathfrak m_\chi^r\otimes_R\eta$
	but not in $\mathfrak m_\chi^{r+1}\otimes_R\eta$.
	By $A_x$-linearity, $\widetilde J$ induces a nonzero
	$H$-equivariant map
	\[
	I_{P,x}(\sigma)/\mathfrak m_\chi I_{P,x}(\sigma)
	\longrightarrow
	(\mathfrak m_\chi^r/\mathfrak m_\chi^{r+1})\otimes_R\eta.
	\]
	Specialization at $\chi$ commutes with normalized parabolic induction,
	so the source identifies canonically with $I_P^G(\chi\otimes\sigma)$.
	Composing with an $R$-linear form on
	$\mathfrak m_\chi^r/\mathfrak m_\chi^{r+1}$ which does not vanish
	on the image gives the required nonzero $(H,\eta)$-equivariant form.
\end{proof}

\begin{corollary}\label{cor cuspidal distinction}
Let $\sigma$ be a finite-length cuspidal smooth $R$-representation of $M$. Then $I_P^G(\sigma)$ is $(H,\eta)$-distinguished if and only if there exists $u\in G$ such that $x_u$ is $M$-admissible and $\sigma$ is $(M_{x_u},\delta_{P,x_u}\eta^u)$-distinguished.
\end{corollary}
\begin{proof}
The ``if'' implication is \Cref{thm distinction} with $\chi=1$. Conversely, let
\[0\neq L\in\Hom_{\mathcal H_R[H]}\big(I_P^G(\sigma),R\otimes\eta\big).\]
Apply the filtration by the double cosets $Pu_iH$ of \Cref{ssec geometric} to the ordinary induced representation $I_P^G(\sigma)$, and choose the smallest $i$ such that the restriction of $L$ to $\cV_i$ is nonzero. It induces a nonzero form on $\cV_i/\cV_{i-1}$. The calculation of \Cref{prop geometric}, with the trivial unramified character, gives
\[\Hom_{\mathcal H_R[(L_i)_{x_i}]}\big(r_{L_i,M}(\sigma),R\otimes\delta_{Q_i,x_i}\eta^{u_i}\big)\neq\{0\}.\]
By cuspidality, $r_{L_i,M}(\sigma)=0$ unless $L_i=M$. Hence $L_i=M$; equivalently, $x_i=x_{u_i}$ is $M$-admissible, and then $Q_i=P$. The displayed nonvanishing is therefore precisely
\[\Hom_{\mathcal H_R[M_{x_i}]}\big(\sigma,R\otimes\delta_{P,x_i}\eta^{u_i}\big)\neq\{0\},\]
which proves the converse.
\end{proof}

\subsection{Explicit bounds for the Galois model}
\label{ssec galois speh}

We work over $R=\C$ and assume that $F$ has characteristic zero.
We apply the denominator argument of \Cref{thm LIP},
using semisimplicity of the Jacquet modules to sharpen
the contribution of each boundary orbit. Let $E/F$ be a quadratic extension and let $D$ be a central division
$F$-algebra of odd index $d$. Then $\Delta=D\otimes_F E$ is a
division algebra. For $r\geq1$, put
\[
G_r=\GL_r(\Delta),\qquad H_r=\GL_r(D),\qquad
\nu=|\operatorname{Nrd}_{\Delta/E}|_E.
\]
We denote by $\theta$ the involution induced by the nontrivial
element of $\operatorname{Gal}(E/F)$, and write
$\tau^*=(\tau^\theta)^\vee$ and $\tau[s]=\tau\nu^s$.
Let $\eta_{E/F}$ be the quadratic character of $F^\times$
attached to $E/F$, and put
$\eta_r=\eta_{E/F}\circ\operatorname{Nrd}_{D/F}$ on $H_r$.

Fix an irreducible unitary cuspidal representation $\rho$ of $G_t$.
Write $s_\rho$ for its positive reducibility exponent and put
$\nu_\rho=\nu^{s_\rho}$. In the notation of
\cite[Section~4.1]{MatJFA}, $s_\rho=l_\rho$ divides $d$;
in particular, $s_\rho$ is odd. For $k,l\geq1$, let
\[
\delta=\operatorname{St}_l(\rho),\qquad
\pi=\operatorname{Speh}(k,\delta),\qquad n=tkl.
\]
Thus $\delta$ is the irreducible quotient of \[
\nu_\rho^{(1-l)/2} \rho\times
\nu_\rho^{(3-l)/2}\rho\times\cdots\times
\nu_\rho^{(l-1)/2}\rho,
\] and $\pi$ is the irreducible quotient of
\[
\nu_\rho^{(k-1)/2}\delta\times
\nu_\rho^{(k-3)/2}\delta\times\cdots\times
\nu_\rho^{-(k-1)/2}\delta,
\]
where $\times$ denotes normalized parabolic induction, as in
\cite[Section~10.6]{MatJFA}.

Set $G=G_{2n}$, $H=H_{2n}$ and let $P=MU$ be the standard
parabolic of type $(n,n)$. Choose $\xi\in E^\times$ with
$\theta(\xi)=-\xi$ and put
\[
u_0=\begin{pmatrix}I_n&-\xi I_n\\ I_n&\xi I_n\end{pmatrix},
\qquad x_0=u_0\theta(u_0)^{-1}
=\begin{pmatrix}0&I_n\\ I_n&0\end{pmatrix}.
\]
The double coset $Pu_0H$ is open and
\[
M_{x_0}=\{\operatorname{diag}(g,\theta(g)):g\in G_n\}.
\]
The natural pairing of $\pi^*\boxtimes\pi$ gives an
$M_{x_0}$-invariant form $\ell$. We write $J_\pi(s)$ for
the restriction of $J_P^G(u_0,1,\ell,\pi^*\boxtimes\pi)$
to the family
\[
I(s)=\pi^*[s]\times\pi[-s].
\]

Let $q_E$ be the residue cardinality of $E$ and set
\[
T=q_E^{-s},\ B=\C[T,T^{-1}],\  U=T-1, 
A=B_{(U)},\ K=\C(T),\ Q_A=K/A.
\]
If $v_E$ is the normalized valuation of $E$, define
\[
\Psi_T(g_1,g_2)
=T^{v_E(\operatorname{Nrd}_{\Delta/E}(g_1))
	-v_E(\operatorname{Nrd}_{\Delta/E}(g_2))},
\qquad
\mathscr I_A=I_P^G\bigl(\Psi_T\otimes_\C
(\pi^*\boxtimes\pi)\bigr).
\]
Let
$
J_\pi:\mathscr I_A\otimes_A K\longrightarrow K
$
be the $H$-invariant $K$-linear form obtained by transporting
$J_P^G(u_0,1,\ell,\pi^*\boxtimes\pi)$ along the isomorphism
$A_{x_0}\simeq B$ sending the universal character
$\Psi_{x_0}$ to $\Psi_T$.
For $f\in\mathscr I_A\otimes_A K$, let \[f_s\in I(s)=I_P^G\bigl(\pi^*[s]\boxtimes\pi[-s]\bigr)\]
be its specialization at $T=q_E^{-s}$, defined whenever
$f$ is regular at that point, and in general for $g\in G$ we think of $f_s$ as a meromorphic section. 
Then
\[
J_\pi(s)(f_s)
=\left.J_\pi(f)\right|_{T=q_E^{-s}}
\]
as
an identity of meromorphic functions of $s$.  The order of the pole of $J_\pi(s)$
at zero is the smallest integer $r\geq0$ such that
$U^rJ_\pi(\mathscr I_A)\subseteq A$.

We first record the Jacquet-module calculation used in the proof.
For a multisegment $\mathfrak m$ on the $\rho$-line, we write $L(\mathfrak m)$ for the corresponding irreducible representation in the Langlands classification (see \cite{Zel}). For $(a,b)$ a partition of $n$, we use
$r_{(a,b)}$ for the normalized Jacquet functor along the standard subgroup of $G_n$ of type $(a,b)$. The following is a consequence of \cite[Theorem 2.1]{KL}, valid in our setting, see \cite[Section 10]{MatJFA}.  

\begin{proposition}\label{propproposition galois speh cuts}
Put
\[
b_i=\frac{k+l}{2}-i,\qquad a_i=b_i-l+1,
\qquad 1\leq i\leq k.
\]
For every integer $0\leq a\leq n$, we have
\begin{equation}\label{eq galois cut summand}
	r_{(a,n-a)}(\pi)\simeq
	\bigoplus_{\substack{
			(c_1,\ldots,c_k)\in\mathbf Z^k\\
			0\leq c_1\leq\cdots\leq c_k\leq l\\
			t(c_1+\cdots+c_k)=a}}
	L\bigl(([b_i-c_i+1,b_i]_\rho)_{i=1}^k\bigr)
	\boxtimes
	L\bigl(([a_i,b_i-c_i]_\rho)_{i=1}^k\bigr).
\end{equation}
\end{proposition}

In the above statement we used the following conventions: empty segments are omitted, $L(\varnothing)$ denotes
	the trivial representation of $G_0$, and an empty direct sum is understood to be zero, so
	$r_{(a,n-a)}(\pi)=0$ whenever $t$ does not divide $a$. We refer to the proof of \cite[Theorem 10.4]{MatJFA} for the argument underlying the proof of the following corollary.

\begin{corollary}\label{cor jacquet conjdual}
With notations as in Proposition \ref{propproposition galois speh cuts}, suppose now that $a>0$ and that a summand
$\alpha\boxtimes\beta$ in \eqref{eq galois cut summand}
satisfies $\alpha^*\simeq\alpha$.
Then $\rho^*\simeq\rho$, and its indexing tuple is
\[
(c_1,\ldots,c_k)
=
(\underbrace{0,\ldots,0}_{p},
\underbrace{l-p,\ldots,l-p}_{k-p})
\]
for some $0\leq p<\min(k,l)$. Consequently,
\begin{equation}\label{eq galois smaller rectangles}
	\alpha\simeq\pi_p
	:=\operatorname{Speh}
	(k-p,\operatorname{St}_{l-p}(\rho)),
	\qquad
	a=a_p:=t(k-p)(l-p).
\end{equation}	
\end{corollary}

We recall that by 
\cite[Corollary~10.1]{MatJFA}, every irreducible
$H_r$-distinguished representation of $G_r$ is conjugate self-dual.
Suppose that $\rho^*\simeq\rho$. By
\cite[Proposition~5.2(2)]{MatJFA}, there is a unique
$\epsilon\in\{0,1\}$ such that
\[
\operatorname{Hom}_{H_t}(\rho,\eta_t^\epsilon)\ne0.
\]
Furthermore, \cite[Proposition~10.12, Proposition~10.13 and Theorem~10.3]{MatJFA}
give, for every $h,j\geq1$,
\begin{equation}\label{eq galois speh distinction}
	\operatorname{Speh}(h,\operatorname{St}_j(\rho))
	\text{ is }H_{thj}\text{-distinguished}
	\quad\Longleftrightarrow\quad
	j-1\equiv\epsilon\pmod2.
\end{equation}

We can now state the result that we had in mind.

\begin{theorem}\label{prop galois speh bound}
Let $m= \min(k,l)$. If $\rho^*\not\simeq\rho$, then $J_\pi(s)$ is holomorphic
	at $s=0$. If $\rho^*\simeq\rho$, then
	\begin{equation}\label{eq galois local denominator}
		U^{\lceil m/2\rceil}J_\pi(\mathscr I_A)\subseteq A, \ \mathrm{i.e.} \ \operatorname{ord}^{\mathrm{pole}}_{s=0}J_\pi(s)
		\leq\left\lceil\frac{m}{2}\right\rceil.
	\end{equation}
\end{theorem}

\begin{proof}
	By \cite[Proposition~5.3]{MatJFA}, the double cosets in
	$P\backslash G/H$ are indexed by the symmetric matrices
	\[
	\begin{pmatrix}a&n-a\\ n-a&a\end{pmatrix},
	\qquad 0\leq a\leq n.
	\]
	The index $a=0$ corresponds to the open orbit $Pu_0H$.
	For $a>0$, put $b=n-a$. We choose the orbit representative $u_a$ so that
	\[
	x_a=u_a\theta(u_a)^{-1}
	=
	\begin{pmatrix}
		I_a&0&0&0\\
		0&0&I_b&0\\
		0&I_b&0&0\\
		0&0&0&I_a
	\end{pmatrix}.
	\] Set 
	\[
	L_a=G_a\times G_b\times G_b\times G_a,
	\qquad
	(L_a)_{x_a}
	=\{(h_1,g,\theta(g),h_2):h_1,h_2\in H_a,\ g\in G_b\},
	\]
	where factors of rank zero are omitted. Let $\mathscr V_a$ be the subquotient of the orbit filtration
	of $\mathscr I_A|_H$ associated with $Pu_aH$.
	Since the Galois modulus character $\delta_{Q_a,x_a}$ is
	trivial (\cite[lemma 2.5.1]{LR}), \Cref{prop geometric} gives
	\begin{equation}\label{eq galois orbit Hom}
		\mathcal E_a:=
		\operatorname{Hom}_{\mathcal H_A[H]}(\mathscr V_a,Q_A)
		\simeq
		\operatorname{Hom}_{\mathcal H_A[(L_a)_{x_a}]}
		\bigl(\Psi_T\otimes\sigma_a,Q_A\bigr),
	\end{equation}
	where
	\begin{equation}\label{eq galois orbit jacquet}
		\sigma_a=r_{(a,b)}(\pi^*)\boxtimes r_{(b,a)}(\pi).
	\end{equation}
    Choose a uniformizer $\varpi_F$ of $F$ and set
	\[
	z_a=(\varpi_F I_a,I_b,I_b,I_a)
	\in Z(L_a)\cap(L_a)_{x_a}.
	\]
	Writing $e=e(E/F)$, we have
	\begin{equation}\label{eq galois central monomial}
		\Psi_T(z_a)=T^{ead}.
	\end{equation}

	By \Cref{propproposition galois speh cuts}, applied also with
	$\rho^*$ in place of $\rho$, the representation $\sigma_a$
	is a finite direct sum of irreducible representations of $L_a$.
	Let $\upsilon$ be such a summand, with central character
	$\omega_\upsilon$, and put
	\[
	\mathcal E_{a,\upsilon}:=
	\operatorname{Hom}_{\mathcal H_A[(L_a)_{x_a}]}
	\bigl(\Psi_T\otimes\upsilon,Q_A\bigr).
	\]
	The element $z_a$ acts on the source by
	$\omega_\upsilon(z_a)T^{ead}$ and acts trivially on the target.
	Consequently,
	\[
	\bigl(1-\omega_\upsilon(z_a)T^{ead}\bigr)
	\mathcal E_{a,\upsilon}=0.
	\]
	If $\omega_\upsilon(z_a)\ne1$ then $1-\omega_\upsilon(z_a)T^{ead}\in A^\times$ 
	so $\mathcal E_{a,\upsilon}=\{0\}$, whereas if not $1-\omega_\upsilon(z_a)T^{ead}\in UA^\times$ hence $U\mathcal E_a=0$, so the equality 
	\[U\mathcal  E_{a,\upsilon}=0\] holds for all irreducible constituents $\upsilon$ of $\sigma_a$, hence \[U\mathcal  E_a=0.\] But since $Q_A[U]\simeq\C$,
	every element of $\mathcal E_a$ factors through an
	$(L_a)_{x_a}$-invariant linear form on
	$\sigma_a=(\Psi_T\otimes\sigma_a)/(U)$.
	Hence
	\[
	\operatorname{Hom}_{(L_a)_{x_a}}(\sigma_a,\C)=0
	\quad\Longrightarrow\quad
	\mathcal E_a=0.
	\]
	
	Set
	\[
	\mathcal S=
	\{a\in\{1,\ldots,n\}:
	\operatorname{Hom}_{(L_a)_{x_a}}(\sigma_a,\C)\ne0\}.
	\]
	The period $J_\pi$ is $A$-valued on the initial subspace of sections
	supported on the open orbit. We therefore apply
	\Cref{cor habile} to the orbit filtration, using the annihilator
	$U$ when $a\in\mathcal S$ and $1$ otherwise. This gives
	\begin{equation}\label{eq galois orbit count bound}
		U^{\#\mathcal S}J_\pi(\mathscr I_A)\subseteq A.
	\end{equation}
	This is the denominator argument in the proof of
	\Cref{thm LIP}, with the annihilators just computed. It remains to bound $\#\mathcal S$.
	If $a\in\mathcal S$, some summand
	$\alpha\boxtimes\beta$ of $r_{(a,b)}(\pi^*)$ has
	$\operatorname{Hom}_{H_a}(\alpha,\C)\ne0$: indeed, the first
	$H_a$-factor of $(L_a)_{x_a}$ acts only on $\alpha$.
	The distinction result recalled before the theorem gives
	$\alpha^*\simeq\alpha$.
	Applying Corollary \ref{cor jacquet conjdual} shows that
	$\rho^*\simeq\rho$. Hence, if $\rho^*\not\simeq\rho$,
	the set $\mathcal S$ is empty, and
	\eqref{eq galois orbit count bound} proves holomorphy at zero. Suppose now that $\rho^*\simeq\rho$. Then $\pi^*\simeq\pi$,
	and the same corollary shows that every $a\in\mathcal S$
	is of the form
	\[
	a=a_p=t(k-p)(l-p),\qquad 0\leq p<m,
	\]
	with the corresponding first factor
	$\alpha\simeq\pi_p=\operatorname{Speh}
	(k-p,\operatorname{St}_{l-p}(\rho))$ distinguished.
	The ranks $a_p$ are strictly decreasing as $p$ increases.
	By \eqref{eq galois speh distinction}, applied with
	$h=k-p$ and $j=l-p$, distinction of $\pi_p$ is equivalent to
	\[
	p\equiv l-1-\epsilon\pmod2.
	\]
	Only one parity for $p$ can therefore occur, and
	\[
	\#\mathcal S
	\leq\#\{p\in\Z:0\leq p<m,\quad
	p\equiv l-1-\epsilon\pmod2\}
	\leq\lceil m/2\rceil.
	\]
	Together with \eqref{eq galois orbit count bound}, this proves
	\eqref{eq galois local denominator}.
\end{proof}

\section{Analytic complements}\label{sec analytic}

In this section, $R=\C$, and $F$ is as before. We extend the convergence result of \cite[Theorem 5.3]{MOYjfa} to the general setting of this paper, where $(\bG^\theta)^\circ(F)\subseteq H\subseteq \bG^\theta(F)$, and much more importantly, we remove the unimodularity assumption and allow the same character twist $\eta$ in the intertwining period, that is already present in the algebraic setting. The first key result is to extend Lagier's bounds for $H$-invariant relative matrix coefficients to the $H$-equivariant setting. We do this in Section \ref{sec Lagier} below. The second main input is to bound the character $\eta$ on $H$ which is done in Section \ref{sec estimate}, and the key idea there is to express its absolute value on $(\bG^\theta)^\circ(F)$ using algebraic characters, and realize these as eigencharacters in algebraic representations of $\bG$. 

Let $\mathbb G'$ be a connected reductive group defined over $F$. We retain the split components $A_{G'}$ and the space  $\fa_{G'}$ defined in Section \ref{sec gen vanish}. We write
\[\fa_{G'}^*:=\operatorname{Hom}_{\R}(\fa_{G'},\R), \quad 
  \fa_{G',\C}^*:=\fa_{G'}^*\otimes_\R\C 
\]
The pairing between a space and its dual is denoted by
$\langle\ ,\ \rangle$. For an involution $\tau$ preserving ${G'}$, we use $(\fa_{G'}^*)^{\pm\tau}$ and   $(\fa_{{G'},\C}^*)^{\pm\tau}$ for the eigenspaces of $\tau$ in $\fa_{G'}^*$ and $\fa_{{G'},\C}^*$ respectively. If 
$X^*(\mathbb{G'})_F$ denotes the group of algebraic characters of $\mathbb{G'}$ defined over $F$, restriction
to $A_{G'}$ identifies $X^*(\mathbb{G'})_F\otimes_\Z\R$ with $\fa_{G'}^*$.

For $M$ a Levi subgroup of $G$, let $\bT_M$ be a maximal $F$-split torus of $\bM$, hence of $\bG$, with roots $R(T_M,M)$ in $\operatorname{Lie}(M)$ and $R(T_M,G)$ in $\operatorname{Lie}(G)$. We denote by $\langle\ , \ \rangle_M$ a $W(M,T_M)$-invariant inner product on $\fa_{T_M}$, and denote by $\mathrm{pr}_{\fa_M}:\fa_{T_M}\to \fa_{M}$ the orthogonal projection with respect to this inner product. We now recall the construction of the coroot $\alpha^\vee\in \fa_M$ corresponding to a root $\alpha\in R(A_M,G)$. Take $\alpha_0 \in  R(T_M,G)$ restricting to $\alpha$, and let $\alpha_0^\vee$ be the corresponding coroot. Then $\alpha^\vee$ is the unique element of $\R\cdot \mathrm{pr}_{\fa_M}(\alpha_0^\vee)$ such that $\langle \alpha , \alpha^\vee \rangle=2$. Normalized as above, the coroot $\alpha^\vee$ is independent of the choice of the lift $\alpha_0$ of $\alpha$. 

\begin{remark}
The assertion in \cite[Section~2.2]{MOYjfa} that the projected coroot is
independent of the lift requires a correction: its positive ray is
independent of the lift, but the vector itself need not be. Hence this issue does not affect any result of \cite{MOY}. For example,
take the standard Siegel Levi $M\simeq\mathrm{GL}_2(F)$ in
$G=\mathrm{Sp}_4(F)$. In the usual cocharacter coordinates,
$\fa_{T_M}=\R^2$ and $\fa_M=\{(x,x):x\in\R\}$, with projection
\[
 \mathrm{pr}_{\fa_M}(x,y)=\bigl((x+y)/2,(x+y)/2\bigr).
\]
Writing $e_1,e_2$ for the dual coordinate functions, the roots
$\beta=2e_1$ and $\gamma=e_1+e_2$ have the same restriction $\alpha$
to $\fa_M$, namely $\alpha(x,x)=2x$. However,
\[
 \begin{aligned}
  \beta^\vee&=(1,0),&
  \mathrm{pr}_{\fa_M}(\beta^\vee)&=(1/2,1/2),\\
  \gamma^\vee&=(1,1),&
  \mathrm{pr}_{\fa_M}(\gamma^\vee)&=(1,1).
 \end{aligned}
\]
The discrepancy comes from the different lengths of $\beta$ and $\gamma$:
restriction identifies a long root and a short root, while their coroots
have different scaling factors. The normalization above removes this
ambiguity.
\end{remark}

For a parabolic subgroup $P=MV$ of $G$, let $R(A_M,P)$ be the roots of $A_M$ in
$\operatorname{Lie}(V)$. We call $\nu\in\fa_M^*$ $P$-dominant, respectively strictly
$P$-dominant, if $\langle\nu,\alpha^\vee\rangle\geq0$, respectively $>0$,
for every $\alpha\in R(A_M,P)$.

For nonnegative functions, we write $f\preccurlyeq g$ if $f\leq Cg$ for
some constant independent of the variables on the indicated domain.
The notation $\preccurlyeq_v$ allows the constant $C$ to depend on the variable $v$. When the subscript is a subset $S$,
$\preccurlyeq_{S}$ means that the bound is uniform under an explicitly stated
restriction involving $S$; the constant $C$ may
depend on $S$.
We use algebraic norms in the sense of
\cite[Sections~18.1--18.2]{Kottwitz05}. We will tacitly use some of their properties proven there, after recalling them below.  Such norms take values
in $[1,\infty)$. For an affine $F$-variety $X$, their equivalence class
is represented by
\[
 \|x\|_X=\max\{1,|f_1(x)|_F,\ldots,|f_r(x)|_F\},
 \qquad x\in X(F),
\]
where $f_1,\ldots,f_r$ generate $F[X]$ as an $F$-algebra. Here two
norms are equivalent if each is bounded by a constant times a positive
power of the other. For a morphism $\phi:X\to Y$ of affine $F$-varieties, there
exists $d>0$ such that
$
\|\phi(\cdot )\|_Y\preccurlyeq\| \cdot \|_X^d.$
If moreover $\phi$ is a closed immersion, then the pullback
$\|\phi(\cdot)\|_Y$ is equivalent to $\|\cdot\|_X$, i.e. there exist $a,b>0$ such that
\[
\|\cdot \|_X^a\preccurlyeq\|\phi(\cdot )\|_Y
\preccurlyeq\|\cdot \|_X^b.
\]

\subsection{Lagier's estimate}\label{sec Lagier}

In \cite{Lagierthesis}, Lagier obtains a bound for $G^\theta$-invariant generalized coefficients, which is then used by Blanc and Delorme to prove convergence of the open intertwining period. Here we extend Lagier's result to our setting.  We first record Delorme's  constant-term formula (\cite{delconstantterm}) with a character
twist. Let $P=MV$ be $\theta$-split and put
$P^-=\theta(P)=MV^-$. Thus $M=P\cap\theta(P)$, and
\[
 H\cap P=H\cap M,\qquad \delta_P|_{H\cap M}=1.
\]
For $c>0$, put
\[
 A_M^-(c):=\{a\in A_M:|\alpha(a)|_F\leq c
                    \text{ for every }\alpha\in R(A_M,P)\},
 \qquad A_M^-:=A_M^-(1).
\]
For a compact open subgroup $K$, we write $V_\pi^K$ for the fixed vectors of
$\pi$ and set  $e_K:=\operatorname{vol}(K)^{-1}\mathbf1_K$. The action of $e_K$ on linear forms is given by 
\[
 (e_K\ell)(v)=\operatorname{vol}(K)^{-1}
                 \int_K\ell(\pi(k^{-1})v)\,dk.
\]

\begin{lemma}\label{lemma:invariance}
Let $K\subset G$ be compact open. There is a compact open subgroup
$K'\subseteq K$ such that, for every smooth representation $\pi$, every
$\ell\in\Hom_H(\pi,\eta)$, and every $v\in V_\pi^K$,
\[
 \ell(\pi(a)v)=(e_{K'}\ell)(\pi(a)v),\qquad a\in A_M^-.
\]
\end{lemma}
\begin{proof}
Shrink $K$ so that it has Iwahori factorization with respect to $(P,P^-)$ in the sense of \cite[Section 2.1 p.936]{delconstantterm}; in particular \[a^{-1}(K\cap V^-)a\subseteq K\cap V^-, \quad a(K\cap V) a^{-1}\subseteq K\cap V\, \quad a\in A_M^-.\] Since $P^-H$ is open, the set
\[
 (K\cap V^-)(K\cap M)(K\cap H\cap\ker\eta)
\]
contains a neighborhood of $1$. Choose $K'$ inside this set. If
$k'=v^-mh$ is such a factorization, then
$a^{-1}m^{-1}(v^-)^{-1}a\in K$ for $a\in A_M^-$, and hence
\[
 \ell(\pi(k'^{-1}a)v)
 =\eta(h)^{-1}\ell\bigl(\pi(a)
           \pi(a^{-1}m^{-1}(v^-)^{-1}a)v\bigr)
 =\ell(\pi(a)v).
\]
Averaging over $K'$ proves the assertion.
\end{proof}

Write $r_{M,G}(\pi)=\delta_P^{-1/2}\pi_V$ for the normalized Jacquet
module along $P$, and $\bar v \in r_{M,G}(V_\pi)$ for the image of $v\in V_\pi$.  

\begin{proposition}\label{prop:constantterm}
Let $\pi$ be a smooth representation of $G$ and let $\ell\in\Hom_H(\pi,\eta)$.
There is a unique linear form $r_{P^-}(\ell)$
on $r_{M,G}(\pi)$ such that, for every compact open subgroup $K$ of $G$, there exists
$0<c<1$ for which
\[
 \ell(\pi(a)v)=\delta_P(a)^{1/2}
 r_{P^-}(\ell)\bigl(r_{M,G}(\pi)(a)\bar v\bigr),
 \qquad v\in V_\pi^K,\quad a\in A_M^-(c).
\]
Moreover,
\[
 r_{P^-}(\ell)\in\Hom_{H\cap M}(r_{M,G}(\pi),\eta).
\]
\end{proposition}
\begin{proof}
We follow \cite[proof of Theorem~3.4]{delconstantterm}.
Put $(\rho,V_P)=r_{M,G}(\pi)$ and let $j_P:V_\pi\to V_P$ be the
quotient map. For $K$ with Iwahori factorization relative to $(P,P^-)$,
put $K_M=K\cap M$. The generalized Jacquet lemma gives
\[
 j_P(V_\pi^K)=V_P^{K_M};
\]
see \cite[(3.7)]{delconstantterm}. Since $a\in A_M$ is central in $M$,
$\rho(a)$ maps $V_P^{K_M}$ bijectively onto itself. Thus the formula in
the statement, for one sufficiently contracting $a$, determines a linear
form on $V_P^{K_M}$ uniquely. Choose $K'\subseteq K$ as in \Cref{lemma:invariance}. Then
$e_{K'}\ell\in\pi^\vee$. Let $j_{P^-}$ be the quotient map from
$\pi^\vee$ to its normalized Jacquet module along $P^-$, and define
\[
 \lambda_K(w)=\big\langle j_{P^-}(e_{K'}\ell),w\big\rangle_P,
 \qquad w\in V_P^{K_M},
\]
using the canonical Jacquet pairing $\big\langle \ ,\ \big\rangle_P$ (\cite[Section 2.2 p.936]{delconstantterm}). The asymptotic formula for smooth matrix coefficients  proved in 
\cite[(2.4) and Lemma~2.1]{delconstantterm}, together with
\Cref{lemma:invariance}, gives the existence of $c_K>0$ such that 
\[
 \ell(\pi(a)v)=\delta_P(a)^{1/2}
 \lambda_K\bigl(\rho(a)j_P(v)\bigr),
 \qquad v\in V_\pi^K,\quad a\in A_M^-(c_K).
\]
If $K_1\subseteq K$ also has Iwahori factorization, uniqueness gives
$\lambda_{K_1}|_{V_P^{K_M}}=\lambda_K$. These spaces exhaust $V_P$ as
$K$ shrinks, so the compatible linear forms above define $\lambda=r_{P^-}(\ell)$ on $V_P$.
Shrinking an arbitrary compact open subgroup to one with Iwahori
factorization proves the stated formula for every $K$. Finally, let $h\in H\cap M$. Since $ha=ah$ for any $a\in A_M$, $\delta_P(h)=1$ and
$\pi(h)V_\pi^K=V_\pi^{hKh^{-1}}$, the form
$\eta(h)^{-1}\lambda\circ\rho(h)$ satisfies the same formula as
$\lambda$. Uniqueness gives $\lambda\circ\rho(h)=\eta(h)\lambda$.
\end{proof}

Let $\{P_i,\ i\in\mathcal I\}$ be a set of representatives of the $H$-conjugacy classes of minimal $\theta$-split parabolic subgroups of $G$, it is finite by \cite{HW}. We put $M_i=P_i\cap \theta(P_i)$, denote by $A_i$ the maximal $\theta$-split subtorus of $A_{M_i}$, and set $A_i^-:=A_i\cap A_{M_i}^-$, where $A_{M_i}^-$ is defined with respect to $P_i$. By \cite{Polardecomp1} (see also \cite{Polardecomp2} when the residual characteristic of $F$ is not $2$), the group $G$ has the polar decomposition
\begin{equation}\label{eq analytic polar}
 G=\bigcup_{i\in\mathcal I}HA_i^-\Omega_G.
\end{equation}
Precisely \cite{Polardecomp1} give the polar decomposition for $G^\theta$, but it holds for $H$ since this latter is of finite index inside $G^\theta$.
Fix further an algebraic norm $\|\ \|$ on $G$ and write
\[
 \|Hg\|_{H\backslash G}:=\|\theta(g)^{-1}g\|.
\]
For a $\theta$-stable Levi subgroup $M$, we similarly use
\[
 \|(H\cap M)m\|_{(H\cap M)\backslash M}
 :=\|\theta(m)^{-1}m\|.
\]

\begin{proposition}\label{prop:lagierestimate}
Let $\pi$ be a smooth representation of finite length of $G$, and let
$\ell\in\Hom_H(\pi,\eta)$. There exists $r\geq0$ such that
\[
 |\ell(\pi(ak)v)|\preccurlyeq_v\|Ha\|_{H\backslash G}^{r},
 \qquad a\in A_i^-,\quad k\in\Omega_G,\quad i\in\mathcal I.
\]
\end{proposition}
\begin{proof}
By \cite[Corollary~I.4.4]{waldplancheral}, there exists $N\geq0$,
depending only on $\pi$, such that for every $v\in V_\pi$ and
$v^\vee\in\pi^\vee$,
\[
 \big|v^\vee(\pi(g)v)\big|
 \preccurlyeq_{v,v^\vee} \|g\|^N,\qquad g\in G.
\]
Fix $i\in\mathcal I$, $v\in  V_\pi$, and choose a compact open subgroup $K$ fixing $v$. Since $A_i^-\subseteq A_{M_i}^-$, \Cref{lemma:invariance}
provides $K_i'\subseteq K$ such that
\[
 \ell(\pi(a)v)=(e_{K_i'}\ell)(\pi(a)v),\qquad a\in A_i^-.
\]
The form $e_{K_i'}\ell$ belongs to $\pi^\vee$, so the preceding bound
applies with the same exponent $N$, independently of $K_i'$ and $v$.

For $a\in A_i$, we have $\theta(a)^{-1}a=a^2$. On a split torus,
the coordinate norm $\max_j\max(|t_j|_F,|t_j|_F^{-1})$ is squared by
$a\mapsto a^2$. Comparison of algebraic norms therefore gives
\[
 \|a\|\preccurlyeq\|a^2\|^{d_i}
 =\|Ha\|_{H\backslash G}^{d_i}
\]
for some $d_i>0$. Hence
$|\ell(\pi(a)v)|\preccurlyeq_v\|Ha\|_{H\backslash G}^{Nd_i}$.
There are finitely many $i$, so we may take $r=N\max_i d_i$.
Finally, smoothness and compactness imply that
$\{\pi(k)v:k\in\Omega_G\}$ is finite. Applying the same bound to
these vectors gives the required uniformity in $k$.
\end{proof}

We shall also use the following consequence of \Cref{prop:lagierestimate}:
there exists $N\geq0$, independent of $v$, such that
\begin{equation}\label{eq global relative growth}
 |\ell(\pi(g)v)|\preccurlyeq_v\|g\|^N,
 \qquad g\in G.
\end{equation}
Indeed, write $g=ha\omega$ in \eqref{eq analytic polar}. The identity
\[
 \theta(g)^{-1}g=\theta(\omega)^{-1}a^2\omega
\]
and compactness of $\Omega_G$ give $\|a^2\|\preccurlyeq\|g\|^b$
for some $b\geq0$. The comparison on split tori used above, together
with $h=g\omega^{-1}a^{-1}$ and regularity of multiplication and
inversion, then gives
\[
 \|a\|+\|h\|\preccurlyeq\|g\|^{b'}
\]
for some $b'\geq0$. Moreover, $|\eta(h)|\preccurlyeq\|h\|^e$ for
some $e\geq0$: on $(\mathbb G^\theta)^\circ(F)$, the absolute value
of a smooth character is a product of real powers of absolute values
of $F$-rational algebraic characters, and this subgroup has finite
index in $H$. Equivariance and \Cref{prop:lagierestimate} now give
\[
 |\ell(\pi(g)v)|
 =|\eta(h)|\,|\ell(\pi(a\omega)v)|
 \preccurlyeq_v\|h\|^e\|a^2\|^r
 \preccurlyeq_v\|g\|^{b'e+br},
\]
which proves \eqref{eq global relative growth}.

\subsection{Convergence in the $\theta$-split case}

Let $P=MV$ be a $\theta$-split parabolic subgroup of $G$, with $M=P\cap\theta(P)$.
Fix an algebraic norm $\|\ \|_M$ on $M$.
Choose a maximal compact subgroup $K$ adapted
to a minimal parabolic contained in $P$, so that $G=PK$.
For $m\in M$, define $H_M(m)\in\fa_M$ by
\[
 \langle\chi,H_M(m)\rangle=\log|\chi(m)|_F,
 \qquad\chi\in X^*(\mathbb M)_F,
\]
and extend it by $H_M(vmk)=H_M(m)$, for $v\in V$ and $k\in K$.
For $\lambda\in\fa_{M,\C}^*$, put
\[
 \chi_\lambda(m)=e^{\langle\lambda,H_M(m)\rangle},\qquad
 f_\lambda(g)=e^{\langle\lambda,H_M(g)\rangle}f(g).
\]
Thus $f_\lambda\in I_P^G(\chi_\lambda\otimes\sigma)$ is the flat section
with $f_\lambda|_K=f|_K$. Let $\rho_P\in\fa_M^*$ be the half-sum, with
multiplicities, of the roots of $A_M$ in $\operatorname{Lie}(V)$; then
$\delta_P(m)=e^{2\langle\rho_P,H_M(m)\rangle}$.

 For any involution $\tau$ preserving $M$ and $c>0$, set
\begin{equation}\label{eq analytic cone}
 \begin{aligned}
 \mathcal D_{M,\C}^{-\tau}(c):=\bigl\{\lambda \in(\fa_{M,\C}^*)^{-\tau}:\;&
 \Re \langle\lambda,\alpha^\vee\rangle>c\\[-2pt]
 &\text{for all }\alpha\in R(A_M,P)\text{ with }\tau(\alpha)<0\bigr\}.
 \end{aligned}
\end{equation}
The parabolic $P$ is understood. In particular when $P$ is $\theta$-split we have 
\[\mathcal D_{M,\C}^{-\theta}(c)= \bigl\{\lambda \in(\fa_{M,\C}^*)^{-\theta}:\; 
\Re \langle\lambda,\alpha^\vee\rangle>c\ 
\text{for all }\alpha\in R(A_M,P)\}.\] The following result generalizes \cite[Theorem 2.16 and Remark 2.17]{BD} and \cite[Lemma 3.2]{MOYjfa}.

\begin{theorem}\label{thm:lipopencase}
\begin{enumerate}
\item Let $\sigma$ be a smooth representation of finite length of $M$,
and let $\ell\in\Hom_{H\cap M}(\sigma,\eta)$. There exists $c>0$ such that
\[
 \int_{(H\cap M)\backslash H}\eta(h)^{-1}\ell(f_\lambda(h))\,dh
\]
converges absolutely for every $f\in I_P^G(\sigma)$ whenever
$\lambda\in\mathcal D_{M,\C}^{-\theta}(c)$.
\item More generally, let $\sigma$ be any smooth representation of $M$.
Suppose that $\zeta:V_\sigma\to\R_{\geq0}$ satisfies
\[
 \zeta(zv)=|z|\zeta(v),\qquad
 \zeta(\sigma(h)v)=|\eta(h)|\zeta(v)
 \quad(z\in\C,\ h\in H\cap M),
\]
and that, for some $r\geq0$ independent of $v$,
\[
 \zeta(\sigma(m)v)\preccurlyeq_v\|m\|_M^r,
 \qquad m\in M.
\]
Then, there exists $c>0$, such that the integral
\[
 \int_{(H\cap M)\backslash H}|\eta(h)|^{-1}
                                \zeta(f_\lambda(h))\,dh
\]
is finite for every $f$ whenever
$\lambda\in\mathcal D_{M,\C}^{-\theta}(c)$.
\end{enumerate}
\end{theorem}

The proof of \Cref{thm:lipopencase} is an easy consequence of \Cref{prop:continuity} below, as we now explain. The proof of \Cref{prop:continuity} is more involved and postponed to \Cref{sec proof of cont}. 

For $v\in V_\sigma$ and $g=v_0mh\in PH$, with $v_0\in V$, define
\begin{equation}\label{eq analytic epsilon vector}
 \epsilon_{\lambda,v}(g):=
 |\eta(h)|^{-1}e^{-\langle\Re\lambda,H_M(m)\rangle}
       \delta_P(m)^{1/2}\zeta(\sigma(m^{-1})v),
\end{equation}
and set $\epsilon_{\lambda,v}(g)=0$ outside $PH$.
This expression is independent
of the decomposition: two choices differ by an element of
$P\cap H=M\cap H$, on which $\delta_P$ and $\chi_\lambda$ are trivial,
and the remaining factors cancel by the equivariance of $\zeta$.

\begin{proposition}\label{prop:continuity}
There exists $c>0$ such that $\epsilon_{\lambda,v}$ is continuous on $G$
for every $v\in V_\sigma$ whenever
$\lambda\in\mathcal D_{M,\C}^{-\theta}(c)$.
\end{proposition}

We prove this proposition in Section \ref{sec proof of cont}. It implies the theorem by the argument of  \cite[(2.31)-(2.33), p.258]{BD}. 
\begin{proof}[Proof of \Cref{thm:lipopencase} assuming \Cref{prop:continuity}]
For~(ii), define
\begin{equation}\label{eq auxiliary function}
 \epsilon_{\lambda,f}(g):=\epsilon_{\lambda,f(g)}(g).
\end{equation}
With compatible quotient measures, we have as in \cite[(2.31)-(2.33), p.258]{BD}:
\[
 \int_{(H\cap M)\backslash H}|\eta(h)|^{-1}
                         \zeta(f_\lambda(h))\,dh
 =\int_{(K\cap P)\backslash K}\epsilon_{\lambda,f}(k)\,dk.
\] Since $f|_K$ is locally constant with finitely many values,
\Cref{prop:continuity} makes the right-hand integrand continuous.
The quotient is compact, proving~(2). For~(1), take $\zeta(v)=|\ell(v)|$ and apply
\eqref{eq global relative growth} to $(M,H\cap M,\theta|_M)$
to obtain the hypothesis in~(2).
\end{proof}

\subsection{Bounding the twisting character}\label{sec estimate}

We prove \Cref{prop:continuity} by obtaining appropriate asymptotics for the character factor $|\eta|$ in
\eqref{eq analytic epsilon vector}. We denote by $\mathbb H^\circ$ the neutral component of $\mathbb H$. 

\begin{proposition}\label{prop:controlHpart}
There exists $d\geq0$ such that, for every compact $C\subseteq G$,
\[
 \max\{|\eta(h)|,|\eta(h)|^{-1}\}
 \preccurlyeq_C\|m\|_M^d
\]
for every decomposition $g=vmh\in C$, with $v\in V$, $m\in M$
and $h\in H$. The exponent $d$ is independent of $C$.
\end{proposition}
\begin{proof} 
	First we prove that  $\max\{|\eta(h)|,|\eta(h)|^{-1}\}\preccurlyeq \|h\|_H^d$ for some $d\geq 0$. Since $H^\circ$ has finite index inside $H$ and translation by
	fixed coset representatives preserves algebraic norms up to
	polynomial equivalence, we may assume  for this step only, that $\mathbb H$ is connected. The majorization now follows from expressing both $|\eta|^{\pm 1}$ as finite products of characters of the form $|\mu_i|_F^{a_i}$, with $a_i\in \R$ and $\mu_i$  algebraic $F$-characters of $H$. Hence it is now enough to prove that $\|h\|_H \preccurlyeq_C\|m\|_M^b$ for some non negative $b$, under the hypothesis of the statement.

	To prove this we consider the morphism
	\[
	j:\theta(\mathbb V)\times\mathbb M\times \mathbb V
	\longrightarrow\mathbb G\times\mathbb M,
	\qquad (w,m,v)\longmapsto(w m v,m).
	\]
	It factors as
	\[
	(w,m,v)
	\longmapsto(w,m,m vm^{-1})
	\longmapsto(wv,m)
	\longmapsto(wv m,m).
	\]
	The first and last maps are isomorphisms. The middle map
	is a closed immersion: by Kostant--Rosenlicht, the orbit
	$\theta(\mathbb V) \mathbb V$ of $1$ under the action
	$(v,w)\cdot z=vzw^{-1}$ is closed, and its scheme-theoretic stabilizer is trivial since
	$\mathbb V\cap\theta(\mathbb V)=\{1\}$.
	Thus $j$ is a closed immersion.
	
	Hence the norms $\max(\|w\|_{\theta(V)},\|m\|_M,\|v\|_V)$ and $\max(\|wmv\|_G,\|m\|_M)$ are equivalent, hence there exists $c>0$ such that 
	\[\max(\|w\|_{\theta(V)},\|m\|_M,\|v\|_V)
	\preccurlyeq \max(\|wmv\|_G,\|m\|_M)^c.\]  From this we deduce that for any  compact $C'\subseteq G$, we have  \[\max(\|w\|_{\theta(V)},\|v\|_V)
	\preccurlyeq_{C'} \|m\|_M^c \] whenever $wmv\in C'$.

Now let $C$ be as in the statement, and put
\[
C'=\{\theta(g)g^{-1}:g\in C\},
\]
which is compact. For $g=vmh\in C$, we have
\[
\theta(g)g^{-1}
=\theta(v)\theta(m)m^{-1}v^{-1}\in C'.
\]
Applying the preceding estimate, and using that inversion on $\mathbb V$ is an isomorphism, we obtain
\[
\|v\|_V
\preccurlyeq_C\|\theta(m)m^{-1}\|_M^{c'}
\preccurlyeq_C\|m\|_M^{c''}
\]
for fixed $c',c''>0$.

Since $\mathbb H$ is closed in $\mathbb G$, the restriction
of $\|\cdot\|_G$ to $H$ is equivalent to $\|\cdot\|_H$.
Moreover, considering the morphism
$(v,m,g)\mapsto m^{-1}v^{-1}g$ gives
\[
\|h\|_H
\preccurlyeq
\max(\|v\|_V,\|m\|_M,\|g\|_G)^a
\preccurlyeq_C\|m\|_M^b
\]
for fixed $a,b>0$, as expected. 
 
\end{proof}

\subsection{Continuity of $\epsilon_{\lambda,v}$}\label{sec proof of cont}

The aim of this section is to prove \Cref{prop:continuity}. We follow \cite[end of Section 2]{BD} closely, but simplify some of their arguments, especially in the proof of \Cref{lm open cell estimates} below. We recall that $\nu\in \fa_{M}^*$ is called \textit{strictly $P$-dominant} if $\langle \nu, \alpha^\vee\rangle>0$ for all $\alpha\in R(A_M,P)$. We denote by $\Delta_P$ the set of simple roots defined as the indecomposable elements of $ R(A_M,P)$.

\begin{lemma}\label{lm open cell estimates}
There exists a $P$-dominant element $\nu_0\in(\fa_M^*)^{-\theta}$
such that, for every compact $C\subseteq G$,
\[
 \|m\theta(m)^{-1}\|_M
 \preccurlyeq_C e^{\langle\nu_0,H_M(m)\rangle}
\]
whenever $m$ occurs in a decomposition $g=vmh\in PH \cap C$, with $v\in V$, $m\in M$
and $h\in H$. Moreover, if $g_n=v_nm_nh_n$ converges to a point
outside $PH$, then
\[
 e^{-\langle\mu,H_M(m_n)\rangle}\longrightarrow0
\]
for every strictly $P$-dominant $\mu\in(\fa_M^*)^{-\theta}$.
\end{lemma}
\begin{proof}
Write $P^-=\theta(P)=MV^-$ and let
$\operatorname{pr}_{\mathfrak v^-}:\mathfrak g\to\mathfrak v^-$
be the projection along $\mathfrak p$. Following \cite[Section~2.2]{ZhiQi}, put
\[
 D_P(g)=\det\bigl(\operatorname{pr}_{\mathfrak v^-}
                  \circ\operatorname{Ad}(g)|_{\mathfrak v^-}\bigr).
\]
This determinant transforms by characters under $P\times P^-$.
At a Weyl representative it is nonzero exactly when the representative
preserves the roots of a maximal $F$-split torus of $M$ in the Lie algebra of $V^-$. Thus the Bruhat decomposition gives
$\{D_P\ne0\}=VMV^-$, and
\[
 D_P(vmv^-)=\det(\operatorname{Ad}(m)|_{\mathfrak v^-});
\]
see \cite[Lemma~2.1]{ZhiQi}. In particular, the coordinate functions
of the Levi projection $VMV^-\to M$ lie in
$\mathcal O_F(\mathbb G)[D_P^{-1}]$.

For $g=vmh$, put $y(g):=g\theta(g)^{-1}$. Its Levi component is
$m\theta(m)^{-1}$, so boundedness on $C$ of the regular numerators composed with $y$ gives
\[
 \|m\theta(m)^{-1}\|_M
 \preccurlyeq_C\max\{1,|D_P(y(g))|_F^{-N}\},
 \qquad g\in C\cap PH,
\]
with $N\geq1$ independent of $C$. Since $\theta(\rho_P)=-\rho_P$,
\begin{equation}\label{eq DP formula}
 |D_P(y(g))|_F=e^{-2\langle\rho_P,H_M(m)-\theta(H_M(m))\rangle}=e^{-4\langle\rho_P,H_M(m)\rangle}.
\end{equation}
The function $|D_P|_F$ is bounded above on $y(C)$, so the maximum
can indeed be replaced by a positive constant times $|D_P(y(g))|_F^{-N}$.
This proves the first assertion with $\nu_0=4N\rho_P$.

Suppose now that $g_n\to g\notin PH$. We claim that
$D_P(y(g_n))\to0$, since otherwise continuity would imply $D_P(y(g))\ne 0$. We assume this in the next few lines for the sake of contradiction. 
The orbit $P\cdot1$ is closed in $X\cap PP^-$ by the tangent-space
argument in the proof of \Cref{lm root separator}. Since
$y(g_n)\in P\cdot1$, this would give $y(g)\in P\cdot1$, hence
$g\in PG^\theta$. But $PH$ is open and closed in $PG^\theta$, as in
Fact~\ref{fact open}, so $g_n\in PH$ would imply $g\in PH$, a contradiction. Now put
\[
H_n^-:=\tfrac12\bigl(H_M(m_n)-\theta H_M(m_n)\bigr).
\] In view of \Cref{eq DP formula}, we deduce that 
\begin{equation}\label{eq bidule tend vers l infini} \langle\rho_P,H_n^-\rangle\to+\infty.\end{equation}

Then let $Q_k$ run through the maximal parabolic subgroups containing $P$,
with Levi components containing $M$, and put
\[\kappa_k\in \left(\frac{\fa_M}{\fa_G}\right)^*,\ \langle \kappa_k,H_M(m)\rangle=\log \delta_{Q_k}(m).\] Hence $\kappa_k$ is a positive multiple of the fundamental weight 
corresponding to the simple root defining $Q_k$ in $\Delta_P$. In particular $\rho_P$ is a positive linear
combination of the $\kappa_k$'s, and from this we deduce, using \Cref{eq bidule tend vers l infini}, that 
\begin{equation}\label{eq john tend ver l'infini}
\max_k \langle\kappa_k,H_n^-\rangle\to+\infty.
\end{equation}
Define the determinant function $D_{Q_k}$ using the opposite parabolic subgroup of $Q_k$ containing $P^-$. Its transformation law gives
\[
 |D_{Q_k}(y(g_n))|_F
 =e^{-\langle\kappa_k,H_M(m_n)-\theta H_M(m_n)\rangle}
 =e^{-2\langle\kappa_k,H_n^-\rangle}.
\]
The left-hand side is bounded, so each
$\langle\kappa_k,H_n^-\rangle$ is bounded below:
\begin{equation}\label{eq truc bounded below}
\exists C\in \R,\ \forall k,n \ \langle\kappa_k,H_n^-\rangle\geq C. 
\end{equation}
Finally every strictly $P$-dominant $\mu$ can be written
\[
 \mu=\sum_k b_k\kappa_k+\chi_0,\qquad b_k>0,
\]
where $\chi_0$ is the restriction to $\fa_M$ of a real linear
combination of $F$-rational algebraic characters of $\mathbb G$.
For each such algebraic character $\chi$,
\[
 2\langle\chi,H_n^-\rangle=\log|\chi(y(g_n))|_F.
\]
Since $y(g_n)$ lies in a compact subset of $G$, the right-hand side
is bounded; hence $\langle\chi_0,H_n^-\rangle$ is bounded as well.
The positivity of $b_k$ together with \Cref{eq john tend ver l'infini} and 
\Cref{eq truc bounded below} therefore imply $\langle\mu,H_n^-\rangle\to+\infty$.
For $\mu\in(\fa_M^*)^{-\theta}$, this is
$\langle\mu,H_M(m_n)\rangle\to+\infty$, as required.
\end{proof}

\begin{proof}[Proof of \Cref{prop:continuity}]
Since $P\times H\to PH$
is open, the function $\epsilon_{\lambda,v}$ is continuous, and actually locally constant, on $PH$. We must check its limit at the complement.

There exists $s>0$ such that every $g\in PH$ admits a decomposition
$g=v_0mh$ satisfying
\begin{equation}\label{eq controlled Levi representatives}
 \max(\|m\|_M,\|m^{-1}\|_M)
 \preccurlyeq\|m\theta(m)^{-1}\|_M^s.
\end{equation}
Indeed, invert a polar decomposition \eqref{eq analytic polar} for
$(M,H\cap M,\theta|_M)$ to choose $m=\omega a$, with $\omega$ in a
fixed compact set and $a$ in one of finitely many $\theta$-split tori.
Then $m\theta(m)^{-1}=\omega a^2\theta(\omega)^{-1}$, and the norm
comparison on split tori used in the proof of \Cref{prop:lagierestimate}
gives \eqref{eq controlled Levi representatives}.

Let $g_n \in PH$ tend to $g\notin PH$, take $C\subseteq G$ containing all $g_n$, and write $g_n=v_nm_nh_n$ as above, as in \Cref{lm open cell estimates}. By
\Cref{prop:controlHpart}, the growth hypothesis on $\zeta$,
and \eqref{eq controlled Levi representatives},
\[
 \begin{aligned}
 |\eta(h_n)|^{-1}\zeta(\sigma(m_n^{-1})v)
 &\preccurlyeq_{v,C}\|m_n\|_M^d\|m_n^{-1}\|_M^r\\
 &\preccurlyeq_{v,C}\|m_n\theta(m_n)^{-1}\|_M^t,
 \end{aligned}
\]
where $t:=s(d+r)$. Using \eqref{eq analytic epsilon vector} and the first assertion of
\Cref{lm open cell estimates}, we obtain
\begin{equation}\label{eq:epsilon-pre-boundary}
 \epsilon_{\lambda,v}(g_n)
 \preccurlyeq_{v,C}
 e^{-\langle\Re\lambda-\rho_P-t\nu_0,H_M(m_n)\rangle}.
\end{equation}
Choose $c>0$ such that
\[
 c>\langle\rho_P+t\nu_0,\alpha^\vee\rangle
 \qquad\text{for every }\alpha\in R(A_M,P).
\]
For $\lambda\in\mathcal D_{M,\C}^{-\theta}(c)$, the element
\[
 \mu:=\Re\lambda-\rho_P-t\nu_0\in(\fa_M^*)^{-\theta}
\]
is strictly $P$-dominant. The second assertion of
\Cref{lm open cell estimates} therefore makes the right-hand side of
\eqref{eq:epsilon-pre-boundary} tend to zero. Since $\epsilon_{\lambda,v}$ is zero outside $PH$, this
proves its continuity on $G$.
\end{proof}

\subsection{Convergence at maximal vertices}

Let $P=MV$ be a parabolic subgroup with a fixed Levi component $M$,
and let $x=x_u$ be $M$-admissible, with $u\in G$.
We keep the symmetric subgroup $H$ and its smooth character $\eta$
trivial on unipotent elements fixed. The datum is
\[
 \ell\in\Hom_{M_x}(\sigma,\delta_{P,x}\eta^u),\qquad
 \eta^u(m)=\eta(u^{-1}mu),\quad m\in M_x.
\]
We write 
$\sigma_\lambda:=\chi_\lambda\otimes\sigma$.
For $\lambda\in(\fa_{M,\C}^*)^{-\theta_x}$ and
$\phi\in I_P^G(\sigma_\lambda)$, put
\begin{equation}\label{eq analytic period H}
 j_P^G(u,\eta,\ell,\sigma_\lambda)(\phi)
 :=\int_{(H\cap u^{-1}Pu)\backslash H}
       \eta(h)^{-1}\ell\bigl(\phi(uh)\bigr)\,dh
\end{equation}
whenever the integral converges absolutely. These are the periods whose convergence we prove below, for $\lambda\in \mathcal D_{M,\C}^{-\theta_x}(c)$ and $c$ large enough.

We say that $(M,x)$ is maximal relative to $P$ if
\[
 \theta_x(\alpha)>0\ \Longrightarrow\ \theta_x(\alpha)=\alpha,
 \qquad\alpha\in\Delta_P.
\]
After conjugating $(P,M)$ into standard position, this agrees with
the notion of a maximal vertex by
\cite[Proposition~4.3 and Definition~4.6]{MOYjfa}.
For such a pair, put
$I=\{\alpha\in\Delta_P:\theta_x(\alpha)=\alpha\}$ and
\[
 \mathbb S=
 \left(\bigcap_{\alpha\in I}
 \ker(\alpha:\mathbb A_M\to\mathbb G_m)\right)^\circ,
 \qquad \mathbb L=Z_{\mathbb G}(\mathbb S).
\]
Let $Q=LU$ be the unique
parabolic subgroup containing $P$ with Levi component $L$.
Then $Q$ is $\theta_x$-split, $P\cap L$ is $\theta_x$-stable, and,
for $c>0$,
\begin{equation}\label{eq maximal vertex structure}
 \begin{gathered}
 P_x=(P\cap L)_x,\qquad
 (\fa_M^*)^{-\theta_x}=(\fa_L^*)^{-\theta_x},\\
 \mathcal D_{M,\C}^{-\theta_x}(c)
 \subseteq\mathcal D_{L,\C}^{-\theta_x}(c).
 \end{gathered}
\end{equation}
See \cite[Proposition~4.2]{MOYjfa} for the first two equalities.
The last inclusion follows from our normalization of relative coroots.

Put
\[
 H_0=H\cap u^{-1}Pu,\qquad
 H_1=H\cap u^{-1}Qu=H\cap u^{-1}Lu.
\]
Conjugation by $u$ identifies $H_0$ with $(P\cap L)_x$ and $H_1$
with $L_x$; in particular, $H_0\backslash H_1$ is compact.

\begin{lemma}\label{lemma:maxvertexestimate}
Let $\sigma$ be a smooth representation of finite length of $M$ and
$\ell\in\Hom_{M_x}(\sigma,\delta_{P,x}\eta^u)$. There exists $r\geq0$
such that, for every $f\in I_{P\cap L}^{L}(\sigma)$,
\[
 \int_{H_0\backslash H_1}
       |\eta(h_1)^{-1}\ell(f(uh_1u^{-1}g))|\,dh_1
 \preccurlyeq_f\|g\|_L^r,
 \qquad g\in L,
\]
where $\|\ \|_L$ is an algebraic norm on $L$.
\end{lemma}
\begin{proof}
Since $Q$ is $\theta_x$-split, $\delta_Q|_{L_x}=1$. The modulus identity
\eqref{eq modular transitivity} gives
\[
 \delta_{P,x}=\delta_{(P\cap L)_x}\delta_{P\cap L}^{-1/2}
 \quad\text{on }M_x.
\]
Thus the covariance of $f$ and $\ell$ makes the quotient integral
well defined, after transporting measures by $h_1\mapsto uh_1u^{-1}$.
By compactness of $H_0\backslash H_1$, it is bounded by a fixed
constant times
\[
 \sup_{h_1\in C_1}|\ell(f(uh_1u^{-1}g))|
\]
for a suitable fixed compact set $C_1\subset H_1$ mapping onto
$H_0\backslash H_1$;
the factor $|\eta(h_1)|^{-1}$ is bounded on $C_1$.

Choose a maximal compact subgroup $K_L$ giving the Iwasawa
decomposition $L=(P\cap L)K_L$. Write
\[
 uh_1u^{-1}g=vmk',\qquad
 v\in V\cap L,\quad m\in M,\quad k'\in K_L.
\]
Since $vm=uh_1u^{-1}g(k')^{-1}$, regularity of multiplication and
of the Levi projection $P\cap L\to M$ gives
\[
 \|m\|_M\preccurlyeq\|g\|_L^d
\]
for some $d\geq0$, uniformly in $h_1\in C_1$ and $k'\in K_L$.
Apply \eqref{eq global relative growth} to $(M,M_x,\theta_x)$ and
the character $\delta_{P,x}\eta^u$. Since $f|_{K_L}$ takes only
finitely many values, and $\delta_{P\cap L}^{1/2}$ has polynomial
growth on $M$, there exist $N,s\geq0$, independent of $f$, such that
\[
 \begin{aligned}
 |\ell(f(uh_1u^{-1}g))|
 &=\delta_{P\cap L}(m)^{1/2}
       |\ell(\sigma(m)f(k'))|\\
 &\preccurlyeq_f\|m\|_M^{N+s}
 \preccurlyeq_f\|g\|_L^{d(N+s)}.
 \end{aligned}
\]
Taking the supremum over $C_1$ proves the lemma.
\end{proof}

\begin{proposition}\label{prop:maximal-convergence}
Let $(M,x_u)$ be maximal relative to $P$, let $\sigma$ be a smooth
representation of finite length of $M$, and let
$\ell\in\Hom_{M_x}(\sigma,\delta_{P,x}\eta^u)$.
There exists $c>0$ such that the integral
\[
 j_P^G(u,\eta,\ell,\sigma_\lambda)(f_\lambda)
 =\int_{(H\cap u^{-1}Pu)\backslash H}
       \eta(h)^{-1}\ell(f_\lambda(uh))\,dh
\]
converges absolutely for every $f\in I_P^G(\sigma)$ whenever
$\lambda\in\mathcal D_{M,\C}^{-\theta_x}(c)$.
\end{proposition}
\begin{proof}
Put $\tau=I_{P\cap L}^{L}(\sigma)$ and, for $\phi\in V_\tau$, set
\[
 \zeta(\phi):=\int_{H_0\backslash H_1}
           |\eta(h_1)^{-1}\ell(\phi(uh_1u^{-1}))|\,dh_1.
\]
The function $\zeta$ is absolutely homogeneous and satisfies
\[
 \zeta\bigl(\tau(uh_1u^{-1})\phi\bigr)
       =|\eta(h_1)|\zeta(\phi),\qquad h_1\in H_1.
\]
\Cref{lemma:maxvertexestimate} gives
\begin{equation}\label{eq one more bound}
 \zeta(\tau(g)\phi)\preccurlyeq_\phi\|g\|_L^r,
 \qquad g\in L,
\end{equation}
with $r$ independent of $\phi$.

In view of \eqref{eq maximal vertex structure}, transitivity of normalized induction gives
\[
 \begin{aligned}
 \Gamma_\lambda &: I_P^G(\sigma_\lambda)
                  \simeq I_Q^G(\chi_\lambda\otimes\tau),\\
 (\Gamma_\lambda\phi)(g)(l)
   &=\delta_Q(l)^{-1/2}\chi_\lambda(l)^{-1}\phi(lg).
 \end{aligned}
\]
Writing $\Gamma=\Gamma_0$, we have
$\Gamma_\lambda f_\lambda=(\Gamma f)_\lambda$, as in
\eqref{eq Gamma transitivity}.
Integration in stages along $H_0\subseteq H_1\subseteq H$, with
compatible quotient measures, therefore gives
\begin{equation}\label{eq maximal absolute stages}
 \begin{aligned}
 &\int_{H_0\backslash H}|\eta(h)^{-1}\ell(f_\lambda(uh))|\,dh\\
 &\qquad=\int_{H_1\backslash H}|\eta(h)|^{-1}
                      \zeta\bigl((\Gamma f)_\lambda(uh)\bigr)\,dh.
 \end{aligned}
\end{equation}
Here $\delta_Q(uh_1u^{-1})=\chi_\lambda(uh_1u^{-1})=1$ for $h_1\in H_1$;
the identity is valid
for nonnegative integrals, before finiteness is known. This is the
integration-in-stages calculation of
\cite[Proposition~3.7]{LuMcomp}; see also
\cite[(5.8)--(5.9)]{MOYjfa}.

Apply \Cref{thm:lipopencase}(ii) to the $\theta$-split parabolic
$u^{-1}Qu$, with Levi $u^{-1}Lu$, representation
$l'\mapsto\tau(ul'u^{-1})$, character $\eta$ on $H$, and section
$g\mapsto(\Gamma f)_\lambda(ug)$.
\Cref{eq one more bound} supplies the growth hypothesis of that
theorem, since conjugation by $u$ preserves polynomial bounds.
Thus the right-hand side of \eqref{eq maximal absolute stages}
is finite for $\lambda\in\mathcal D_{L,\C}^{-\theta_x}(c)$ with $c$
sufficiently large, hence for
$\lambda\in\mathcal D_{M,\C}^{-\theta_x}(c)$ by
\eqref{eq maximal vertex structure}.
\end{proof}

\subsection{Convergence in general}

We deduce convergence for general intertwining periods from convergence
at maximal vertices, following \cite[Section~5.1]{MOYjfa}.

We retain the maximal $F$-split torus $A_0$ and the minimal parabolic
subgroup $P_0$ fixed in \Cref{ssec intertwining}. Put
\[
 W=N_G(A_0)/Z_G(A_0),\qquad
 W_M=N_M(A_0)/Z_M(A_0)
\]
for each standard Levi subgroup $M$, viewing $W_M$ as a subgroup of $W$.
The length function on $W$ is defined by the positive roots determined
by $P_0$.

Let $P=MV$ and $P_1=M_1V_1$ be standard parabolic subgroups.
Choose $w\in W$ carrying $M$ onto $M_1$ and of minimal length in
$wW_M$, choose a lift $\dot w\in N_G(A_0)$, and let $n\in\dot wM$.
Let $\sigma$ be a smooth representation of finite length of $M$, and put
\[
 ({}^n\sigma)(m_1)=\sigma(n^{-1}m_1n),\qquad m_1\in M_1.
\]
For $\lambda\in\fa_{M,\C}^*$, the transported parameter
$n\lambda=w\lambda\in\fa_{M_1,\C}^*$ is characterized by
\[
 \langle w\lambda,H_{M_1}(m_1)\rangle
 =\langle\lambda,H_M(n^{-1}m_1n)\rangle,
 \qquad m_1\in M_1.
\]
In particular,
${}^n(\sigma_\lambda)=\chi_{w\lambda}\otimes{}^n\sigma$.

At regular parameters, the specialization of the operator
$M(n,\sigma)$ of \Cref{ssec intertwining} at $\chi_\lambda$ is a
$G$-intertwining operator
\[
 M(n,\sigma,\lambda):
 I_P^G(\sigma_\lambda)
 \longrightarrow
 I_{P_1}^G(\chi_{w\lambda}\otimes{}^n\sigma).
\]
In its domain of absolute convergence, it is given by
\begin{equation}\label{eq:analytic-intertwiner}
 (M(n,\sigma,\lambda)\phi)(g)
 =\int_{(V_1\cap nVn^{-1})\backslash V_1}
                       \phi(n^{-1}vg)\,dv,
 \qquad \phi\in I_P^G(\sigma_\lambda).
\end{equation}
For $c>0$, put
\[
 \begin{aligned}
 \mathcal D_{M,\C}^w(c):=\bigl\{\nu\in\fa_{M,\C}^*:\;&
 \Re\langle\nu,\beta^\vee\rangle>c\\[-2pt]
 &\text{for all }\beta\in R(A_M,P)
       \text{with }w\beta\in-R(A_{M_1},P_1)\bigr\}.
 \end{aligned}
\]
For $c$ sufficiently large, \eqref{eq:analytic-intertwiner} converges
absolutely whenever $\lambda\in\mathcal D_{M,\C}^w(c)$.
After increasing $c$, the operator is invertible throughout this region;
see \cite[Section~2.6 and Proposition~2.3]{MOYjfa}.

We now specialize to an edge
$(M,x)\stackrel n\searrow(M_1,x_1)$ associated to
$\alpha\in\Delta_P$, as in \Cref{ssec FE}; thus $w=s_\alpha$ and
$n\in\dot s_\alpha M$.
Here $s_\alpha=w_0^{L_\alpha}w_0^M$, where $L_\alpha$ is the
standard Levi defined in \Cref{ssec FE}, and $w_0^{L_\alpha}$ and
$w_0^M$ are the longest elements of $W_{L_\alpha}$ and $W_M$,
respectively.
Choose $u$ with $x=x_u$ and put $u_1=nu$, so $x_1=x_{u_1}$.
By \Cref{lm edge data}, the same linear form
$\ell\in\Hom_{M_x}(\sigma,\delta_{P,x}\eta^u)$ belongs to
\[
 \Hom_{(M_1)_{x_1}}
       ({}^n\sigma,\delta_{P_1,x_1}\eta^{u_1}).
\]
Conjugation by $n$ transports the anti-invariant parameter spaces:
\[
 n(\fa_{M,\C}^*)^{-\theta_x}
 =(\fa_{M_1,\C}^*)^{-\theta_{x_1}}.
\]
For $\lambda\in(\fa_{M,\C}^*)^{-\theta_x}$, the characters
$\chi_\lambda$ and $\chi_{n\lambda}$ are therefore trivial on $M_x$
and $(M_1)_{x_1}$, respectively.

\begin{proposition}\label{prop:analytic fe}
Let $(M,x)\stackrel n\searrow(M_1,x_1)$ be an edge associated to
$\alpha\in\Delta_P$, and put $u_1=nu$ as above.
Let $\sigma$ be a smooth representation of finite length of $M$, and let
$\ell\in\Hom_{M_x}(\sigma,\delta_{P,x}\eta^u)$.
Choose $c>0$ so that \eqref{eq:analytic-intertwiner} converges absolutely
for $\lambda\in\mathcal D_{M,\C}^{s_\alpha}(c)$.
For $\lambda\in(\fa_{M,\C}^*)^{-\theta_x}\cap\mathcal D_{M,\C}^{s_\alpha}(c)$, absolute convergence of the
left-hand side period below implies absolute convergence of the right-hand side 
period, and, with compatible quotient measures,
\begin{equation}\label{eq:analytic-edge-H}
 \begin{aligned}
 &j_P^G(u,\eta,\ell,\sigma_\lambda)(\phi)\\
 &\qquad=j_{P_1}^G(u_1,\eta,\ell,
             \chi_{n\lambda}\otimes{}^n\sigma)
       \bigl(M(n,\sigma,\lambda)\phi\bigr).
 \end{aligned}
\end{equation}
\end{proposition}
\begin{proof}
The proof is essentially the same as \cite[Proposition~5.1]{MOYjfa}, and as there relies on \cite[Lemma~2.1]{MOYjfa}. The extension to our more general context is straightforward, and we refer to \Cref{prop FE} for the details of this extension.
\end{proof}

\begin{theorem}\label{thm:convergence in general}
Let $P=MV$ be standard, let $(M,x)$ be a vertex with $x=x_u$, and let $\sigma$ be a
smooth representation of finite length of $M$. Let $\eta$ be a smooth
character of $H$, trivial on unipotent elements, and let
\[
 \ell\in\Hom_{M_x}(\sigma,\delta_{P,x}\eta^u).
\]
There exists a nonempty open convex set
$C\subseteq(\fa_{M,\C}^*)^{-\theta_x}$ such that
\[
 j_P^G(u,\eta,\ell,\sigma_\lambda)(f_\lambda)
 =\int_{(H\cap u^{-1}Pu)\backslash H}
       \eta(h)^{-1}\ell(f_\lambda(uh))\,dh
\]
converges absolutely for every $f\in I_P^G(\sigma)$ whenever $\lambda\in C$.
\end{theorem}
\begin{proof}
In view of \Cref{prop:maximal-convergence} and \Cref{prop:analytic fe}, the proof of \cite[Theorem 5.3]{MOYjfa} applies verbatim. 
\end{proof}

\begin{corollary}\label{cor final}
With the notation of \Cref{thm:convergence in general},
\begin{equation}\label{eq meromorphic cor}
 \lambda\longmapsto
 j_P^G(u,\eta,\ell,\sigma_\lambda)(f_\lambda)
\end{equation}
is holomorphic on $C$ and extends
meromorphically to $(\fa_{M,\C}^*)^{-\theta_x}$. Its continuation is
the specialization at $\chi_\lambda$ of
$J_P^G(u,\eta,\ell,\sigma)$ on the corresponding flat section.
\end{corollary}
\begin{proof}

Holomorphy on $C$ follows from \Cref{thm:convergence in general}.
Put $J:=J_P^G(u,\eta,\ell,\sigma)$ and
\[
d:=\prod_{\cO_j\subseteq\partial\cO_x}p_{x,j}
\in A_x\smallsetminus\{0\}.
\]
By \eqref{eq direct denominator}, $J$ sends $I_{P,x}(\sigma)$ to $A_x[d^{-1}]$. Hence, whenever $d(\chi_\lambda)\neq0$,
evaluation at $\chi_\lambda\in \cX_x$ defines a linear form
\[
J_\lambda\in\Hom_H(I_P^G(\sigma_\lambda),\eta).
\]
Fix such a $\lambda\in C$. Let $\C_\lambda$ denote $\C$,
with trivial $H$-action and with $A_x$-action given by
	$a\cdot z=a(\chi_\lambda)z$.
	Use the filtration of \Cref{lm adapted filtration}, with
	$\cV_x^0=\cV_{a+1}$ and $\cV_N=I_{P,x}(\sigma)$.
	For every $a+2\leq j\leq N$, the quotient
	$\cV_j/\cV_{j-1}$ corresponds to an orbit
	$\cO_j\subseteq\partial\cO_x$.
	Applying \Cref{cor kill} with $W=\C_\lambda$, and using
	$p_{x,j}(\chi_\lambda)\neq0$, gives
	\[
	\Hom_{\mathcal H_{A_x}[H]}
	\bigl(\cV_j/\cV_{j-1},\C_\lambda\otimes\eta\bigr)=0.
	\]
	It follows that the restriction map
	\[\Hom_{\mathcal H_{A_x}[H]}
		\bigl(I_{P,x}(\sigma),\C_\lambda\otimes\eta\bigr) \longrightarrow
		\Hom_{\mathcal H_{A_x}[H]}
		\bigl(\cV_x^0,\C_\lambda\otimes\eta\bigr)
	\]
	is injective.
	
	Precompose $J_\lambda$ and
	$j_P^G(u,\eta,\ell,\sigma_\lambda)$ with the natural surjection
	\[
	I_{P,x}(\sigma)\longrightarrow I_P^G(\sigma_\lambda).
	\]
	The resulting $A_x$-linear, $H$-equivariant forms agree on
	$\cV_x^0$ by the support formula
	\eqref{eq support characterization}.
	The previously proved injectivity therefore shows that
	\begin{equation}\label{eq final}
		J_\lambda=j_P^G(u,\eta,\ell,\sigma_\lambda), \ \lambda\in C, \ d(\chi_\lambda)\neq 0.
	\end{equation}  
	
	Multiply \eqref{eq final} by $d(\chi_\lambda)$ and apply it to a flat section $\phi_\lambda$.
	Both sides are holomorphic on $C$, the algebraic side being the
	pullback of a regular function on $\cX_x$ by
	\eqref{eq direct denominator}.
	Since $\{\lambda\in C:d(\chi_\lambda)\neq0\}$ is dense in $C$,
	the equality extends to all of $C$.
	Thus the algebraic period, pulled back along
	$\lambda\mapsto\chi_\lambda$, gives the claimed meromorphic
	continuation.
\end{proof}

\bibliographystyle{amsalpha}
\bibliography{Algebraic_LIP}

\end{document}